\batchmode
\makeatletter
\def\input@path{{/home/andrew/}}
\makeatother
\documentclass[oneside,british,english]{amsart}
\usepackage[T1]{fontenc}
\usepackage[latin9]{inputenc}
\usepackage{geometry}
\usepackage{verbatim}
\usepackage{amstext}
\usepackage{amsthm}
\usepackage{amssymb}
\usepackage[all]{xy}

\makeatletter
\numberwithin{equation}{section}
\numberwithin{figure}{section}
\theoremstyle{plain}
\newtheorem{thm}{\protect\theoremname}
\theoremstyle{definition}
\newtheorem{example}[thm]{\protect\examplename}
\theoremstyle{remark}
\newtheorem*{rem*}{\protect\remarkname}
\theoremstyle{plain}
\newtheorem{fact}[thm]{\protect\factname}
\theoremstyle{definition}
\newtheorem{defn}[thm]{\protect\definitionname}
\theoremstyle{plain}
\newtheorem{prop}[thm]{\protect\propositionname}
\theoremstyle{plain}
\newtheorem{lem}[thm]{\protect\lemmaname}
\theoremstyle{plain}
\newtheorem{cor}[thm]{\protect\corollaryname}

\usepackage{hyperref}
\usepackage{amsmath}

\@ifundefined{showcaptionsetup}{}{%
 \PassOptionsToPackage{caption=false}{subfig}}
\usepackage{subfig}
\makeatother

\usepackage{babel}
\addto\captionsbritish{\renewcommand{\corollaryname}{Corollary}}
\addto\captionsbritish{\renewcommand{\definitionname}{Definition}}
\addto\captionsbritish{\renewcommand{\examplename}{Example}}
\addto\captionsbritish{\renewcommand{\factname}{Fact}}
\addto\captionsbritish{\renewcommand{\lemmaname}{Lemma}}
\addto\captionsbritish{\renewcommand{\propositionname}{Proposition}}
\addto\captionsbritish{\renewcommand{\remarkname}{Remark}}
\addto\captionsbritish{\renewcommand{\theoremname}{Theorem}}
\addto\captionsenglish{\renewcommand{\corollaryname}{Corollary}}
\addto\captionsenglish{\renewcommand{\definitionname}{Definition}}
\addto\captionsenglish{\renewcommand{\examplename}{Example}}
\addto\captionsenglish{\renewcommand{\factname}{Fact}}
\addto\captionsenglish{\renewcommand{\lemmaname}{Lemma}}
\addto\captionsenglish{\renewcommand{\propositionname}{Proposition}}
\addto\captionsenglish{\renewcommand{\remarkname}{Remark}}
\addto\captionsenglish{\renewcommand{\theoremname}{Theorem}}
\providecommand{\corollaryname}{Corollary}
\providecommand{\definitionname}{Definition}
\providecommand{\examplename}{Example}
\providecommand{\factname}{Fact}
\providecommand{\lemmaname}{Lemma}
\providecommand{\propositionname}{Proposition}
\providecommand{\remarkname}{Remark}
\providecommand{\theoremname}{Theorem}

\begin{document}
\title[Ultralimits of Wasserstein spaces and $CD(K,\infty)$ spaces]{Ultralimits of Wasserstein spaces and metric measure spaces with Ricci
curvature bounded from below}
\author{Andrew Warren}
\address{\selectlanguage{british}%
Carnegie Mellon University\\ 5000 Forbes Avenue\\ Pittsburgh, PA
15213\\ USA; and,}
\address{\selectlanguage{british}%
Institut des Hautes Études Scientifiques\\35 Route de Chartres\\91440
Bures-sur-Yvette\\France}
\email{\selectlanguage{british}%
warren@ihes.fr}
\selectlanguage{english}%
\begin{abstract}
We investigate the stability of the Wasserstein distance, a metric
structure on the space of probability measures arising from the theory
of optimal transport, under metric ultralimits. We first show that
if $(X_{i},d_{i})_{i\in\mathbb{N}}$ is a sequence of metric spaces
with metric ultralimit $(\hat{X},\hat{d})$, then the $p$-Wasserstein
space $(\mathcal{P}_{p}(\hat{X}),W_{p})$ embeds isometrically in
a canonical fashion into the metric ultralimit of the sequence of
$p$-Wasserstein spaces $(\mathcal{P}_{p}(X_{i}),W_{p})$. Second,
using a notion of ultralimit of metric measure spaces modeled on the
one introduced by Elek, we use the machinery of ultralimits of Wasserstein
spaces to prove that an ultralimit of $CD(K,\infty)$ spaces is a
$CD(K,\infty)$ space. This provides a new proof that the $CD(K,\infty)$
property is stable under pointed measured Gromov convergence. Along
the way, we establish some basic results on how the Loeb measure construction
interacts with Wasserstein distances as well as integral functionals,
which may be of independent interest.
\end{abstract}

\maketitle

The plan of the paper is as follows. In Section \ref{sec:Introduction},
we review the necessary background from optimal transport, as well
as the theory of metric ultralimits and related notions such as Loeb
measures. We also present the notion of ultralimit of metric measure
spaces we will use in Section \ref{sec:Ultralimits-of-CD(K,infty)}.
In Section \ref{sec:Ultralimits-of-Wasserstein}, we show in Theorem
\ref{thm:isometric embedding} that the $p$-Wasserstein space $(\mathcal{P}_{p}(\hat{X}),W_{p})$
embeds isometrically in a canonical fashion into the metric ultralimit
of the sequence of $p$-Wasserstein spaces $(\mathcal{P}_{p}(X_{i}),W_{p})$,
and explore some basic features of said metric ultralimit. We also
study the interaction between the Loeb measure construction and the
ultralimit of $p$-Wasserstein spaces construction (see especially
Theorem \ref{thm:radon representation}, Corollary \ref{cor:monad Loeb pushdown consistency},
and Lemma \ref{lem:bounded diameter pushfwd lifting}). Finally, in
Section \ref{sec:Ultralimits-of-CD(K,infty)} we apply the machinery
developed in Section \ref{sec:Ultralimits-of-Wasserstein}, and give
a proof that a metric measure ultralimit of $CD(K,\infty)$ spaces
is again a $CD(K,\infty)$ space (Theorem \ref{thm:main}).

\section{Introduction\label{sec:Introduction}}

\subsection{Facts from optimal transport}

Let $(X,d)$ be a complete, separable metric space. Given $p\in[1,\infty)$,
the $p$-Wasserstein metric on the space $\mathcal{P}_{p}(X)$ of
Borel probability measures with finite $p$th moments is defined as
follows:
\[
W_{p}(\mu,\nu):=\inf_{\gamma\in\Pi(\mu,\nu)}\left(\int_{X\times X}d(x,y)^{p}d\gamma(x,y)\right)^{1/p}.
\]
Here $\Pi(\mu,\nu)$ denotes the set of all couplings of the measures
$\mu$ and $\nu$. 

The $p$-Wasserstein metrics possess numerous intriguing geometric
features. For instance:
\begin{itemize}
\item When $(X,d)$ is compact, $W_{p}$ metrizes the weak convergence of
probability measures; when $(X,d)$ is not necessarily compact, convergence
in $W_{p}$ is equivalent to weak convergence together with convergence
of $p$th moments. %
{} In particular, $(\mathcal{P}_{p}(X),W_{p})$ enjoys a version of
the Glivenko-Cantelli theorem: if $\mu_{n}:=\frac{1}{n}\sum_{i=1}^{n}\delta_{x_{i}}$
is an \emph{empirical measure for $\mu$}, viz., the points $x_{i}$
are i.i.d. samples from $\mu$, then $W_{p}(\mu_{n},\mu)\rightarrow0$
with probability 1. In particular, discrete measures are dense in
$W_{p}$.
\item Let $(X,d)$ be a \emph{geodesic metric space}, that is, given any
$y,z\in X$ there exists a continuous function $f:[0,1]\rightarrow(X,d)$
with $f(0)=y$, $f(1)=z$, and $d(f(t),f(s))=|t-s|d(y,z)$ for all
$t,s\in[0,1]$ (and in this case we say $f(t)$ is a constant speed
$d$-geodesic connecting $y$ and $z$). Then, $(\mathcal{P}_{p}(X),W_{p})$
is also a geodesic metric space \cite[Theorem 2.10]{ambrosio2013user}.
\item In the case $p=2$, $(\mathcal{P}_{2}(X),W_{2})$ comes equipped with
a \emph{formal Riemannian metric structure} \cite{otto2000generalization,otto2001geometry},
so that numerous objects and computations from (finite-dimensional)
Riemannian geometry have analogues in $(\mathcal{P}_{2}(X),W_{2})$. 
\item In particular, there is a well-developed theory of \emph{gradient
flows} in $(\mathcal{P}_{2}(X),W_{2})$ \cite{ambrosio2008gradient},
whereby numerous parabolic PDEs (such as the Fokker-Planck equation
\cite{jordan1998variational}, the porous medium equation \cite{otto2001geometry},
etc.) can be recast as infinite-dimensional ODEs of the form $\dot{\mu}_{t}=-\nabla F(\mu_{t})$,
where $F$ is some functional on $\mathcal{P}_{2}(X)$. %
\item In the case where $(X,d)$ is itself a connected Riemannian manifold,
there is a deep connection between geometric properties of the manifold
$(X,d)$ and those of the ``manifold'' $(\mathcal{P}_{2}(X),W_{2})$.
For instance, it is known \cite[Theorem 1]{von2005transport} that
$(X,d)$ has a uniform Ricci curvature lower bound of $K\in\mathbb{R}$
iff the \emph{relative entropy }functional $H(\mu\mid\text{vol})=\int_{X}\frac{d\mu}{d\text{vol}}\log\frac{d\mu}{d\text{vol}}d\text{vol}$
(where vol denotes the Riemannian volume measure on $X$) is $K$-geodesically
convex on the space $(\mathcal{P}_{2}(X),W_{2})$, that is, for any
curve $\mu_{t}:[0,1]\rightarrow\mathcal{P}_{2}(X)$ which is a constant
speed $W_{2}$-geodesic connecting $\mu_{0}$ and $\mu_{1}$, 
\[
H(\mu_{t}\mid\text{vol})\leq(1-t)H(\mu_{0}\mid\text{vol})+tH(\mu_{1}\mid\text{vol})-K\frac{t(1-t)}{2}W_{2}^{2}(\mu_{0},\mu_{1}).
\]
\item Moreover, it is sometimes possible to turn around and use the consequences
of geometric properties of a manifold $(X,d)$, in the space $(\mathcal{P}_{2}(X),W_{2})$,
as \emph{surrogates} for geometric properties on $(X,d)$; doing so
offers an avenue to \emph{synthetically} extend geometric notions
from manifolds to more general nonsmooth spaces, as long as the same
geometric reasoning goes through in the space $(\mathcal{P}_{2}(X),W_{2})$.
\end{itemize}
Towards this last point, such a strategy has been pursued successfully
in the setting of the Lott-Sturm-Villani theory of synthetic Ricci
curvature \cite{lott2009ricci,sturm2006geometry1,sturm2006geometry2}
(see also the more recent review article by Ambrosio \cite{ambrosio2018calculus}),
which we now very briefly discuss. %

Let $(X,d)$ be a complete separable metric measure space, and let
$\lambda$ be a $\sigma$-finite Borel measure on $(X,d)$ which is
finite on bounded sets. We say that $(X,d,\lambda)$ is a ``strong
$CD(K,\infty)$'' space provided that the relative entropy functional
$H(\cdot\mid\lambda):\mathcal{P}_{2}(X)\rightarrow\mathbb{R}\cup\{\infty\}$
is $K$-geodesically convex on the space $(\mathcal{P}_{2}(X),W_{2})$.
Likewise, we say that $(X,d,\lambda)$ is a ``$CD(K,\infty)$''
space provided that for every $\mu_{0}$ and $\mu_{1}$ in $\mathcal{P}_{2}(X)$,
there exists a constant speed $W_{2}$-geodesic $(\mu_{t})_{t\in[0,1]}$
connecting $\mu_{0}$ and $\mu_{1}$, such that for all $t\in[0,1]$,
\[
H(\mu_{t}\mid\lambda)\leq(1-t)H(\mu_{0}\mid\lambda)+tH(\mu_{1}\mid\lambda)-K\frac{t(1-t)}{2}W_{2}^{2}(\mu_{0},\mu_{1}).
\]
Evidently, every strong $CD(K,\infty)$ space is also $CD(K,\infty)$,
and every $CD(K,\infty)$ space for which there exists only one constant
speed $W_{2}$-geodesic connecting any two measures in $\mathcal{P}_{2}(X)$
which are absolutely continuous with respect to the reference measure
$\lambda$, is (trivially) a strong $CD(K,\infty)$. Moreover, it
was proved in \cite{rajala2014non} that every strong $CD(K,\infty)$
actually has this latter ``unique constant speed $W_{2}$-geodesic
between measures in $\mathcal{P}_{2,ac(\lambda)}(X)$'' property.
The $CD(K,\infty)$ property has numerous consequences, many of which
are surveyed in \cite{villani2008optimal}, but we give one important
example. In the case where $K>0$, a metric measure space which is
$CD(K,\infty)$ satisfies the following \emph{$K$-log-Sobolev inequality}:
\[
\forall\mu\in\mathcal{P}_{2}(X)\quad H(\mu\mid\lambda)\leq\frac{1}{2K}\int_{\frac{d\mu}{d\lambda}>0}\frac{\left|\nabla^{-}\left(\frac{d\mu}{d\lambda}\right)\right|^{2}}{\frac{d\mu}{d\lambda}}d\lambda
\]
where $|\nabla^{-}f|(x):=\limsup_{y\rightarrow x}\frac{\max\{f(x)-f(y),0\}}{d(x,y)}$.
(See \cite[Theorem 30.22]{villani2008optimal} for a proof.) A central
question in the theory of synthetic Ricci curvature is that of \emph{stability}:
that is, given a ``reasonable'' notion of convergence of metric
measure spaces (and note that even devising a notion of convergence
of metric measure spaces is a nontrivial issue), one would like to
know whether properties like $CD(K,\infty)$ are preserved along convergent
sequences.

In this article, we require a small relaxation of the ``basic assumptions''
underlying the standard setup of the Wasserstein distance. Indeed,
at the beginning of this subsection, we defined the $p$-Wasserstein
distances in the setting where the underlying metric space is Polish
(i.e. complete and separable). In what follows, it will be beneficial
to also be able to define the $p$-Wasserstein distance atop a metric
space which is merely complete, but not necessarily separable. That
separability can be dispensed with is non-obvious, since pedagogical
accounts of the $p$-Wasserstein distances (such as those in \cite{ambrosio2013user,santambrogio2015optimal,villani2003topics})
present many arguments which seem to invoke separability in an essential
way. 

However, it was observed in \cite{plotkin2018free} that there is
a relatively straightforward way to adapt the $p$-Wasserstein distance
to arbitrary complete metric spaces. Instead of taking the space $\mathcal{P}_{p}(X)$
to be the space of all \emph{Borel} probability measures on $X$ with
finite $p$th moment, one instead declares $\mathcal{P}_{p}(X)$ to
be the space of \emph{Radon} probability measures on $X$ with finite
$p$th moment. (In the case where $X$ is separable, every Borel probability
measure is also Radon \cite[Theorem 7.1.7]{bogachev2007measure2},
so we are justified in using the same notation.\footnote{What's more, even when $X$ is \emph{not} separable, the ``restriction''
to Radon probability measures is in some sense a restriction in the
mildest way possible, since the existence of a complete metric space
with a Borel probability measure that is not Radon requires the existence
of a real-valued measurable cardinal (this follows from \cite[Theorem 438H]{fremlin2003measure}),
and is therefore independent of ZFC. We thank Taras Banakh for bringing
this set-theoretic issue to our attention.}) Crucially, Radon probability measures automatically have separable
support, and that this is enough to recover a number of basic facts
about $(\mathcal{P}_{p}(X),W_{p})$ familiar from the separable setting.
In particular, it is shown in \cite{plotkin2018free} that \emph{discrete}
probability measures are $W_{p}$-dense even when $X$ is merely a
complete metric space, which is a fact we will invoke repeatedly.
Note, however, that the fact that the general theory around $p$-Wasserstein
distances also ``works'' when considering Radon measures on complete
metric spaces had been alluded to earlier in the optimal transport
literature, for example \cite[Remark 2.8]{lisini2007characterization}.

\subsection{Ultraproduct preliminaries}

We briefly outline the theory of metric ultralimits. We do so in the
language of nonstandard analysis\footnote{Here we use the term ``nonstandard analysis'' essentially to refer
to a collection of specialized terminology and theorems surrounding
objects which are ``constructed'' using ultrafilters, ultraproducts,
et cetera. This is in contrast to, for instance, nonstandard analysis
in the style of Nelson's \emph{Internal Set Theory} \cite{nelson1977internal}
and related approaches \cite{kanovei2013nonstandard}, which conservatively
extend ZFC itself.}; it is possible to define metric ultralimits without explicit recourse
to e.g. the hyperreals, but in the opinion of the author, doing so
comes at a cost of conceptual clarity and technical facility. The
reader may consult \cite{goldblatt2012lectures} for general background.
As well, \cite[Chapter 10]{dructu2018geometric} gives a concise summary
of nonstandard analysis notions tailored for metric geometry/geometric
group theory applications, and --- aside from its omission of Loeb
measures, which we address below --- covers essentially all of the
nonstandard analysis we will use in this article. We also adopt some
of their notation.%

Let $\omega$ denote a non-principal ultrafilter on $\mathbb{N}$.
The hyperreals $\mathbb{R}^{\omega}$ are the non-archimedian field
defined by taking an ultrapower of $\mathbb{R}$: namely, given two
sequences of reals $(r_{i})$ and $(r_{i}^{\prime})$, we identify
$(r_{i})$ and $(r_{i}^{\prime})$ under the equivalence $\sim_{\omega}$
if and only if $\{i\in\mathbb{N}:r_{i}=r_{i}^{\prime}\}\in\omega$;
then, $\mathbb{R}^{\omega}:=\mathbb{R}^{\mathbb{N}}/\sim_{\omega}$.
We write $[(r_{i})]\in\mathbb{R}^{\omega}$ for such an equivalence
class, (or sometimes $[(r_{i})_{i\in\mathbb{N}}]$ if we wish to emphasize
that it is the index $i$ which is addressed by the ultrafilter $\omega$).
Formally this is analogous to the Cauchy sequence construction of
$\mathbb{R}$ from $\mathbb{Q}$, except here we are using a much
finer equivalence relation. By making the injection 
\[
\mathbb{R}\ni r\mapsto[(r)]\in\mathbb{R}^{\omega}
\]
(where $(r)$ is the constant sequence with each term as $r$), we
see that $\mathbb{R}$ embeds canonically into $\mathbb{R}^{\omega}$;
we therefore say, abusively, that $\mathbb{R}\subset\mathbb{R}^{\omega}$.
$\mathbb{R}^{\omega}$ inherits the algebraic structure of $\mathbb{R}$
via the quotient with respect to $\omega$; for instance, we overload
$+$ by making the definition $[(r_{i})]+[(r_{i}^{\prime})]:=[(r_{i}+r_{i}^{\prime})]$.
(Of course one must check this is well-defined.) We use the $\approx$
symbol (read as ``approximately equal'', or ``infinitesimally close'')
as follows: 
\begin{align*}
[(r_{i})]\approx[(r_{i}^{\prime})] & \iff(\forall\varepsilon\in\mathbb{R}_{+})(|[(r_{i})]-[(r_{i}^{\prime})]|<\varepsilon)\\
 & \iff(\forall\varepsilon\in\mathbb{R}_{+})\{i\in\mathbb{N}\mid|r_{i}-r_{i}^{\prime}|<\varepsilon\}\in\omega.
\end{align*}
For example, $[(10^{-i})]\approx[(0)]$ but $[(10^{-i})]\neq[(0)]$.
If there exists an $r\in\mathbb{R}$ such that $[(r_{i})]\approx r$,
we say that $[(r_{i})]$ is \emph{nearstandard}. Such an $r$ is necessarily
unique. On the subset of nearstandard elements of $\mathbb{R}^{\omega}$,
we define the \emph{standard part mapping $st(\cdot)$ }by $st([(r_{i})])=r$. 

A turn of phrase we will use frequently in the sequel, to indicate
that some property $p$ holds on a set of indices $i\in\mathbb{N}$
that belongs to the ultrafilter $\omega$, is that $p$ holds $\omega$-a.s.;
the terminology is justified by the fact that $\omega$ is a finitely
additive probability measure on $\mathbb{N}$.

If $(X_{i})$ is a sequence of algebraic structures of the same type,
equipped with a notion of equality, one can consider the ultraproduct
$X^{\omega}$ defined by $X^{\omega}/\sim_{\omega}$, where now we
say $(x_{i})\sim_{\omega}(x_{i}^{\prime})$ iff $\{i\in\mathbb{N}:x_{i}=x_{i}^{\prime}\}\in\omega$
(a.k.a. iff $x_{i}=x_{i}^{\prime}$ $\omega$-a.s.). Of particular
interest for us is the case where each $X_{i}$ is a metric space. 

Given a sequence of real-valued functions $f_{i}:X_{i}\rightarrow\mathbb{R}$,
we can consider the ultraproduct $f^{\omega}$ defined as follows:
if $x:=[(x_{i})]\in X^{\omega}$, then we make the pointwise definition
\[
f^{\omega}(x):=[f_{i}(x_{i})].
\]
Of course, $[f_{i}(x_{i})]\in\mathbb{R}^{\omega}$, so that $f^{\omega}:X^{\omega}\rightarrow\mathbb{R}^{\omega}$.
(This is indicative of the functoriality of the ultraproduct construction.)
Likewise, given a sequence of metric spaces $(X_{i},d_{i})$, we can
consider the ultraproduct of metric spaces $(X^{\omega},d^{\omega})$,
where 
\[
d^{\omega}(x,x^{\prime}):=[d_{i}(x_{i},x_{i}^{\prime})].
\]
Hence the space $(X^{\omega},d^{\omega})$ is a \emph{hyperreal }metric
space. Analytically this space may be ill-behaved, but it benefits
from the Theorem of \L o\'{s} (stated in one (informal) form as Theorem
10.38 in \cite{dructu2018geometric}) which indicates that \emph{logically}
$(X^{\omega},d^{\omega})$ is a well-behaved limit of the sequence
$(X_{i},d_{i})$. 

The aim of the \emph{metric ultralimit} construction, which we now
introduce, is to produce a real metric space which inherits as much
structure as possible from $(X^{\omega},d^{\omega})$. One problem
to overcome is that the space $(X^{\omega},d^{\omega})$ will, in
general, contain points which are hyperfinite distance apart, that
is, $d^{\omega}(x,x^{\prime})\in\mathbb{R}^{\omega}\backslash\mathbb{R}$
and $d^{\omega}(x,x^{\prime})>r$ for every $r\in\mathbb{R}$. To
address this, we instead work with pointed metric spaces $(X_{i},d_{i},e_{i})$
where $e_{i}$ is the ``origin'' of $X_{i}$. Then $e:=[(e_{i})]\in X^{\omega}$
can be taken as the origin of $X^{\omega}$. Let $X_{lim}^{\omega}:=\{x\in X^{\omega}:st(d^{\omega}(e,x))<\infty\}$.
Then we define the \emph{metric ultralimit} of $(X_{i},d_{i},e_{i})$,
denoted by $(\hat{X},\hat{e},\hat{d})$ or $(\hat{X},\hat{d})$ (if
the distinguished points $e_{i}$ and $\hat{e}$ are evident, or unimportant)
by
\[
(\hat{X},\hat{d},\hat{e}):=(X_{\lim}^{\omega},d^{\omega})/\approx_{d^{\omega}}
\]
where the equivalence relation $\approx_{d^{\omega}}$ is given by
$x\approx_{d^{\omega}}x^{\prime}\iff st(d^{\omega}(x,x^{\prime}))=0$,
or equivalently $x\approx_{d^{\omega}}x^{\prime}\iff d^{\omega}(x,x^{\prime})\approx0$.
Here, $\hat{d}$ is the natural quotient of $d^{\omega}$ under $\approx_{d^{\omega}}$:
given $x,x^{\prime}\in X_{\lim}^{\omega}$ and associated equivalence
classes $[x]_{\approx_{d^{\omega}}}$, $[x^{\prime}]_{\approx_{d^{\omega}}}$,
we define 
\[
\hat{d}([x]_{\approx_{d^{\omega}}},[x^{\prime}]_{\approx_{d^{\omega}}}):=st(d^{\omega}(x,x^{\prime})).
\]
(Of course, $x=[(x_{i})_{i\in\mathbb{N}}]$ for some sequence $(x_{i})_{i\in\mathbb{N}}$
of points each belonging to $X_{i}$; but we will generally avoid
burdensome notation like $\left[\left[\left(x_{i}\right)_{i\in\mathbb{N}}\right]\right]_{\approx_{d^{\omega}}}$
if possible.) Lastly, $\hat{e}$ is the equivalence class containing
$[(e_{i})]$. The structure $(\hat{X},\hat{d})$ is a bona fide (real)
metric space, unlike $(X^{\omega},d^{\omega})$. %

By way of analogy with the mapping $st(\cdot)$ from the nearstandard
elements of $\mathbb{R}^{\omega}$ to $\mathbb{R}$, in the sequel
we also use $st_{\hat{d}}(\cdot):(X_{\lim}^{\omega},d^{\omega})\rightarrow(\hat{X},\hat{d})$
to denote the map which takes a point $x^{\omega}=[(x_{i})]$ in $(X_{\lim}^{\omega},d^{\omega})$
and returns $[x^{\omega}]_{\approx_{d^{\omega}}}\in\hat{X}$; we also
say that $st_{\hat{d}}(x^{\omega})$ is the \emph{pushdown} of $x^{\omega}$
(with respect to the metric $\hat{d}$). Dually, given a point $y\in\hat{X}$,
we say that $y^{\omega}=[(y_{i})]\in X_{\lim}^{\omega}$ is a \emph{lifting}
of $y$ provided $y^{\omega}\in st_{\hat{d}}^{-1}(y)$. 
\begin{example}
For a trivial example of a metric ultraproduct, take $\mathbb{R}$
with the distinguished point $0$. The ultraproduct of the constant
sequence $(\mathbb{R},|\cdot|,0)$, as pointed metric spaces, is $(\mathbb{R}^{\omega},|\cdot|^{\omega},0)$;
the metric ultralimit of the constant sequence $(\mathbb{R},|\cdot|,0)$
is therefore just $(\mathbb{R},|\cdot|,0)$. However, it can be shown
that the metric ultralimit of the constant sequence $(\mathbb{Q},|\cdot|_{\mathbb{R}},0)$
is not $(\mathbb{Q},|\cdot|_{\mathbb{R}},0)$ but $(\mathbb{R},|\cdot|_{\mathbb{R}},0)$
--- this is related to the fact that the metric ultralimit is automatically
a complete metric space%
. %

For a nontrivial example, one might consider a sequence of pointed
metric measure spaces $(X_{i},e_{i},d_{i},\mu_{i})$ as well as their
associated Hilbert spaces $L^{2}(X_{i},\mu_{i})$. Then we can both
study the metric ultralimit $(\hat{X},\hat{d},\tilde{\mu})$ (what
exactly the ``right'' limit measure $\tilde{\mu}$ is, we address
in the next subsection, when considering ultraproducts of metric measure
spaces --- it turns out to be closely related to the \emph{Loeb measure}
associated to the ultraproduct of the measures $\mu_{i}$), and also,
if we make $L^{2}(X_{i},\mu_{i})$ into a pointed metric space with
the ``origin'' as the constant zero function and, of course, the
metric as $\Vert\cdot\Vert_{L^{2}(\mu_{i})}$, we can study the metric
ultralimit $(\widehat{L^{2}(X,\mu)},\widehat{\Vert\cdot\Vert_{L^{2}(\mu)}})$.
The relationship between $L^{2}(\hat{X},\tilde{\mu})$ and $\widehat{L^{2}(X,\mu)}$
is not \emph{a priori} obvious, but is addressed briefly in \cite{heinrich1980ultraproducts}
(see Theorem 5.2 therein, for a related statement about Orlicz spaces,
but note that much of their terminology conflicts with ours); it turns
out that $L^{2}(\hat{X},\tilde{\mu})$ embeds isometrically into $\widehat{L^{2}(X,\mu)}$
as a closed subspace. We encounter a similar situation below in the
context of Wasserstein spaces, in Theorem \ref{thm:isometric embedding}.
\end{example}

\begin{rem*}
(Relationship between metric ultralimits and Gromov-Hausdorff limits)
We emphasize that the metric ultralimit construction should be understood
as a generalization of Gromov-Hausdorff limits. Theorem 10.46 in \cite{dructu2018geometric}
indicates (in different notation) that, given a sequence of pointed,
\emph{proper} metric spaces $(X_{i},e_{i},d_{i})$ converging to a
Gromov-Hausdorff limit $(X,e,d)$, then $(X,e,d)$ and $(\hat{X},\hat{d})$
are isometric. On the other hand, the metric ultralimit of a sequence
of pointed metric spaces is guaranteed to exist even when no Gromov-Hausdorff
limit does (but depends on the choice of ultrafilter, and in the absence
of a Gromov-Hausdorff limit the different metric ultralimits with
respect to different ultrafilters need not be isometric to each other).
At the same time, the metric ultralimit remains a well-defined object
when considering non-proper (or even non-locally compact) metric spaces,
like say $(\mathcal{P}_{2}(X),W_{2})$ with $X$ a noncompact metric
space.
\end{rem*}
For the sake of motivation, we give three examples of metric ultralimits
which are of some intrinsic interest.
\begin{example}
\label{exa:(Special-metric-ultralimits)}(Special metric ultralimits)
Fix a metric space $(X,d)$. 

(1) \emph{The constant ultralimit}. Fix a point $x_{0}\in X$, and
take each term of our sequence of metric spaces to simply be $(X,x_{0},d)$.
Then we can form the space $(\hat{X},\hat{d})$ just as we would for
a non-constant sequence. (This space is also called the \emph{nonstandard
hull of the ultrapower of $X$}.) In general, if $X$ is locally sequentially
compact (i.e. the Bolzano-Weierstrass theorem holds in the completion),
then $(\hat{X},\hat{d})$ is isometric to the completion of $(X,d)$%
. If however $(X,d)$ is \emph{not} locally sequentially compact,
then $(\hat{X},\hat{d})$ will be much bigger than the completion
of $(X,d)$. Indeed, if $(x_{i})$ is a bounded sequence in $X$ which
fails to have any convergent subsequence, then there is no constant
$y$ in the completion of $(X,d)$ such that $d(x_{i},y)\rightarrow0$,
regardless of subsequence; this implies that, for $\varepsilon>0$
small enough, there is \emph{no }infinite subset of $\mathbb{N}$
such that $\{i\mid d(x_{i},y_{i})<\varepsilon\}\in\omega$ (where
$(y_{i})$ is a Cauchy sequence converging to $y$ in the completion
of $(X,d)$) which in turn shows that $[(x_{i})]\neq[(y_{i})]$ in
$(\hat{X},\hat{d})$. On the other hand, given any two sequences $(x_{i})$
and $(z_{i})$ which are both bounded, have no convergent subsequence
in $(X,d)$, and are asymptotically bounded away \emph{from each other}
(or more precisely, there is some $\varepsilon>0$ such that $d(x_{i},z_{i})>\varepsilon$
$\omega$-a.s.), then $(x_{i})$ and $(z_{i})$ are included in distinct
equivalence classes $[(x_{i})]$ and $[(z_{i})]$ within $\hat{X}$.
It is therefore reasonable to view $(\hat{X},\hat{d})$, as a ``locally
sequentially compactification'' of $(X,d)$ with base point $x_{0}$.

(2) \emph{Ultratangent space}. This is a nonsmooth generalization
of the tangent space of a point on a manifold. Given a distinguished
point $x_{0}\in X$, our sequence of pointed metric spaces is $(X,x_{0},\lambda_{i}d)$,
where $\lambda_{i}\rightarrow+\infty$ is a sequence of positive constants;
the tangent cone is the object $(\hat{X},x_{0},\hat{d})$, where 
\[
\hat{d}([x]_{\approx_{d^{\omega}}},[x^{\prime}]_{\approx_{d^{\omega}}}):=st(\lambda^{\omega}d^{\omega}(x,x^{\prime})):=st([(\lambda_{i}d_{i}(x_{i},x_{i}^{\prime}))])
\]
(and where $x=[(x_{i})]$ and $x^{\prime}=[(x_{i}^{\prime})]$). For
an example of usage of the ultratangent space, we refer the reader
to \cite[Chapter 11]{alexander2019alexandrov}; it is also shown in
\cite[Theorem 3.4.1]{alexander2019alexandrov} that the ultratangent
space contains the usual (nonsmooth) tangent space from Alexandrov
geometry as a subset.

(3) \emph{Asymptotic cone}. Instead of blowing up a metric space at
a point, we can, in Gromov's parlance, zoom out infinitely far away.
Take as a sequence of pointed metric spaces $(X,x_{0},\lambda_{i}^{-1}d)$
where again $\lambda_{i}\rightarrow+\infty$. Borrowing two examples
from \cite[Ch.2]{gromov1993}: if $(X,d)$ is an abelian group equipped
with the word metric (like $(\mathbb{Z}^{d},\Vert\cdot\Vert_{1})$),
and (by convention) our distinguished point is the identity element,
then the metric ultralimit of $(X,\lambda_{i}^{-1}d,e)$ (as well
as the pointed Gromov-Hausdorff limit) is isometric to $\mathbb{R}^{N}$
for some finite $N$ (in fact the rank of the group). On the other
hand, if $(X,d)$ is a $\delta$-hyperbolic space, then the pointed
Gromov-Hausdorff limit of $(X,\lambda_{i}^{-1}d,e)$ does not exist,
but the metric ultralimit of $(X,\lambda_{i}^{-1}d,e)$ is isometric
to a real tree (see e.g. \cite[Definition 3.60]{dructu2018geometric}). 
\end{example}

\begin{rem*}
Metric ultralimits, in the specific case of asymptotic cones, have
become a reasonably standard tool in metric geometry and adjacent
parts of geometric group theory \cite{alexander2019alexandrov,bridson2013metric,dructu2018geometric,kapovich2009hyperbolic,roe2003lectures}.
In the latter area, their usage dates to \cite{van1984gromov}, which
offered an alterate route, employing ultraproducts, to Gromov's theorem
on the equivalence of polynomial growth and virtual nilpotency. Subsequently
asymptotic cones were popularized in the geometric group theory community
by \cite{gromov1993}. 
\end{rem*}
We close this section by mentioning several geometric properties which
are stable under metric ultralimits. For each of these, the proof
is nearly immediate from the definitions of all the relevant objects;
we give citations to point out that these facts are standard in the
literature.
\begin{itemize}
\item If a sequence of pointed metric spaces $(X_{i},d_{i},e_{i})_{i\in\mathbb{N}}$
has the property that for $\omega$-almost all $i\in\mathbb{N}$,
$(X_{i},d_{i})$ is a geodesic metric space, then the metric ultralimit
$(\hat{X},\hat{d})$ is also a geodesic metric space \cite[Lemma 10.51]{dructu2018geometric}.
In fact, the same still holds when the spaces $(X_{i},d_{i})$ are
merely complete length spaces \cite[Corollary 3.5.2]{alexander2019alexandrov}.
\item For each (or even $\omega$-almost all) $(X_{i},d_{i},e_{i})$, suppose
that $(X_{i},d_{i})$ is a $CAT(K_{i})$ space, and suppose that $K:=st([(K_{i})])\in\mathbb{R}$.
Then $(\hat{X},\hat{d})$ is a $CAT(K)$ space (see \cite[Proposition 8.1.6]{alexander2019alexandrov}
or \cite[Lemma 10.53]{dructu2018geometric}). In other words, synthetic
upper sectional curvature bounds in the sense of Alexandrov are stable
under metric ultralimits. 
\item Likewise, for each (or even $\omega$-almost all) $(X_{i},d_{i},e_{i})$,
suppose that $(X_{i},d_{i})$ is a $CBB(K_{i})$ space, and suppose
that $K:=st([(K_{i})])\in\mathbb{R}$. Then $(\hat{X},\hat{d})$ is
a $CBB(K)$ space \cite[Proposition 7.1.7]{alexander2019alexandrov}.
In other words, synthetic lower sectional curvature bounds in the
sense of Alexandrov are stable under metric ultralimits.
\end{itemize}
On the other hand, the author is not aware of any prior direct investigation
into the stability of synthetic \emph{Ricci curvature} lower bounds
with respect to ultralimits, either in the sense of Lott-Sturm-Villani
(which we have already mentioned) or in the sense of Bakry-Émery (which
we have not discussed, but see e.g. \cite{bakry2013analysis} for
an overview, and \cite{ambrosio2015bakry,erbar2015equivalence} for
the relationship of this notion to Lott-Sturm-Villani synthetic Ricci
curvature bounds).

\subsection{Loeb measures and metric measure spaces}

The theory of Loeb measures, originally introduced in \cite{loeb1975conversion},
is developed in many sources, for instance \cite{albeverio2009nonstandard,cutland2000loeb,goldblatt2012lectures,loeb2015nonstandard},
albeit with minor variations; our setup is very similar to the one
from \cite{conley2013ultraproducts} (which considers probability
measures on standard Borel spaces, where we consider Radon probability
measures on complete metric spaces), and only slightly more general
than that of \cite{elek2012measure} or \cite{bergelson2014multiple}
(which restrict attention to ultraproducts of discrete measures). 

Consider as before a sequence of pointed metric spaces $(X_{i},d_{i},e_{i})$.
Each $(X_{i},d_{i},e_{i})$ comes equipped with a Borel $\sigma$-algebra
$\mathcal{B}_{i}$ generated by $d_{i}$. When we pass to the ultraproduct
$(X^{\omega},d^{\omega},e^{\omega})$, we can also equip $X^{\omega}$
with the \emph{ultraproduct measure algebra} $\mathcal{B}^{\omega}:=\prod_{\omega}\mathcal{B}_{i}$.
Note that typically $\mathcal{B}^{\omega}$ is \emph{not} closed under
countable unions. Here, an arbitrary element of $\mathcal{B}^{\omega}$
has the form $\prod_{\omega}B_{i}$ for $B_{i}\in\mathcal{B}_{i}$
(and of course, $x^{\omega}\in X^{\omega}$ belongs to $B^{\omega}$
iff $x_{i}\in B_{i}$ $\omega$-a.s). Likewise, given a sequence of
Borel measures $\mu_{i}$ on $X_{i}$, these are of course each a
function $\mu_{i}:\mathcal{B}_{i}\rightarrow[0,1]$ satisfying the
usual axioms. Consequently the ultraproduct $\mu^{\omega}$ is a set-valued
function from $\mathcal{B}^{\omega}$ to $[0,1]^{\omega}$, with $\mu^{\omega}(B^{\omega}):=\left[\left(\mu_{i}(B_{i})\right)\right]$.

Given an ultraproduct measure algebra $\mathcal{A}^{\omega}$ on $X^{\omega}$,
let $\sigma(\mathcal{A}^{\omega})$ denote the\emph{ $\sigma$-algebra}
generated by $\mathcal{A}^{\omega}$. This is the smallest $\sigma$-algebra
containing $\mathcal{A}^{\omega}$. If we have a distinguished ultraproduct
measure $\mu^{\omega}$ in mind, we can also perform the following
construction%
: the function 
\[
st\circ\mu^{\omega}:\mathcal{A}^{\omega}\rightarrow[0,1]
\]
\[
st\circ\mu^{\omega}(A^{\omega}):=st(\mu^{\omega}(A^{\omega}))
\]
turns out to be a premeasure on the space $X^{\omega}$; we complete
$st\circ\mu^{\omega}$ by invoking the Carathéodory extension theorem,
and denote the resulting measure on $\hat{X}$ by $\mu_{L}$, and
denote by $L(\mathcal{A}^{\omega})$ the set of $\mu_{L}$-measurable
sets. The objects $\mu_{L}$ and $L(\mathcal{A}^{\omega})$ are denoted
the \emph{Loeb measure} associated to $\mu^{\omega}$, and the \emph{Loeb
$\sigma$-algebra }generated by $\mathcal{A}^{\omega}$, respectively.
We will also sometimes write $\mu_{L}=\text{Loeb}(\mu^{\omega})$
to indicate that $\mu_{L}$ is the Loeb measure associated to $\mu^{\omega}$.

Let us provide some motivation. Suppose we wanted to develop a theory
of ``ultralimit of metric spaces with distinguished functionals'':
we would take tuples $(X_{i},d_{i},e_{i},F_{i})$, and we would be
interested in the natural analogue of the ultralimit $(\hat{X},\hat{d},\hat{e})$.
Certainly, when generating the \emph{ultraproduct}, we already \emph{have
}a notion of ultraproduct of functions, so the naive way to define
``ultraproduct of pointed metric spaces with distinguished functionals''
is just to also take the ultraproduct $F^{\omega}$ of the $F_{i}$'s,
and attach it to the ultraproduct of pointed metric spaces, to get
the tuple $(X^{\omega},d^{\omega},e^{\omega},F^{\omega})$. Likewise,
the naive way to produce a function $\hat{F}:\hat{X}\rightarrow\mathbb{R}$,
given $F^{\omega}:X^{\omega}\rightarrow\mathbb{R}^{\omega}$, is to
post-compose with the standard part map, namely, 
\[
\hat{F}(st[(x)]):=st(F^{\omega}[(x)])
\]
and then check whether this operation is well-defined; if so, we have
our limiting ``ultralimit of pointed metric spaces with distinguished
functionals'', namely $(\hat{X},\hat{d},\hat{e},\hat{F})$.

In order to define a reasonable notion of ``ultraproduct of pointed
metric \emph{measure} spaces'', we would like to do something like
this for measures, and push down a measure $\mu^{\omega}$ on $X^{\omega}$
onto a measure on $\hat{X}$. But instead of applying the standard
part map to points in $X^{\omega}$, we would now apply it to sets
in the measure algebra $\mathcal{A}^{\omega}$; and then one would
need to check that given a Borel set $B$ on $(\hat{X},\hat{d})$,
that $st^{-1}(B)$ is $\mathcal{A}^{\omega}$-measurable; with this
in hand, one could define a measure $\mu$ on $\hat{X}$ by $\mu(B)=\mu^{\omega}(st^{-1}(B))$.
But now the problems arise: it turns out that $st$ is \emph{not}
measurable from $\mathcal{A}^{\omega}$ to the Borel sets on $\hat{X}$
--- for one thing, the latter is a $\sigma$-algebra, thus countably
additive, while $\mathcal{A}^{\omega}$ is not! Therefore, we need
a \emph{surrogate} object for $\mu^{\omega}$ that \emph{can} be pushed
down onto $\hat{X}$ (at least under some regularity conditions),
and it so happens that the Loeb measure $\mu_{L}$ is the correct
surrogate object for the job. (Indeed, the Loeb measure was devised
precisely to be such a surrogate, in the original paper \cite{loeb1975conversion},
just in a slightly different setting than ours.)

In the sequel we require several facts connecting the measures $\mu^{\omega}$
and $\mu_{L}$, and their associated integrals and integrable functions.
{} These facts are all now classical in the Loeb measure literature;
proofs may be found, using various setups of the ``underlying machinery''
of nonstandard analysis, \cite{albeverio2009nonstandard,anderson82representation,cutland2000loeb,svindland2018ultrafilter},
for instance. In particular, proofs of the following facts require
only that the ``universe'' of internal objects satisfies \L o\'{s}'s
theorem together with the \emph{countable saturation property}, and
this latter property is automatically imposed by any non-principal
ultrafilter on $\mathbb{N}$. To provide a more specific reference
for the reader, we mention \cite[Ch. 2.2]{svindland2018ultrafilter},
which uses the same foundational setup as ours but slightly different
terminology.
\begin{fact}
(see \cite[Prop. 2.2.11 and Cor. 2.2.12]{svindland2018ultrafilter})\label{fact:Loeb symmetric difference}
Let $(X^{\omega},\mathcal{A}^{\omega},\mu^{\omega})$ be an internal
measure space, with associated Loeb measure space $(X^{\omega},L(\mathcal{A}^{\omega}),\mu^{\omega})$.
Let $A\in L(\mathcal{A}^{\omega})$. Then:
\begin{enumerate}
\item For every $\varepsilon>0$, there exist $B^{\omega},C^{\omega}\in\mathcal{A^{\omega}}$
with $B^{\omega}\subseteq A\subseteq C^{\omega}$ and $\mu^{\omega}(C^{\omega}\backslash B^{\omega})<\varepsilon$.
\item There exists an $A^{\omega}\in\mathcal{A}^{\omega}$ such that $\mu_{L}(A\Delta A^{\omega})=0$.
\end{enumerate}
\end{fact}

\begin{defn}
Fix an internal measure space $(X^{\omega},\mathcal{A}^{\omega},\mu^{\omega})$.
Let $f^{\omega}:X^{\omega}\rightarrow\mathbb{R}^{\omega}$ be an internal
function which is $\mathcal{A}^{\omega}$-measurable. We say that
$f^{\omega}$ is $S$-integrable provided that $st(f^{\omega}):X^{\omega}\rightarrow\mathbb{R}\cup\{-\infty,\infty\}$
is $\mu_{L}$-integrable and 
\[
st\left(\int_{X^{\omega}}|f^{\omega}|d\mu^{\omega}\right)=\int_{X^{\omega}}|st(f^{\omega})|d\mu_{L}.
\]
\end{defn}

\begin{fact}
(see \cite[Thm. 2.2.19]{svindland2018ultrafilter} for the case where
$\mu^{\omega}(X^{\omega})<\infty$)\label{fact:Loeb integrability}
Let $f^{\omega}:X^{\omega}\rightarrow\mathbb{R}^{\omega}$ be an internal
function which is $\mathcal{A}^{\omega}$-measurable. Then, $f^{\omega}$
is $S$-integrable iff
\begin{enumerate}
\item $\int_{X^{\omega}}|f^{\omega}|d\mu^{\omega}$ is finite,
\item if $A\in\mathcal{A}^{\omega}$ and $\mu^{\omega}(A)\approx0$, then
$\int_{A}|f^{\omega}|d\mu^{\omega}\approx0$; and
\item if $A\in\mathcal{A}^{\omega}$ and $f^{\omega}(a)\approx0$ for all
$a\in A$, then $\int_{A}|f^{\omega}|d\mu^{\omega}\approx0$. (Note
this automatically holds if $\mu^{\omega}(X^{\omega})$ is finite.)
\end{enumerate}
\end{fact}

\begin{defn}
Fix an internal measure space $(X^{\omega},\mathcal{A}^{\omega},\mu^{\omega})$.
Let $g:X^{\omega}\rightarrow\mathbb{R}$. We say that $f^{\omega}:X^{\omega}\rightarrow\mathbb{R}^{\omega}$
is an \emph{internal lifting} of $g$ provided that $st(f^{\omega}(x))=g(x)$
holds $\mu_{L}$-almost everywhere.
\end{defn}

\begin{fact}
\label{fact:S-integrable lifting}A function $g:X^{\omega}\rightarrow\mathbb{R}$
is $\mu_{L}$-integrable iff it has an $S$-integrable internal lifting
$f^{\omega}$. In this case, $st\left(\int_{X^{\omega}}|f^{\omega}|d\mu^{\omega}\right)=\int_{X^{\omega}}|g|d\mu_{L}$.
\end{fact}

\begin{proof}
The forward direction is provided by \cite[Thm. 2.2.21]{svindland2018ultrafilter};
the converse direction is automatic from the definitions of $S$-integrability
and internal lifting.
\end{proof}
We also have the following handy sufficient condition for as internal
function $f^{\omega}$ to be $S$-integrable.
\begin{fact}
(see \cite[Prop. 2.2.20]{svindland2018ultrafilter})\label{fact:p>1 S-integrable}
Let $p>1$. Suppose that $st\left(\int_{X^{\omega}}|f^{\omega}|^{p}d\mu^{\omega}\right)<\infty$.
Then, $f^{\omega}$ is $S$-integrable, and also $st\circ f^{\omega}\in L^{p}(\mu_{L})$.
\end{fact}

\begin{rem*}
Under the heuristic that ``taking an ultraproduct, then taking a
standard part, is much like taking a limit'' (indeed this is literally
the definition of an ultralimit, say, of real numbers), we see that
the $S$-integrable functions are precisely those for which an ``ultralimit
analogue'' of the dominated convergence theorem (or, more precisely,
the Vitali convergence theorem) holds. In particular, Fact \ref{fact:p>1 S-integrable}
is related to the fact that sequences of functions bounded in $L^{p}$
with $p>1$ are automatically uniformly integrable (by the theorem
of de la Vallée Poussin), so that the Vitali convergence theorem is
applicable.
\end{rem*}
This concludes our background discussion of Loeb measures. Before
we give the definition of a \emph{metric-measure ultralimit}, however,
there is a snag, more or less general to all notions of convergence
for metric measure spaces, and observed for instance in \cite{sturm2006geometry1}.
While we can use some specific construction to \emph{produce} a candidate
metric measure space limiting object, in a specific instantiation,
it's desirable for the property ``is a metric measure space limit
of a given sequence of metric measure spaces'' to be invariant under
isomorphism of metric measure spaces. And if our candidate limit definition
has the form of ``the underlying pointed metric space converge, with
respect to our favourite notion of convergence of pointed metric spaces,
and also the measures on top converge in some associated sense'',
then invariance under metric measure space isomorphism can \emph{fail}
when the measures do not have full support, in which case we should
really only care about isometry of the underlying measure spaces \emph{on
the supports of the measures} (and also the distinguished points $e$
in the two spaces!).

With all this preparatory discussion, we can now present the following
definition. Due to the restricted interests of our present work, we
specialize the definition for Radon probability measures only. %
{} We mention that our definition is a variant on the one introduced
by Elek in \cite{elek2012samplings}, who we believe to be the first
to consider ultralimits of metric measure spaces.
\begin{defn}
\label{def:(ultralimit-of-pmm-space}(ultralimit of pointed metric
measure spaces)\emph{ Let $(X_{i},d_{i},e_{i},\mu_{i})$ be a sequence
of tuples where $(X_{i},d_{i},e_{i})$ are pointed metric spaces and
$\mu_{i}$ is a Radon probability measure on $X_{i}$. We say that
$(X_{i},d_{i},e_{i},\mu_{i})$ converges to $(X,d,e,\mu)$ (with $\mu$
also a Radon probability measure) in the sense of pointed metric measure
ultralimits if}
\begin{enumerate}
\item \emph{The Loeb measure $\mu_{L}$ associated to $\mu^{\omega}$ pushes
forward via $\text{st}_{\hat{d}}$ to a Radon probability measure
$\tilde{\mu}$ on $(\hat{X},\hat{d},\hat{e})$, that is, $\tilde{\mu}=\mu_{L}\circ st_{\hat{d}}^{-1}$;
and}
\item \emph{There exists a partial isometry $\iota$ between some $X_{0}\subseteq X$
and some $\hat{X}_{0}\subseteq\hat{X}$, such that $\iota$ maps $e$
to $\hat{e}$ and $supp(\mu)$ to $supp(\tilde{\mu})$, and for every
Borel set $B\subset supp(\tilde{\mu})$, we have that $\mu(\iota^{-1}(B))=\tilde{\mu}(B)$
(and vice versa). }
\end{enumerate}
\end{defn}

Note that there is a certain asymmetry between the metric and measure
structures here, because, while we can \emph{always} produce a limiting
metric space (in the metric ultraproduct sense), it is \emph{not}
generally the case that condition (1) is satisfied \cite{pasqualetto2021ultralimits},
so in particular there does \emph{not} always exist a pointed metric
measure ultralimit of a sequence of pointed metric measure spaces
$(X_{i},d_{i},e_{i},\mu_{i})$. %
{} However, Lemma \ref{lem:Loeb tightness pushdown criterion} below
gives a characterization of which Loeb probability measures satisfy
condition (1). 

\subsubsection*{Note.}

While our work was in progress, we learned of the recent preprint
\cite{pasqualetto2021ultralimits}, which is also concerned with ultralimits
of metric measure spaces. \cite{pasqualetto2021ultralimits} considers
a notion of metric measure ultralimit which allows for a more general
limiting measure $\mu$ (not even necessarily Borel), and investigates
many structural properties of metric measure ultralimits. In particular,
\cite{pasqualetto2021ultralimits} shows that the notion of metric
measure ultralimits they consider extends the notion of \emph{pointed
measured Gromov }(pmG) convergence from \cite{gigli2015convergence},
and moreover our Definition \ref{def:(ultralimit-of-pmm-space} is
actually \emph{equivalent} to pmG convergence, \emph{precisely} because
we assume that the reference measure on the limiting space is Radon.
(See in particular \cite[Theorems 8.3, 11.4, and 12.2]{pasqualetto2021ultralimits}
and our remark following Lemma \ref{lem:Loeb tightness pushdown criterion}
below.) This actually allows \cite{pasqualetto2021ultralimits} to
deduce various synthetic geometric ultralimit stability results by
indirect means, including our Theorem \ref{thm:main}, as they discuss
in their introduction, in particular by quoting existing results on
$CD(K,\infty)$ stability under pmG convergence from works such as
\cite{gigli2015convergence}.

~

Since our definition of ``metric-measure ultralimit'' stipulates
that $\hat{\mu}$ is a Loeb measure, it would be nice to know that
\emph{any} reasonable measure $\mu$ on $(X,d)$ can be viewed as
a Loeb measure. Similar results are already known in the literature,
for instance:
\begin{thm}
(Anderson \cite{anderson82representation}) Let $\nu$ be a Radon
measure on a Hausdorff space $(X,\mathcal{T})$. Then the measure
space $(X,\overline{\sigma(\mathcal{T})},\nu)$ can be represented
as a Loeb measure space, where $\overline{\sigma(\mathcal{T})}$ is
the completion of the Borel $\sigma$-algebra generated by $\mathcal{T}$. 
\end{thm}

Therefore, we aim to prove a variant of Anderson's theorem, where
$\nu$ lives on $\hat{X}$ and its Loeb measure representative lives
on $X^{\omega}$. We do so below, in Theorem \ref{thm:radon representation}. 

\section{Ultralimits of Wasserstein spaces\label{sec:Ultralimits-of-Wasserstein}}

Given a sequence $(X_{i},d_{i},e_{i})$ of metric spaces with distinguished
points (which we do not always notate explicitly), we can form the
metric ultralimit space $(\hat{X},\hat{d})$ as discussed in the previous
section. At the same time, for each $i\in\mathbb{N}$, we can consider,
for $p\in[1,\infty)$, the (pointed) $p$-Wasserstein space atop $(X_{i},d_{i},e_{i})$,
namely $(\mathcal{P}_{p}(X_{i}),W_{p},\delta_{e_{i}})$; as well as
the $p$-Wasserstein space atop ($\hat{X},\hat{d})$, namely $(\mathcal{P}_{p}(\hat{X}),W_{p})$
--- recall, from the previous section, that this denotes the space
of Radon probability measures on the (complete) metric space $\hat{X}$
with finite $p$th moments, equipped with the metric 
\[
W_{p}(\mu,\nu):=\left(\inf_{\gamma\in\Pi(\mu,\nu)}\int_{\hat{X}\times\hat{X}}\hat{d}(x,y)^{p}d\gamma(x,y)\right)^{1/p}
\]
where $\Pi(\mu,\nu)$ is the space of couplings of the measures $\mu$
and $\nu$. 

At the same time, the sequence $(\mathcal{P}_{p}(X_{i}),W_{p},\delta_{e_{i}})_{i\in\mathbb{N}}$
is\emph{ }also a sequence of metric spaces (in fact a sequence of
complete metric spaces, provided each $(X_{i},d_{i})$ is complete),
so we can also consider the metric ultralimit of the sequence $(\mathcal{P}_{p}(X_{i}),W_{p})_{i\in\mathbb{N}}$,
which we denote $(\widehat{\mathcal{P}_{p}(X)},\hat{W}_{p})$. Let
us recap the metric ultraproduct construction, reviewed in the previous
section. in this specific instance. Explicitly, this space is constructed
by first forming the ultraproduct $\left(\mathcal{P}_{p}(X)^{\omega},W_{P}^{\omega},[(\delta_{e_{i}})]\right)$,
where 
\[
\mathcal{P}_{p}(X)^{\omega}:=\left\{ [(\mu_{i})_{i\in\mathbb{N}}]:\forall i\in\mathbb{N},\mu_{i}\in\mathcal{P}_{p}(X_{i})\right\} 
\]
and 
\[
W_{p}^{\omega}\left(\left[(\mu_{i})_{i\in\mathbb{N}}\right],\left[(\nu_{i})_{i\in\mathbb{N}}\right]\right):=\left[\left(W_{p}(\mu_{i},\nu_{i})\right)\right].
\]
We then restrict to the ``limited'' subset of $\mathcal{P}_{p}(X)^{\omega}$,
which is the subset of points $[(\mu_{i})]\in\mathcal{P}_{p}(X)^{\omega}$
which are a nearstandard distance from the distinguished point $[(\delta_{e_{i}})]$,
and which we denote $\mathcal{P}_{p}(X)_{\lim}^{\omega}$: 
\[
\mathcal{P}_{p}(X)_{\lim}^{\omega}:=\left\{ [(\mu_{i})]\in\mathcal{P}_{p}(X)^{\omega}:st\left(W_{p}^{\omega}\left([(\mu_{i})],[(\delta_{e_{i}})]\right)\right)<\infty\right\} .
\]
Finally, $(\widehat{\mathcal{P}_{p}(X)},\hat{W}_{p})$ is formed by
quotienting $\mathcal{P}_{p}(X)_{\lim}^{\omega}$ with respect to
the equivalence relation $W_{p}^{\omega}\left([(\mu_{i})],[(\nu_{i})]\right)\approx0$:
\[
\widehat{\mathcal{P}_{p}(X)}:=\mathcal{P}_{p}(X)_{\lim}^{\omega}/\approx_{W_{p}^{\omega}}.
\]
Points in $\widehat{\mathcal{P}_{p}(X)}$ (i.e. equivalence classes
of points $[(\mu_{i})]\in\mathcal{P}_{p}(X)_{\lim}^{\omega}$) are
denoted by $\hat{\mu}$ or $st_{\hat{W}_{p}}[(\mu_{i})]$, and 
\[
\hat{W}_{p}(\hat{\mu},\hat{\nu}):=st\left(W_{p}^{\omega}\left([(\mu_{i})],[(\nu_{i})]\right)\right)\text{ where }st_{\hat{W}_{p}}[(\mu_{i})]=\hat{\mu},st_{\hat{W}_{p}}[(\nu_{i})]=\hat{\nu}.
\]
(If one prefers, the space $(\widehat{\mathcal{P}_{p}(X)},\hat{W}_{p})$
is also a pointed metric space in a canonical fashion: simply add
the distinguished point $st_{\hat{W}_{p}}[(\delta_{e_{i}})]$, namely
the equivalence class containing the ultraproduct of the distinguished
points $\delta_{e_{i}}$.%
) 

The following theorem addresses the relationship between $(\mathcal{P}_{p}(\hat{X}),W_{p})$
and $(\widehat{\mathcal{P}_{p}(X)},\hat{W}_{p})$.
\begin{thm}
\label{thm:isometric embedding}Let $p\in[1,\infty)$. Suppose that
$(X_{i},d_{i})$ is a sequence of pointed complete metric spaces with
metric ultralimit $(\hat{X},\hat{d})$. Then,%
{} there is a canonical isometric embedding of $(\mathcal{P}_{p}(\hat{X}),W_{p})$
into $(\widehat{\mathcal{P}_{p}(X)},\hat{W}_{p})$%
, extending the map 
\[
(\mathcal{P}_{p}(\hat{X}),W_{p})\ni\delta_{st_{\hat{d}}[(x_{i})]}\longmapsto st_{\hat{W}_{p}}\left[\left(\delta_{x_{i}}\right)\right]\in(\widehat{\mathcal{P}_{p}(X)},\hat{W}_{p}).
\]
\end{thm}

\begin{proof}
Fix an $N\in\mathbb{N}$; let $\mu=\frac{1}{N}\sum_{j=1}^{N}\delta_{\hat{x}_{j}}$
and $\mu^{\prime}=\frac{1}{N}\sum_{j=1}^{N}\delta_{\hat{x}_{j}^{\prime}}$,
where the $\hat{x}_{j}$'s and $\hat{x}_{j}^{\prime}$'s are arbitrary
points in $\hat{X}$; clearly $\mu,\mu^{\prime}\in\mathcal{P}_{p}(\hat{X})$.
For each such point, let $[(x_{i,j})]\in X^{\omega}$ be a lifting
of $\hat{x}_{j}$ (and similarly for $[(x_{i,j}^{\prime})]$ and $\hat{x}_{j}^{\prime}$).
Then we can also consider Dirac measures on each $X_{i}$, of the
form $\mu_{i}:=\frac{1}{N}\sum_{j=1}^{N}\delta_{x_{i,j}}$ for $x_{i,j}\in X_{i}$,
and similarly $\mu_{i}^{\prime}:=\frac{1}{N}\sum_{j=1}^{N}\delta_{x_{i,j}^{\prime}}$.
Then, we can consider $[(\mu_{i})],[(\mu_{i}^{\prime})]\in\mathcal{P}_{p}(X)^{\omega}$
as well as $\hat{\mu}=st_{\hat{W}_{p}}[(\mu_{i})]$ and $\hat{\mu^{\prime}}=st_{\hat{W}_{p}}[(\mu_{i}^{\prime})]$
in $\widehat{\mathcal{P}_{p}(X)}$; and it holds that
\[
\hat{W_{p}}(\hat{\mu},\hat{\mu^{\prime}}):=st(W_{p}^{\omega}([(\mu_{i})],[(\mu_{i}^{\prime})]):=st([(W_{p}(\mu_{i},\mu_{i}^{\prime}))]).
\]
On the other hand, %
we'd like to say that $W_{p}(\mu,\mu^{\prime})=\hat{W}_{p}(\hat{\mu},\hat{\mu^{\prime}})$.
Indeed, this is actually sufficient to prove the theorem: consider
a Cauchy sequence $(\mu_{k})$ in $(\mathcal{P}_{p}(\hat{X}),W_{p})$,
converging to $\nu$, such that each $\mu_{k}$ is an average of Dirac
masses. Then (possibly by representing $\mu_{k}$ and $\mu_{k^{\prime}}$
as averages of an artificially larger, equal number of Dirac masses;
the least common multiple suffices) we have that $W_{p}(\mu_{k},\mu_{k^{\prime}})=\hat{W}_{p}(\hat{\mu}_{k},\hat{\mu}_{k^{\prime}})$,
which implies that $(\hat{\mu}_{k})$ is a Cauchy sequence in $(\widehat{\mathcal{P}_{p}(X)},\hat{W}_{p})$.
Since $(\mathcal{P}_{p}(\hat{X}),W_{p})$ and $(\widehat{\mathcal{P}_{p}(X)},\hat{W}_{p})$
are both complete metric spaces, it follows that we can extend the
map 
\[
\mathcal{P}_{p}(\hat{X})\ni\mu\mapsto\hat{\mu}\in\widehat{\mathcal{P}_{p}(X)}
\]
from an isometry on the set of Dirac clouds in $\mathcal{P}_{p}(\hat{X})$,
to an isometry on all of $\mathcal{P}_{p}(\hat{X})$, by density%
. In particular, the image of $\mathcal{P}_{p}(\hat{X})$ inside $\widehat{\mathcal{P}_{p}(X)}$
under this embedding is automatically $\hat{W}_{p}$-closed. 

However, it can be readily seen that the fact that $W_{p}(\mu,\mu^{\prime})=\hat{W}_{p}(\hat{\mu},\hat{\mu^{\prime}})$
follows from Birkhoff's theorem on convex polytopes (cf. Proposition
2.1 in \cite{peyre2019computational}). That is, 
\[
W_{p}^{p}(\mu,\mu^{\prime})=\frac{1}{N}\sum_{j=1}^{N}\hat{d}^{p}(\hat{x}_{j},\hat{x}_{\sigma(j)}^{\prime})
\]
where $\sigma:\{N\}\rightarrow\{N\}$ is some permutation. But by
definition, $\hat{d}(\hat{x}_{j},\hat{x}_{\sigma(j)}^{\prime})=st((d^{\omega}([(x_{i,j})],[(x_{i,\sigma(j)}^{\prime})])$.
At the same time, for each $i\in\mathbb{N}$, 
\[
W_{p}^{p}(\mu_{i},\mu_{i}^{\prime})=\frac{1}{N}\sum_{j=1}^{N}d_{i}^{p}(x_{i,j},x_{i,\sigma_{i}(j)}^{\prime})
\]
where $\sigma_{i}:\{N\}\rightarrow\{N\}$ is again a permutation.
It then follows from \L o\'{s}'s theorem (really just careful manipulation
of our definition of an ultraproduct of functions) that
\[
\left(W_{p}^{\omega}([(\mu_{i})],[(\mu_{i}^{\prime})])\right)^{p}=\frac{1}{N}\sum_{j=1}^{N}\left(d^{\omega}(x^{\omega},x_{\sigma^{\omega}(j)}^{\prime\omega})\right)^{p}
\]
where $\sigma^{\omega}$ is the ultraproduct of the $\sigma_{i}$'s
(but since $\{N\}$ is a finite set, $\sigma^{\omega}$ is just some
permutation on $N$ elements; in particular, it is the permutation
amongst the $\sigma_{i}$'s that has $\omega$ measure 1). It follows,
from the definition of $\hat{d}$ and $\hat{W}_{2}$, that
\[
\hat{W}_{p}^{p}(\hat{\mu},\hat{\mu}^{\prime})=\frac{1}{N}\sum_{j=1}^{N}\hat{d}^{p}(\hat{x}_{j},\hat{x}_{\sigma^{\omega}(j)}^{\prime}).
\]

Moreover, it also holds by \L o\'{s}'s theorem that 
\[
\left(W_{p}^{\omega}([(\mu_{i})],[(\mu_{i}^{\prime})])\right)^{p}\leq\frac{1}{N}\sum_{j=1}^{N}\left(d^{\omega}(x^{\omega},x_{\tilde{\sigma}(j)}^{\prime\omega})\right)^{p}
\]
 for any $\tilde{\sigma}\ne\sigma^{\omega}$, which implies that 
\[
\hat{W}_{p}^{p}(\hat{\mu},\hat{\mu}^{\prime})\leq\frac{1}{N}\sum_{j=1}^{N}\hat{d}^{p}(\hat{x}_{j},\hat{x}_{\tilde{\sigma}(j)}^{\prime})
\]
for any $\tilde{\sigma}\neq\sigma^{\omega}$. In particular, 
\[
\hat{W}_{p}^{p}(\hat{\mu},\hat{\mu}^{\prime})\leq\frac{1}{N}\sum_{j=1}^{N}\hat{d}^{p}(\hat{x}_{j},\hat{x}_{\sigma(j)}^{\prime})=W_{p}^{p}(\mu,\mu^{\prime}).
\]
But we \emph{also} know that 
\[
\frac{1}{N}\sum_{j=1}^{N}\hat{d}^{p}(\hat{x}_{j},\hat{x}_{\tilde{\sigma}(j)}^{\prime})\geq W_{p}^{p}(\mu,\mu^{\prime})
\]
for any $\tilde{\sigma}\neq\sigma$. Plugging in $\sigma^{\omega}$
for $\tilde{\sigma}$ in the expression above, we conclude $\hat{W}_{p}^{p}(\hat{\mu},\hat{\mu}^{\prime})\geq W_{p}^{p}(\mu,\mu^{\prime})$.
Hence $\hat{W}_{p}^{p}(\hat{\mu},\hat{\mu}^{\prime})=W_{p}^{p}(\mu,\mu^{\prime})$
as desired. 
\end{proof}
\begin{defn}
\label{def:isometric-embedding}Let $\iota:(\mathcal{P}_{p}(\hat{X}),W_{p})\hookrightarrow(\widehat{\mathcal{P}_{p}(X)},\hat{W}_{p})$
denote the isometric embedding constructed in the proof of the previous
theorem: namely, if $\mu\in\mathcal{P}_{p}(\hat{X})$ is a discrete
measure of the form $\mu=\frac{1}{N}\sum_{j=1}^{N}\delta_{\hat{x}_{j}}$
where $\hat{x}_{j}=st_{\hat{d}}[(x_{j,i})]$, then 
\[
\iota\left(\frac{1}{N}\sum_{j=1}^{N}\delta_{\hat{x}_{j}}\right)=st_{\hat{W}_{p}}\left[\left(\frac{1}{N}\sum_{j=1}^{N}\delta_{x_{j,i}}\right)_{i\in\mathbb{N}}\right].
\]
The map $\iota$ is then extended to all of $(\mathcal{P}_{p}(\hat{X}),W_{p})$
by density.
\end{defn}

We have shown that $(\mathcal{P}_{p}(\hat{X}),W_{p})\hookrightarrow(\widehat{\mathcal{P}_{p}(X)},\hat{W}_{p})$,
but when are the two spaces isometric? The following proposition indicates
a sufficient condition. We remark that it is also possible to prove
the proposition indirectly, by appealing to existing results on the
equivalence with the Gromov-Hausdorff limit. Nevertheless, we give
such a direct, ``intrinsic'', argument below, since doing so turns
out to be more informative for us in the sequel. 

Recall that a metric space $(X,d)$ is \emph{totally bounded} if,
for every $\varepsilon>0$, there exists some integer $k(\varepsilon)$
such that $X$ can be covered with $k$ many $d$-balls of radius
$\varepsilon$. Call $k(\varepsilon)$ the $\varepsilon$-covering
number of $X$; in other words, $(X,d)$ is totally bounded if it
has finite $\varepsilon$-covering number for every $\varepsilon>0$. 

More generally, we can talk about the $\varepsilon$-covering number
of a subset of $X$, or even a family of sets belonging to different
spaces. In particular, if $(F_{i})_{i\in I}$ is a family of sets,
where each $F_{i}$ is contained in a metric space $(X_{i},d_{i})$,
we say that $(F_{i})_{i\in I}$ is \emph{uniformly totally bounded}
if there is a single $k(\varepsilon)$ which simultaneously is an
$\varepsilon$-covering number for all $F_{i}$'s, for every $\varepsilon>0$. 

For example, compact sets in metric spaces are totally bounded (conversely,
closed totally bounded sets are compact), and balls in infinite-dimensional
Banach spaces are not totally bounded. Note that totally bounded sets
automatically have finite diameter. They are also automatically separable:
take a sequence of small quantities $\varepsilon_{n}>0$ converging
to zero; then the set 
\[
\bigcup_{n\in\mathbb{N}}\{x\in X:x\text{ is a center of an \ensuremath{\varepsilon_{n}}-ball in an \ensuremath{\varepsilon_{n}}-covering of }X\}
\]
is dense, and can be taken to be countable. 
\begin{prop}
\label{prop:compact lifting}Let $(X_{i},d_{i},e_{i})$ be a sequence
of pointed metric spaces with metric ultralimit $(\hat{X},\hat{d})$.
Let $F$ be a compact set in $\hat{X}$ and let $(F_{i})_{i\in\mathbb{N}}$
be a sequence of sets, each contained in $X_{i}$, such that $st_{\hat{d}}(F^{\omega})=F$.
Then, it automatically holds that the sequence $(F_{i})_{i\in\mathbb{N}}$
is uniformly totally bounded $\omega$-a.s.; and, the isometric image
w.r.t. $\iota$ of the subspace $(\mathcal{P}_{p}(F),W_{p})$ of $(\mathcal{P}_{p}(\hat{X}),W_{p})$
is identical to $(\widehat{\mathcal{P}_{p}(F)},\hat{W}_{p})$, the
metric ultralimit of the spaces $(\mathcal{P}_{p}(F_{i}),W_{p})$. 
\end{prop}

\begin{proof}
$F\subseteq\hat{X}$ is compact iff it is both closed and totally
bounded, the latter meaning that for every $\varepsilon>0$ there
exists some integer $k(\varepsilon)$ such that $\hat{X}$ can be
covered with $k$ many $\hat{d}$-balls of radius $\varepsilon$ (call
$k(\varepsilon)$ the $\varepsilon$-covering number of $F$). %
That is, for every $\varepsilon>0$ there exist $\hat{x}_{1},\ldots,\hat{x}_{k(\varepsilon)}$
such that 
\[
(\forall\hat{x}\in F)\qquad\bigvee_{j=1}^{k(\varepsilon)}\hat{d}(\hat{x}_{j},\hat{x})<\varepsilon
\]
so we can pass to arbitrary liftings of the distinguished points $\hat{x}_{1},\ldots,\hat{x}_{k(\varepsilon)}$
of the form $x_{1}^{\omega},\ldots,x_{k(\varepsilon)}^{\omega}$,
all belonging to $st_{\hat{d}}^{-1}(F)$; in fact, we can always select
$x_{1}^{\omega},\ldots,x_{k(\varepsilon)}^{\omega}$ to belong to
$F^{\omega}:=[(F_{i})]$. Lifting all other points $\hat{x}\in F$,
we see that 
\[
(\forall x^{\omega}\in st_{\hat{d}}^{-1}(F))\qquad\bigvee_{j=1}^{k(\varepsilon)}d^{\omega}(x_{j}^{\omega},x^{\omega})<\varepsilon
\]
so in particular 
\[
(\forall x^{\omega}\in F^{\omega})\qquad\bigvee_{j=1}^{k(\varepsilon)}d^{\omega}(x_{j}^{\omega},x^{\omega})<\varepsilon.
\]
Setting $x_{j}^{\omega}=[(x_{j,i})]$ for each $j\in\{1,\ldots,k(\varepsilon)\}$,
this means that for $\omega$-almost all $i\in\mathbb{N}$,
\[
(\forall x_{i}\in F_{i})\qquad\bigvee_{j=1}^{k(\varepsilon)}d_{i}(x_{j,i},x_{i})<\varepsilon.
\]
That is, $k(\varepsilon)$ is an $\varepsilon$-covering number for
$\omega$-almost all $F_{i}$. Quantifying over all $\varepsilon>0$,
this shows that $\omega$-almost all of the $F_{i}$'s are necessarily
totally bounded, with uniform covering number $k(\varepsilon)$.

Now, observe the metric ultraproduct of the spaces $(F_{i}\cup\{e_{i}\},d_{i},e_{i})_{i\in\mathbb{N}}$
is equal to $(\hat{F}\cup\{\hat{e}\},\hat{d},\hat{e})$; in what follows,
we assume that $e_{i}\in F_{i}$ $\omega$-almost surely. (There is
no loss of generality, since the ambient space $(\hat{X},\hat{d})$
is indifferent to changes of the distinguished point that do not change
which set in $X^{\omega}$ is $X_{lim}^{\omega}$, and here we are
only moving $\hat{e}$ a finite distance, say $dist(F,\hat{e})$,
hence $e^{\omega}$ a finite distance, hence $X_{\lim}^{\omega}$
is unchanged.) At the same time, let $D>diam(F)$. Observe that
\[
(\forall x^{\omega},y^{\omega}\in st_{\hat{d}}^{-1}(F))\qquad d^{\omega}(x^{\omega},y^{\omega})<D
\]
This implies (since $\hat{e}\in K$, hence $e^{\omega}\in K^{\omega}$)
that
\[
(\forall x^{\omega}\in F^{\omega})\qquad d^{\omega}(x^{\omega},e^{\omega})<D
\]
which in turn means that in this case, $F^{\omega}=F_{\lim}^{\omega}$,
the set of points in $F^{\omega}$ which are a finite distance from
$e^{\omega}$.

Now, an inspection of \cite[Proposition 1.1]{boissard2014mean} indicates\emph{}\footnote{This is a crude use of the error estimate from \cite{boissard2014mean}:
we don't care, here, about random i.i.d. samples, rather, we just
want \emph{any} Dirac cloud with uniform weights approximating a given
measure. Specifically, the reason why a uniform approximation is
necessary is that, in the absence of such, it may be the case that
an arbitrary element of $\mathcal{P}_{p}(X)^{\omega}$ could only
be approximated by a ``Dirac cloud'' with \emph{hyperfinitely} many
atoms, or by a Dirac cloud where some of the weights are infinitesimal.
In this case our proof breaks down.

We also remark that the paper \cite{boissard2014mean} assumes that
the underlying space is Polish. This is not a problem for us, since
we have restricted to a set $F$ which is compact, hence separable.} %
{} that, since the $F_{i}$'s have $\omega$-a.s. diameter less than
$D$, and some $\omega$-a.s. uniform $\varepsilon$-covering number%
: for each $\delta>0$ there exists a \emph{uniform $N(\delta)$}
(depending only on the covering number function $k(\varepsilon)$
and $D$)\emph{ }such that for each $i\in I$ and each $\nu_{i}\in\mathcal{P}_{p}(F_{i})$,
there exists a Dirac cloud with uniform weights $\mu_{\delta,i}:=\frac{1}{N(\delta)}\sum_{\ell=1}^{N(\delta)}\delta_{x_{\ell,i}}$
with $W_{p}(\mu_{i},\nu_{i})<\delta$. 

Therefore, let $\hat{\nu}\in(\widehat{\mathcal{P}_{p}(F)},\hat{W}_{p})$
with representative $[(\nu_{i})]\in(\mathcal{P}_{p}(F)^{\omega},W_{p}^{\omega})$,
where the $\nu_{i}$'s are otherwise arbitrary; and given $\delta>0$,
let $\hat{\mu}_{\delta}$ denote the image of $[(\mu_{\delta,i})]\in\mathcal{P}_{p}(F)_{\lim}^{\omega}$
in $\widehat{\mathcal{P}_{p}(F)}$, where the $\mu_{\delta,i}$'s
are as just described. %
Likewise let $\mu_{\delta}\in\mathcal{P}_{p}(F)$ denote $\frac{1}{N(\delta)}\sum_{\ell=1}^{N}\delta_{\hat{x}_{\ell}}$,
where $\hat{x}_{\ell}=st_{\hat{d}}[(x_{\ell,i})]$ for each $\ell=1,\ldots,N$.
In Theorem \ref{thm:isometric embedding}, we showed that the isometric
embedding $\iota$ maps $\mu_{\delta}$ to $\hat{\mu}_{\delta}$.

We now select a decreasing sequence of small quantities $\delta_{k}>0$:
for each $i\in I$, let $(\mu_{i,\delta_{k}})$ denote a Cauchy sequence
of Dirac clouds in $(\mathcal{P}_{p}(F_{i}),W_{p})$ with $N(\delta)$
atoms converging to $\nu_{i}$, such that $W_{p}(\mu_{i,k},\nu_{i})<\delta_{k}$.
Then $\hat{\mu}_{\delta_{k}}$ is a Cauchy sequence in $(\widehat{\mathcal{P}_{p}(F)},\hat{W}_{p})$,
converging to $\hat{\nu}$, with the same property. But $(\mu_{\delta_{k}})$
is also a Cauchy sequence in $(\mathcal{P}_{p}(F),W_{p})$ (call its
limit $\nu$); therefore, under the isometric embedding described
in Definition \ref{def:isometric-embedding}, that $\iota\nu=\hat{\nu}$.
But since $\hat{\nu}$ was chosen arbitrarily, we conclude that the
embedding is actually a surjection. 
\end{proof}
\begin{rem*}
In the setting of the previous proposition, there is an obvious canonical
injection of $(\widehat{\mathcal{P}_{p}(F)},\hat{W}_{p})$ into $(\widehat{\mathcal{P}_{p}(X)},\hat{W}_{p})$,
as follows: the space $(\mathcal{P}(F)_{\lim}^{\omega},W_{p}^{\omega})$
is the space of internal probability measures supported on $F^{\omega}$
with nearstandard $p$th moments, equipped with hyperreal metric $W_{p}^{\omega}$,
and this space is contained as a subset inside $\mathcal{P}(X)_{\lim}^{\omega},W_{p}^{\omega})$.
The injection of $(\widehat{\mathcal{P}_{p}(F)},\hat{W}_{p})$ into
$(\widehat{\mathcal{P}_{p}(X)},\hat{W}_{p})$ is then induced according
to the following diagram: 
\[
\xymatrix{(\mathcal{P}(F)_{\lim}^{\omega},W_{p}^{\omega})\ar@{->>}[d]_{st_{\hat{W}_{p}}}\ar@{}[r]|-*[@]{\subseteq} & \mathcal{P}(X)_{\lim}^{\omega},W_{p}^{\omega})\ar@{->>}[d]^{st_{\hat{W}_{p}}}\\
(\widehat{\mathcal{P}_{p}(F)},\hat{W}_{p})\ar@{-->}[r] & (\widehat{\mathcal{P}_{p}(X)},\hat{W}_{p})
}
\]
\end{rem*}
Up to this point, we have used very little of the detailed structure
of the distances $W_{p}^{\omega}$ and $\hat{W}_{p}$, and this is
a good point to say a bit more. Interpreted literally, the ``fine
structure'' of the space $\widehat{\mathcal{P}_{p}(X)}$ is as follows.
First, a point in $\mathcal{P}_{p}(X)^{\omega}$ is an ultraproduct
of probability measures. In other words, if $(X_{i},d_{i})$ is equipped
with the Borel $\sigma$-algebra $\mathcal{B}_{i}$ %
, we have $\mu^{\omega}:\mathcal{B}^{\omega}\rightarrow[0,1]^{\omega}$
is an internal function which is finitely additive, but not necessarily
countably additive. Second, $\mu^{\omega}=[(\mu_{i})]\approx_{W_{p}^{\omega}}\nu^{\omega}=[(\nu_{i})]$,
and thus $\mu^{\omega}$ and $\nu^{\omega}$ get mapped to the same
point in $\widehat{\mathcal{P}_{p}(X)}$, iff there exists a sequence
of couplings $(\gamma_{i})_{i\in\mathbb{N}}$ where $\gamma_{i}\in\Pi(\mu_{i},\nu_{i})$,
such that 
\[
\left[\left(\left(\int_{X_{i}\times X_{i}}\left(d_{i}(x,y)\right)^{p}d\gamma_{i}(x,y)\right)^{1/p}\right)\right]\approx0.
\]

In fact, it will be conceptually convenient to make the following
definition.
\begin{defn}
(Space of internal couplings) Given, $\mu^{\omega},\nu^{\omega}\in\mathcal{P}_{p}(X)^{\omega}$,
define $\Pi^{\omega}(\mu^{\omega},\nu^{\omega})$, the space of \emph{internal
couplings }of $\mu^{\omega}$ and $\nu^{\omega}$, as follows: 
\[
[(\gamma_{i})]:=\gamma^{\omega}\in\Pi^{\omega}(\mu^{\omega},\nu^{\omega})\iff\gamma_{i}\in\Pi(\mu_{i,}\nu_{i})\text{ \ensuremath{\omega}-almost surely.}
\]
\end{defn}

With this definition in hand, we see that $W_{2}^{\omega}(\mu^{\omega},\nu^{\omega})\approx0$
iff there exists an internal coupling $\gamma^{\omega}\in\Pi^{\omega}(\mu^{\omega},\nu^{\omega})$
such that $\int_{X^{\omega}\times X^{\omega}}(d^{\omega}(x,y)^{p}d\gamma^{\omega}(x,y)\approx0$.

On the other hand, while a point in $\widehat{\mathcal{P}_{p}(X)}$
carries the structure of an equivalency class of ultraproducts of
measures that are all $W_{2}^{\omega}$-infinitesimally close to each
other, if we wanted to view a point in $\widehat{\mathcal{P}_{p}(X)}$
\emph{as a measure} it is not entirely obvious how to do this (except
for the portion of $\widehat{\mathcal{P}_{p}(X)}$ that we can identify
with $\mathcal{P}_{p}(\hat{X})$!). This is especially awkward since
it is desirable to pass measure-theoretic data between points in $\mathcal{P}_{p}(\hat{X})$
and their liftings in $\mathcal{P}_{p}(X)^{\omega}$ (or, what is
much the same, sequences of measures in the sequence of spaces $(\mathcal{P}_{p}(X_{i}))_{i\in\mathbb{N}}$),
and superficially, such data must ``go through'' the space $\widehat{\mathcal{P}_{p}(X)}$
(cf. Figure 2.1(b)). It turns out that, to some extent, one can use
Loeb measures as a ``workaround'', as will be explained in the remainder
of this section.

\begin{figure}
\subfloat[]%
{

\xymatrix{ & \left(X^{\omega},d^\omega\right)\ar@{}[r]|-*[@]{\supseteq} & \left(X_{\lim}^{\omega},d^{\omega}\right)\ar@{->>}[dd]^{/\approx}\\ \left(X_{i},d_i\right)_{i\in\mathbb{N}}\ar[ru]^{\omega}\ar@{-->}[rrd]  \\  &  & (\hat{X},\hat{d}) }}

\subfloat[]%
{\xymatrix{ & \left(\mathcal{P}_{p}(X)^{\omega},W_{p}^{\omega}\right)\ar@{}[r]|-*[@]{\supseteq} & \left(\mathcal{P}_{p}(X)_{\lim}^{\omega},W_{p}^{\omega}\right)\ar@{->>}[d]^{/\approx}\ar@{->>}[dr]^{L}\\ \left(\mathcal{P}_{p}(X_{i}),W_{p}\right)_{i\in\mathbb{N}}\ar[ru]^{\omega}\ar@{-->}[rrd] &  & \left(\widehat{\mathcal{P}_{p}(X)},\hat{W}_{p}\right) & \left(\mathcal{P}_{p,L}(X)_{\lim}^{\omega},W_{p,L}\right)\\  &  & \left(\mathcal{P}_{p}(\hat{X}),W_{p}\right)\ar@{^{(}->}[u]_{\iota}\ar@{-->}[ru] }

}

\caption{Cartoons in the style of commutative diagrams, depicting generic proof
strategies when working with the metric ultraproduct construction.
(A) The arrow $\stackrel{\omega}{\protect\longrightarrow}$ represents
the ultrafilter quotient map which takes a sequence $(x_{i})_{i\in\mathbb{N}}\in(X_{i},d_{i})_{i\in\mathbb{N}}$
and returns a point $x^{\omega}\in(X^{\omega},d^{\omega})$; this
arrow is meant to suggest the use of \L o\'{s}'s theorem to deduce
facts about the ultraproduct $(X^{\omega},d^{\omega})$ from the sequence
of metric spaces $(X_{i},d_{i})_{i\in\mathbb{N}}$. When the spaces
$(X_{i},d_{i})$ are moreover \emph{pointed} metric spaces, it makes
sense to talk about the \emph{limited} subset $(X_{\lim}^{\omega},d^{\omega})$
of $(X^{\omega},d^{\omega})$, as discussed in the introduction; from
$(X_{\lim}^{\omega},d^{\omega})$, we create the metric ultralimit,
$(\hat{X},\hat{d})$, by quotienting using the equivalence class $x^{\omega}\sim y^{\omega}\protect\iff d^{\omega}(x^{\omega},y^{\omega})\approx0$.
By analytic study of this quotient map, we ultimately deduce facts
about $(\hat{X},\hat{d})$; this is represented by the dashed arrow
$\protect\dashrightarrow$. (B) Here we consider ultraproducts of
Wasserstein spaces. Since Wasserstein spaces are themselves metric
spaces, the situation is in many respects similar. However, there
is an additional complication, in that, as shown in Theorem \ref{thm:isometric embedding},
the limiting space we are ultimately interested in, namely $(\mathcal{P}_{p}(\hat{X}),W_{p})$,
is not generally the metric ultralimit of $(\mathcal{P}_{p}(X_{i}),W_{p})_{i\in\mathbb{N}}$,
but rather merely embeds isometrically (as indicated by the arrow
$\stackrel{\iota}{\protect\hookrightarrow}$) into the metric ultralimit
$(\widehat{\mathcal{P}_{p}(X)},\hat{W}_{p})$. Furthermore, in the
rightmost column in the diagram we include the space $(\mathcal{P}_{p,L}(X)_{\lim}^{\omega},W_{p,L})$,
which is the space of Loeb measures associated to points in $\mathcal{P}_{p}(X)_{\lim}^{\omega}$,
with a ``distance'' $W_{p,L}$ defined in terms of couplings $\gamma_{L}$
which are themselves Loeb measures, see Definition \ref{def:omega-couplings}.
Note that the space $(\mathcal{P}_{p,L}(X)_{\lim}^{\omega},W_{p,L})$
is a space of measures with potentially pathological ``metric'',
whereas $(\widehat{\mathcal{P}_{p}(X)},\hat{W}_{p})$ is a metric
space but whose points do not generally have an obvious measure structure.
The dashed arrow from $(\mathcal{P}_{p}(\hat{X}),W_{p})$ to $(\mathcal{P}_{p,L}(X)_{\lim}^{\omega},W_{p,L})$
is justified by Theorem \ref{thm:radon representation}.}
\end{figure}
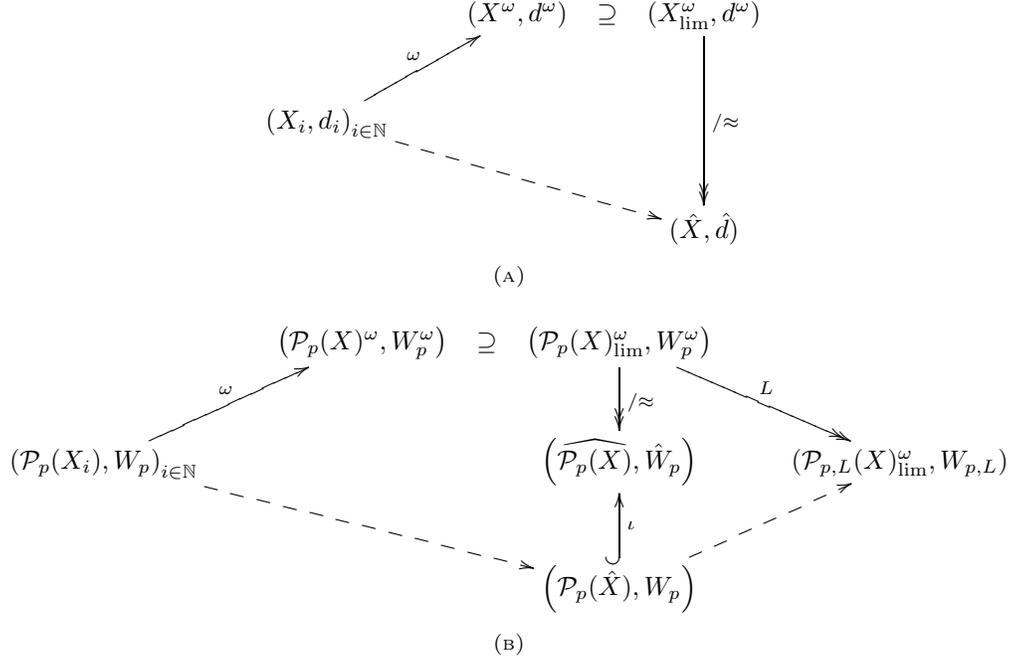

We recall the notion of a \emph{lifting} from $(\hat{X},\hat{d})$
to $(X_{\lim}^{\omega},d^{\omega})$, from the introduction. Likewise,
we say that $\mu^{\omega}\in(\mathcal{P}(X)_{\lim}^{\omega},W_{p}^{\omega})$
is a lifting of $\mu\in(\mathcal{P}_{p}(\hat{X}),W_{p})$ provided
that $\iota\mu=st_{\hat{W}_{p}}(\mu^{\omega})$, where $\iota$ is
the isometric embedding of $(\mathcal{P}_{p}(\hat{X}),W_{p})$ into
$(\widehat{\mathcal{P}_{p}(X)},\hat{W}_{p})$ and $st_{\hat{W}_{p}}(\cdot)$
is the quotient map from $(\mathcal{P}(X)_{\lim}^{\omega},W_{p}^{\omega})$
onto $(\widehat{\mathcal{P}_{p}(X)},\hat{W}_{p})$. Dually, we say
that $\mu=\iota^{-1}st_{\hat{W}_{p}}(\mu^{\omega})$ is the \emph{pushdown}
of $\mu^{\omega}$ onto $(\mathcal{P}_{p}(\hat{X}),W_{p})$ provided
$\mu^{\omega}$ is in the domain of $st_{\hat{W}_{p}}(\cdot)$ and
$st_{\hat{W}_{p}}(\mu^{\omega})$ is in the range of $\iota$. 

At the same time, given $\mu^{\omega}\in(X_{\lim}^{\omega},d^{\omega})$,
we can construct the Loeb measure associated to $\mu^{\omega}$, denoted
by $\mu_{L}$, as discussed in Section \ref{sec:Introduction}. Note
that $\mu_{L}$ is a (real-valued) probability measure on $(X_{\lim}^{\omega},d^{\omega})$.
Using the map $st_{\hat{d}}:X_{\lim}^{\omega}\rightarrow\hat{X}$,
we might also consider the pushforward of $\mu_{L}$, namely $\mu_{L}\circ st_{\hat{d}}^{-1}$,
which is a probability measure on $(\hat{X},\hat{d})$, at least \emph{provided
that $st_{\hat{d}}$ is measurable}. It has already been observed
in \cite{elek2012samplings} that as a map between $X_{\lim}^{\omega}$
(with the Loeb $\sigma$-algebra $\sigma(\mathcal{B}^{\omega})$)
and $\hat{X}$ (with the Borel $\sigma$-algebra generated by $\hat{d}$),
measurability of $st_{\hat{d}}$ can \emph{fail} when $\hat{X}$ is
not separable; consequently, checking that $st_{\hat{d}}$ is measurable
\emph{when restricted to the supports of $\mu_{L}$ and $\mu$ }is
necessary.
\begin{lem}
\label{lem:measurability of st}Suppose that $A\subset(\hat{X},\hat{d})$
is separable and closed. Let $\widehat{\mathcal{B}}$ denote the Borel
$\sigma$-algebra generated by $\hat{d}$ on $\hat{X}$. Then $st_{\hat{d}}:X_{\lim}^{\omega}\rightarrow\hat{X}$
is measurable with respect to the $\sigma$-algebras $\sigma(\mathcal{B}^{\omega})$
and $\widehat{\mathcal{B}}\cap A$.
\end{lem}

\begin{proof}
The argument is a minor extension of \cite[Lemma 3.1]{elek2012samplings}.
Indeed, let $\{x_{j}\}_{i=1}^{\infty}$ be a countable dense set in
$A$. Then for all $n\in\mathbb{N}$, $A\subseteq\bigcup_{j}B(x_{j},1/n)$,
so 
\[
A\subseteq\bigcap_{n=1}^{\infty}\bigcup_{j=1}^{\infty}B(x_{j},1/n).
\]
On the other hand, $\overline{\{x_{j}\}_{j=1}^{\infty}}\subseteq A$
since $A$ is closed. It follows %
{} that 
\[
A=\bigcap_{n=1}^{\infty}\bigcup_{j=1}^{\infty}B(x_{j},1/n).
\]
Consequently, $st^{-1}(A)\in\sigma(\mathcal{B}^{\omega})$ by \cite[Lemma 3.1]{elek2012samplings}.
Moreover, $st^{-1}(A\cap B(x,\varepsilon))\in\sigma(\mathcal{B}^{\omega})$
for every ball $B(x,\varepsilon)$ in $\hat{X}$, since balls have
 $\sigma(\mathcal{B}^{\omega})$-measurable preimages with respect
to $st(\cdot)$, again by \cite[Lemma 3.1]{elek2012samplings}. It
follows (from the $\pi$-$\lambda$ theorem) %
{} that every set in $\widehat{\mathcal{B}}\cap A$ has a $st(\cdot)$-preimage
which is in $\sigma(\mathcal{B}^{\omega})$, since $\widehat{\mathcal{B}}\cap A$
is countably generated by sets of the form $B(x_{j},1/n)\cap A$. 
\end{proof}
\begin{rem*}
Suppose $\mu$ is a Radon probability measure on $(\hat{X},\hat{d})$.
Then $supp(\mu)$ is closed and separable, so the preceding lemma
shows that we can always work with $st_{\hat{d}}^{-1}$ if we restrict
to the support of some Radon measure of interest. %
\end{rem*}
In the sequel, we will sometimes want to consider the \emph{lifting}
$\mu^{\omega}$ of a given $\mu\in(\mathcal{P}_{p}(\hat{X}),W_{p})$,
but also sometimes work with a Loeb measure $\mu_{L}$ whose $st_{\hat{d}}$-pushforward
is $\mu$; better still, if $\mu_{L}$ is the Loeb measure associated
to some $\mu^{\omega}$ which is a lifting of $\mu$. In what follows,
we show that this is all doable.
\begin{lem}
\label{lem:diagonal sequence}Suppose that $(\mu_{n})_{n\in\mathbb{N}}\in\mathcal{P}_{p}(\hat{X})$
is a $W_{p}$-Cauchy sequence with limit $\mu$. For each $\mu_{n}$,
fix a sequence of measures $(\mu_{n,i})$ in $\mathcal{P}_{p}(X_{i})$,
such that $\iota^{-1}\circ st_{\hat{W}_{p}}[(\mu_{n,i})]=\mu_{n}$.
Then the following hold:
\begin{enumerate}
\item For every $\varepsilon>0$, there exists an $N(\varepsilon)$ such
that for all $n,m\geq N(\varepsilon)$, and $\omega$-almost all $i\in\mathbb{N}$,
$W_{p,i}(\mu_{n,i},\mu_{m,i})<\varepsilon.$ In fact $N(\varepsilon)$
is the modulus of convergence of $(\mu_{n})$ in $(\mathcal{P}_{p}(\hat{X}),W_{2})$.
\item (Diagonal sequence) Let $N(\varepsilon)$ be the modulus of convergence
from (1). Then, $[(\mu_{N(1/i),i})]\in\mathcal{P}_{p}(X)^{\omega}$
is a lifting of $\mu$. Moreover, so is any sequence $[(\mu_{\Phi(i),i})]$
where $\Phi(i)\geq N(1/i)$ for $\omega$-almost all $i\in\mathbb{N}$.
\end{enumerate}
\end{lem}

\begin{proof}
(1) First, it is clear from construction that $W_{p}(\mu_{n},\mu_{m})<\varepsilon$
iff the same holds for any liftings $\mu_{n}^{\omega}=[(\mu_{n,i})]$
and $\mu_{m}^{\omega}=[(\mu_{m,i})]$ in $\mathcal{P}_{p}(X)^{\omega}$.
{} And this in turn, holds iff (1) holds. 

(2) In what follows, we take the function $N(\varepsilon)$ from (1)
to be nondecreasing as $\varepsilon\rightarrow0$, without loss of
generality. 

Consider the diagonal sequence $(\mu_{N(1/i),i})$ of measures. Fix
$\varepsilon>0$ arbitrary. Let $i_{0}$ be sufficiently large that
$\varepsilon>1/i_{0}$. Then, note that if $i\geq i_{0}$, it holds
that if $n,m\geq N(1/i)$, then also $n,m\geq N(1/i_{0})$. So since
\[
\{i\in\mathbb{N}\mid\forall n,m\geq N(1/i_{0}),W_{p,i}(\mu_{n,i},\mu_{m,i})<1/i_{0}\}\in\omega
\]
it holds that 
\[
\{i\in\mathbb{N}\mid i\geq i_{0}\}\cap\{i\in\mathbb{N}\mid\forall n,m\geq N(1/i_{0}),W_{p,i}(\mu_{n,i},\mu_{m,i})<1/i_{0}\}\in\omega
\]
and since 
\begin{multline*}
\{i\in\mathbb{N}\mid i\geq i_{0}\wedge\forall n,m\geq N(1/i_{0}),W_{p,i}(\mu_{n,i},\mu_{m,i})<1/i_{0}\}\subseteq\\
\{i\in\mathbb{N}\mid i\geq i_{0}\wedge\forall n\geq N(1/i_{0})\forall m\geq N(1/i),W_{p,i}(\mu_{n,i},\mu_{m,i})<1/i_{0}\}
\end{multline*}
it also holds, for $\omega$-almost all $i\in\mathbb{N}$, that $W_{p,i}(\mu_{n,i},\mu_{m,i})<1/i_{0}$
whenever $n\geq N(1/i_{0})$ and $m\geq N(1/i)$. In particular, it
holds $\omega$-almost surely that 
\[
\forall n\geq N(1/i_{0}),W_{p,i}(\mu_{n,i},\mu_{N(1/i),i})<1/i_{0}.
\]
Consequently, by construction, we have that 
\[
\forall n\geq N(1/i_{0}),W_{p}^{\omega}([(\mu_{n,i})],[(\mu_{N(1/i),i})])<1/i_{0}.
\]
Now, let $\mu$ be the limit of $(\mu_{n})$ in $(\mathcal{P}_{p}(\hat{X}),W_{p})$,
{} and let $[(\mu_{-,i})]\in\mathcal{P}(X)^{\omega}$ be any lifting
of $\mu$. Note that since $N(\varepsilon)$ is the modulus of convergence
for $(\mu_{n})$, we have that 
\[
\forall n\geq N(1/i_{0}),W_{p}^{\omega}([(\mu_{n,i})],[(\mu_{-,i})])<1/i_{0}.
\]
Consequently, by the triangle inequality in $(\mathcal{P}_{p}(X)^{\omega},W_{p}^{\omega})$,
\[
W_{p}^{\omega}([(\mu_{-,i})],[(\mu_{N(1/i),i})])<2/i_{0}.
\]
But our choice of $\varepsilon>0$ (and thus $i_{0}$) was arbitrary,
hence $W_{p}^{\omega}([(\mu_{-,i})],[(\mu_{N(1/i),i})])$ is smaller
than every positive standard real, hence $W_{p}^{\omega}([(\mu_{-,i})],[(\mu_{N(1/i),i})])\approx0$.
Thus, $[(\mu_{N(1/i),i)})]$ is \emph{also} a lifting of $\mu$. 

The exact same reasoning also allows us to pick a different $n_{i}\geq N(1/i)$
for each $i$, so the remaining claim in the statement of (2) holds.
\end{proof}
\begin{thm}
\label{thm:radon representation}Let $p\in[1,\infty)$. Let $\mu$
be a Radon measure on $\hat{X}$, with $\mu\in\mathcal{P}_{p}(\hat{X})$.
Then there exists a Loeb measure $\mu_{L}$ on $X_{\lim}^{\omega}$
whose $st_{\hat{d}}$-pushforward is $\mu$. Moreover, we can take
$\mu_{L}$ to be the Loeb measure associated to a lifting $\mu^{\omega}$
of $\mu$ which is hyperfinite (i.e. an ultraproduct of discrete measures).
\end{thm}

\begin{proof}
Let $(\mu_{n})$ be a sequence of discrete probability measures converging
in $W_{p}$ to $\mu\in\mathcal{P}_{p}(\hat{X})$, such that $supp(\mu_{n})\subseteq supp(\mu)$,
and such that each $\mu_{n}$ has the form
\[
\mu_{n}=\frac{1}{k_{n}}\sum_{j=1}^{k_{n}}c_{j,n}\delta_{\hat{x}_{j,n}};\quad\hat{x}_{j,n}\in supp(\mu).
\]
For example, one might take $\mu_{n}$ to be the $n$th empirical
measure for $\mu$ (in this case $c_{j,n}=1$). We let $N(\varepsilon)$
denote the modulus of convergence for $(\mu_{n})_{n\in\mathbb{N}}$
in $(\mathcal{P}_{p}(\hat{X}),W_{p})$. For each $\hat{x}_{j,n}$,
pick a sequence $(x_{j,n,i})$ of points in $(X_{i},d_{i})$ so that
$st_{\hat{d}}[(x_{j,n,i})]=\hat{x}_{j,n}$. Define the lifting of
$\mu_{n}$ to an internal measure supported within $X_{\lim}^{\omega}$
like so: 
\[
\tilde{\mu}_{n}=\frac{1}{k_{n}}\sum_{j=1}^{k_{n}}c_{j,n}\delta_{[(x_{j,n,i})]}:=\left[\left(\frac{1}{k_{n}}\sum_{j=1}^{k_{n}}c_{j,n}\delta_{x_{j,n,i}}\right)_{i\in\mathbb{N}}\right]:=[(\tilde{\mu}_{n,i})].
\]
Let $\mathcal{B}^{\omega}$ denote the internal measure algebra of
internal Borel sets. For trivial reasons, $\tilde{\mu}_{n}$ extends
directly to $\sigma(\mathcal{B}^{\omega})$ from $\mathcal{B}^{\omega}$,
and only assigns standard real measures to sets, so $\tilde{\mu}_{n}$
may be easily confused with its associated Loeb measure $\tilde{\mu}_{n,L}$
(but we resist abusive identification). 

Let $B$ be any Borel set in $\hat{X}\cap supp(\mu)$. It is clear
that $\mu_{n}(B)=\tilde{\mu}_{n,L}(st^{-1}(B))$, simply because the
locations of the constituent atoms of $\tilde{\mu}_{n}$ (and thus
$\tilde{\mu}_{n,L}$) correspond exactly to those of $\mu_{n}$, up
to an arbitrary selection of a point in each pre-image $st_{\hat{d}}^{-1}\{\hat{x}_{j,n}\}$.
At the same time, $\mu_{n}(B)\rightarrow\mu(B)$ for every continuity
set $B\subset\hat{X}\cap supp(\mu)$ (that is, a Borel set in $\hat{X}\cap supp(\mu)$
with $\mu(\partial B)=0$); this follows from the fact that the topology
induced by $W_{p}$ is stronger than that of the weak convergence
of probability measures. In particular, for each such $B$, the sequence
$\{\tilde{\mu}_{n,L}(st^{-1}(B))\}_{n\in\mathbb{N}}$ is Cauchy, since
$\mu_{n}(B)=\tilde{\mu}_{n,L}(st^{-1}(B))$ for each $n\in\mathbb{N}$. 

It would be nice to proceed as follows: passing to a nonstandard extension
of the sequence of objects $\tilde{\mu}_{n}$, we find that $\tilde{\mu}_{N}(st^{-1}(B))\approx\tilde{\mu}_{M}(st^{-1}(B))$
for all strictly hyperfinite $N,M$, and moreover $\tilde{\mu}_{N}(st^{-1}(B))\approx\mu(B)$;
and thus it suffices to take the Loeb measure associated to any such
internal measure $\tilde{\mu}_{N}$ as a representative of $\mu$.
However, we have not developed the theory of nonstandard extensions
in this article, and so in the interest of self-containment, we instead
reason by way of the diagonal lifting construction from the previous
lemma.

To wit, let $B(x_{0},\delta)\cap supp(\mu)$ be an open ball in the
restriction of $(\hat{X},\hat{d})$ to the support of $\mu$ which
is a continuity set for $\mu$ (so in particular, $\mu(\partial B(x_{0},\delta))=0$).
Let 
\[
\Psi_{(x_{0},\delta)}(\varepsilon):=\underset{N\in\mathbb{N}}{\text{argmin}}\left\{ (\forall n,m\geq N)|\mu_{n}(B(x_{0},\delta)\cap supp(\mu))-\mu_{m}(B(x_{0},\delta)\cap supp(\mu))|<\varepsilon\right\} .
\]
In other words, $\Psi_{(x_{0},\delta)}(\varepsilon)$ is a modulus
of convergence for the sequence $(\mu_{n}(B(x_{0},\delta)\cap supp(\mu)))_{n\in\mathbb{N}}$.
From part (2) of Lemma \ref{lem:diagonal sequence}, we know that
if $\Phi:\mathbb{N}\rightarrow\mathbb{N}$ is greater, $\omega$-a.s.,
than $(\max\{N(1/i),\Psi_{(x_{0},\delta)}(1/i)\})_{i\in\mathbb{N}}$,
it then it holds that $[(\mu_{\Phi(i)})]$ is a lifting of $\mu$.
At the same time, letting $\mu_{\Phi,L}:=\text{Loeb}[(\mu_{\Phi})]$,
we claim that $\mu_{\Phi,L}(st_{\hat{d}}^{-1}(B(x_{0},\delta)\cap supp(\mu)))=\mu(B(x_{0},\delta)\cap supp(\mu))$. 

Let us demonstrate why the claim holds. Fix $\varepsilon>0$. Then,
if $n,m\geq\Psi_{(x_{0},\delta)}(\varepsilon)$, it holds that 
\[
|\mu_{n}(B(x_{0},\delta)\cap supp(\mu))-\mu_{m}(B(x_{0},\delta)\cap supp(\mu))|<\varepsilon
\]
and hence 
\[
|\tilde{\mu}_{n,L}(st_{\hat{d}}^{-1}(B(x_{0},\delta)\cap supp(\mu)))-\tilde{\mu}_{m,L}(st_{\hat{d}}^{-1}(B(x_{0},\delta)\cap supp(\mu)))|<\varepsilon.
\]
Since for each $n\in\mathbb{N}$, $\mu_{n}$ is supported on a finite
number of points, it follows that there is some minimum distance between
the boundary $\partial(B(x_{0},\delta)\cap supp(\mu))$ and any atom
in $\mu_{n}$. Consequently, if $\eta:=[(\eta_{i})]$ is any positive
infinitesimal, then (since $st_{\hat{d}}^{-1}(B(x_{0},\delta)\cap supp(\mu))=st_{\hat{d}}^{-1}(B(x_{0},\delta))\cap st_{\hat{d}}^{-1}supp(\mu)$)
\[
\tilde{\mu}_{n,L}(st_{\hat{d}}^{-1}(B(x_{0},\delta)\cap supp(\mu)))=\tilde{\mu}_{n,L}(B(x_{0}^{\omega},\delta+\eta)\cap st_{\hat{d}}^{-1}supp(\mu));
\]
\[
\tilde{\mu}_{m,L}(st_{\hat{d}}^{-1}(B(x_{0},\delta)\cap supp(\mu)))=\tilde{\mu}_{m,L}(B(x_{0}^{\omega},\delta+\eta)\cap st_{\hat{d}}^{-1}(supp(\mu))
\]
where $x_{0}^{\omega}$ is any lifting of $x_{0}$, and 
\[
B(x_{0}^{\omega},\delta+\eta):=\{x^{\omega}\in X^{\omega}\mid d^{\omega}(x_{0}^{\omega},x^{\omega})<\delta+\eta\}.
\]
Note that $B(x_{0}^{\omega},\delta+\eta)$ is the ultraproduct of
the sequence of open balls $B(x_{0,i},\delta+\eta_{i})\subseteq X_{i}$. 

Let $S^{\omega}:=[(S_{i})]$ be any internal measurable set containing
$st_{\hat{d}}^{-1}supp(\mu)$. We deduce (using the fact that $\tilde{\mu}_{n}$
has no mass outside of $st_{\hat{d}}^{-1}supp(\mu)$) that
\[
(\forall n,m\geq\Phi_{(x_{0},\delta)}(\varepsilon))|\tilde{\mu}_{n}(B(x_{0}^{\omega},\delta+\eta)\cap S^{\omega})-\tilde{\mu}_{m}(B(x_{0}^{\omega},\delta+\eta)\cap S^{\omega})|<\varepsilon.
\]
In turn, this holds iff 
\[
\{i\in\mathbb{N}\mid(\forall n,m\geq\Phi_{(x_{0},\delta)}(\varepsilon))|\tilde{\mu}_{n,i}(B(x_{0,i},\delta+\eta_{i})\cap S_{i})-\tilde{\mu}_{m,i}(B(x_{0,i},\delta+\eta_{i})\cap S_{i})|<\varepsilon\}\in\omega.
\]
It follows, by the same reasoning as in the proof of part 2 of Lemma
\ref{lem:diagonal sequence}, that if $\Phi(i)\geq\Phi_{(x_{0},\delta)}(1/i)$,
it holds that 
\[
\{i\in\mathbb{N}\mid n\geq\Phi_{(x_{0},\delta)}(\varepsilon)|\tilde{\mu}_{\Phi(i),i}(B(x_{0,i},\delta+\eta_{i})\cap S_{i})-\tilde{\mu}_{n,i}(B(x_{0,i},\delta+\eta_{i})\cap S_{i})|<\varepsilon\}\in\omega
\]
and therefore, 
\[
|\tilde{\mu}_{\Phi}(B(x_{0}^{\omega},\delta+\eta)\cap S^{\omega})-\tilde{\mu}_{n}(B(x_{0}^{\omega},\delta+\eta)\cap S^{\omega})|<\varepsilon.
\]
It follows that 
\[
|\tilde{\mu}_{\Phi}(B(x_{0}^{\omega},\delta+\eta)\cap S^{\omega})-\tilde{\mu}_{n,L}(B(x_{0}^{\omega},\delta+\eta)\cap S^{\omega})|<\varepsilon
\]
and hence 
\[
|\tilde{\mu}_{\Phi}(B(x_{0}^{\omega},\delta+\eta)\cap S^{\omega})-\tilde{\mu}_{n,L}(st_{\hat{d}}^{-1}(B(x_{0},\delta)\cap supp(\mu)))|<\varepsilon.
\]
Sending $\varepsilon\rightarrow0$, we deduce that for any infinitesimal
$\eta>0$, 
\[
\tilde{\mu}_{\Phi}(B(x_{0}^{\omega},\delta+\eta)\cap S^{\omega})\approx\lim_{n\rightarrow\infty}\tilde{\mu}_{n,L}(st_{\hat{d}}^{-1}(B(x_{0},\delta)\cap supp(\mu))).
\]

Now, we can consider $\tilde{\mu}_{\Phi}(B(x_{0}^{\omega},\delta+\eta)\cap S^{\omega})$
as an internal function of $\eta$, that is, from $\mathbb{R}_{+}^{\omega}$
to $\mathbb{R}_{+}^{\omega}$; it follows from the overspill principle
\cite[Chapter 11; viz. Theorem 11.9.1]{goldblatt2012lectures} that
for any $\varepsilon_{0}\in\mathbb{R}_{+}$, there exists some $\eta_{0}\in\mathbb{R}_{+}$
such that for all positive hyperreal $\eta<\eta_{0}$, 
\[
|\tilde{\mu}_{\Phi}(B(x_{0}^{\omega},\delta+\eta)\cap S^{\omega})-\lim_{n\rightarrow\infty}\tilde{\mu}_{n,L}(st_{\hat{d}}^{-1}(B(x_{0},\delta)\cap supp(\mu)))|<\varepsilon_{0}
\]
as well, and hence 
\[
|\tilde{\mu}_{\Phi,L}(B(x_{0}^{\omega},\delta+\eta)\cap S^{\omega})-\lim_{n\rightarrow\infty}\tilde{\mu}_{n,L}(st_{\hat{d}}^{-1}(B(x_{0},\delta)\cap supp(\mu)))|<\varepsilon_{0}.
\]
Pick $\eta_{1}\approx0$ and $\eta_{2}=\frac{\eta_{0}}{2}$; using
the fact that 
\[
\left(B(x_{0}^{\omega},\delta+\eta_{1})\cap S^{\omega}\right)\subset\left(st_{\hat{d}}^{-1}(\bar{B}(x_{0},\delta))\cap S^{\omega}\right)\subset\left(B(x_{0}^{\omega},\delta+\eta_{2})\cap S^{\omega}\right)
\]
we deduce that 
\[
|\tilde{\mu}_{\Phi,L}(st_{\hat{d}}^{-1}(\bar{B}(x_{0},\delta))\cap S^{\omega})-\lim_{n\rightarrow\infty}\tilde{\mu}_{n,L}(st_{\hat{d}}^{-1}(B(x_{0},\delta)\cap supp(\mu)))|<\varepsilon_{0}.
\]
Since $\varepsilon_{0}>0$ was arbitrary, we deduce that 
\[
\tilde{\mu}_{\Phi,L}(st_{\hat{d}}^{-1}(\bar{B}(x_{0},\delta))\cap S^{\omega})=\lim_{n\rightarrow\infty}\tilde{\mu}_{n,L}(st_{\hat{d}}^{-1}(B(x_{0},\delta)\cap supp(\mu)))
\]
But since $\mu(\delta B(x_{0},\delta))=0$, it holds that 
\[
\lim_{n\rightarrow\infty}\tilde{\mu}_{n,L}(st_{\hat{d}}^{-1}(B(x_{0},\delta)\cap supp(\mu)))=\lim_{n\rightarrow\infty}\tilde{\mu}_{n,L}(st_{\hat{d}}^{-1}(\bar{B}(x_{0},\delta)\cap supp(\mu))).
\]
Finally, we select $S^{\omega}$ so that $\tilde{\mu}_{\Phi,L}(st_{\hat{d}}^{-1}supp(\mu)\Delta S^{\omega})=0$.
Hence, 
\begin{align*}
\tilde{\mu}_{\Phi,L}(st_{\hat{d}}^{-1}(\bar{B}(x_{0},\delta))\cap S^{\omega}) & =\tilde{\mu}_{\Phi,L}(st_{\hat{d}}^{-1}(\bar{B}(x_{0},\delta))\cap st_{\hat{d}}^{-1}supp(\mu))\\
 & =\tilde{\mu}_{\Phi,L}(st_{\hat{d}}^{-1}(\bar{B}(x_{0},\delta)\cap supp(\mu)))
\end{align*}
and so $\tilde{\mu}_{\Phi,L}\circ st_{\hat{d}}^{-1}(B(x_{0},\delta))=\mu(B(x_{0},\delta))$,
and the claim is proved.

We finish the proof by considering a countable family of continuity
sets for $\mu$ of the form $B(x_{0},\delta)\cap supp(\mu)$ which
generate the $\sigma$-algebra of Borel sets in the support of $\mu$;
such a countable family exists since $\mu$ has separable support.
Selecting a function $\tilde{\Phi}(i)$ which is $\omega$-almost
surely larger than $N(\varepsilon)$ as well as the convergence modulus
$\Phi_{B(x_{0},\delta)\cap supp(\mu)}$ for \emph{every} $B(x_{0},\delta)$
in the countable family (for example, take a diagonalization of all
these convergence moduli), we conclude that $\tilde{\mu}_{\tilde{\Phi},L}\circ st_{\hat{d}}^{-1}=\mu$. 
\end{proof}
In fact, it is possible to modify the preceding argument to get a
superficially stronger result, namely: 
\begin{cor}
Let $\mu$ be an $\sigma$-finite Radon measure on $\hat{X}$ (that
is, $\mu$ can be written as a countable sum of Radon probability
measures). Then, there exists an internal measure $\mu^{\omega}$
on $X^{\omega}$ whose Loeb measure $\mu_{L}$ pushes onto $\mu$. 
\end{cor}

\begin{proof}
(sketch) To drop the assumption that $\mu\in\mathcal{P}_{p}(\hat{X})$,
we simply replace the metric structure of $W_{p}$ with one that is
well-defined on all of $\mathcal{P}(\hat{X})$ and has the property
that discrete measures are dense, such as the Lévy-Prokhorov metric
$d_{LP}$; that is, one considers the spaces $(\mathcal{P}(\hat{X}),d_{LP})$
and $(\mathcal{P}(X)^{\omega},d_{LP}^{\omega})$. The argument from
the preceding lemma and theorem then goes through otherwise unchanged,
and allows us to deduce a version of the previous theorem for any
Radon probability measure on $\hat{X}$. The case of $\sigma$-finite
Radon measures on $\hat{X}$ then follows immediately by globalization. 
\end{proof}
{} %

At this point, is worth briefly touching on ``what goes wrong''
for the space $(\widehat{\mathcal{P}_{p}(X)},\widehat{W}_{p})$ to
be strictly bigger than $(\mathcal{P}_{p}(\hat{X}),W_{p})$ when $\hat{X}$
is not compact, since we are able to do so now. Consider, in particular,
the case where $(X_{i},d_{i})$ is a constant sequence where each
$(X_{i},d_{i})$ is $\mathbb{R}^{d}$ with the Euclidean norm. The
\emph{concentration-compactness }phenomenon is that, if we have a
sequence of (Radon) measures on $\mathbb{R}^{d}$, one of three things
can happen (up to a subsequence): the sequence is \emph{tight}, the
sequence \emph{diffuses towards infinity}, or only a \emph{fraction}
of the mass is tight and the remainder of the mass has more complicated
behavior. %
{} In the case where a sequence of measures $(\mu_{i})$ is asymptotically
diffuse, or even a finite fraction of the mass diffuses towards infinity,,
we have, in the ultraproduct, that $W_{p}^{\omega}(\delta_{e}^{\omega},\mu^{\omega})$
is a strictly hyperfinite hyperreal quantity (because the $p$th moment
is a strictly hyperreal quantity), and so $(\mu_{i}^{\omega})$ is
excluded from the domain when we construct $\widehat{\mathcal{P}_{p}(X)}$.
{} However, even when $W_{p}^{\omega}(\delta_{e_{i}}^{\omega},\mu_{i}^{\omega})$
is finite, it is still possible that $st_{\hat{W}_{p}}(\mu^{\omega})\in\widehat{\mathcal{P}_{p}(X)}$
does not lie in (the embedded image of) $(\mathcal{P}_{p}(\hat{X}),W_{p})$,
as the next example illustrates. %

\begin{example}
\label{exa:bad two point ultra-measure}(Prototypical instance showing
that $(\mathcal{P}_{2}(\hat{X}),W_{2})\subsetneq(\widehat{\mathcal{P}_{2}(X)},\hat{W}_{2})$).
Consider the case where, in the construction of the ultralimit $(\hat{X},\hat{d})$,
we take $X_{i}=\mathbb{R}$ for all $i\in\mathbb{N}$, with the usual
metric. Note that in this situation, $\hat{X}=\mathbb{R}$. Obviously
part (3) of the preceding theorem does not apply to this case since
$\mathbb{R}$ is not compact. We consider the (ultraproducts of) measures
$\delta_{0}$ and $(1-N^{-1})\delta_{0}+N^{-1}\delta_{N^{1/2}}$ in
$\mathcal{P}(\mathbb{R})^{\omega}$, with $N=[(N_{i})]\in\mathbb{N}^{\omega}\backslash\mathbb{N}$.
Then,
\[
W_{2}^{\omega}(\delta_{0},(1-N^{-1})\delta_{0}+N^{-1}\delta_{N^{1/2}})=\left(N^{-1}d^{2}(0,N^{1/2})\right)^{1/2}=1.
\]
More generally, the measure $(1-N^{-1})\delta_{0}+N^{-1}\delta_{N^{1/2}}$
is a finite distance away from \emph{any }measure in $\mathcal{P}(\mathbb{R})^{\omega}$
which is supported inside $\mathbb{R}_{\lim}^{\omega}$. Indeed, let
$\mu^{\omega}=[(\mu_{i})]$ be supported within $[-R,R]$ where $R\in\mathbb{R}$;
note that in this case, $\mu_{i}$ is supported within $[-R,R]$ $\omega$-almost
surely. Then, by definition.
\[
W_{2}^{\omega}(\mu^{\omega},(1-N^{-1})\delta_{0}+N^{-1}\delta_{N^{1/2}})=\left[\left(\left(\inf_{\gamma_{i}\in\Pi(\mu_{i},(1-N_{i}^{-1})\delta_{0}+N_{i}^{-1}\delta_{N_{i}^{1/2}}}\int|x-y|^{2}d\gamma_{i}(x,y)\right)^{1/2}\right)_{i\in\mathbb{N}}\right].
\]
Observe that for each $i\in\mathbb{N}$, the atom $N_{i}^{-1}\delta_{N_{i}^{1/2}}$
must be transported into the region $[-R,R]$, at cost bounded between
$N_{i}^{-1/2}|N_{i}^{1/2}+R|$ and $N_{i}^{-1/2}|N_{i}^{1/2}-R|$.
But $|N_{i}|\rightarrow\infty$, so in particular, 
\[
W_{2}^{\omega}(\mu^{\omega},(1-N^{-1})\delta_{0}+N^{-1}\delta_{N^{1/2}})\geq\left[\left(\min\left\{ N_{i}^{-1/2}|N_{i}^{1/2}+R|,N_{i}^{-1/2}|N_{i}^{1/2}-R|\right\} \right)_{i\in\mathbb{N}}\right]\approx1.
\]

On the other hand, since $W_{2}^{\omega}(\delta_{0},(1-N^{-1})\delta_{0}+N^{-1}\delta_{N^{1/2}})=1$,
we know that $(1-N^{-1})\delta_{0}+N^{-1}\delta_{N^{1/2}}\in(\mathcal{P}(\mathbb{R})_{\lim}^{\omega},W_{2}^{\omega})$%
. At the same time, given any $\nu\in(\mathcal{P}_{2}(\mathbb{R}),W_{2})$
with compact support (say inside $[-R,R]$), we can then consider
a lifting $\nu^{\omega}$ of $\nu$ which (in accordance with Proposition
\ref{prop:compact lifting}) is supported inside a totally bounded
set which pushes down to $[-R,R]$, hence also has bounded diameter.
This implies that
\[
W_{2}^{\omega}(\nu^{\omega},(1-N^{-1})\delta_{0}+N^{-1}\delta_{N^{1/2}})\gtrapprox1.
\]
Hence, (letting $\hat{\nu}=\iota\nu=st_{\hat{W}_{2}}(\nu^{\omega})$)
\[
\hat{W}_{2}(\hat{\nu},\widehat{(1-N^{-1})\delta_{0}+N^{-1}\delta_{N^{1/2}}})\geq1.
\]
 Since measures with compact support are dense inside $(\mathcal{P}_{2}(\mathbb{R}),W_{2})$,
it follows, by density, that for any $\nu$ in $(\mathcal{P}_{2}(\mathbb{R}),W_{2})$,
\[
\hat{W}_{2}(\hat{\nu},\widehat{(1-N^{-1})\delta_{0}+N^{-1}\delta_{N^{1/2}}})\geq1
\]
as well.
\end{example}

\begin{rem*}
There are two other reasons why it may be the case that $\hat{\mu}\in\widehat{\mathcal{P}_{p}(X)}\backslash\iota\mathcal{P}_{p}(\hat{X})$.
The first is that, given a $\mu^{\omega}\in\mathcal{P}_{p}(X)^{\omega}$
and its associated Loeb measure $\mu_{L}$, then the $st_{\hat{d}}$-pushforward
of $\mu_{L}$, because of failure of the measurability of $st_{\hat{d}}$
on the support of $\mu_{L}$. If this occurs for an \emph{entire}
equivalency class $\hat{\mu}$ in $\widehat{\mathcal{P}_{p}(X)}$,
then no element in that equivalency class is a lifting of a measure
$\mu\in\mathcal{P}_{p}(\hat{X})$, so in particular $\hat{\mu}\notin\iota\mathcal{P}_{p}(\hat{X})$.
In principle, it is also possible that we \emph{do} have $st_{\hat{d}}$-measurability
without separable support on $\hat{X}$; in this case, $\mu_{L}$
still pushes forward to a measure on $\hat{X}$ but since the support
is not separable, $\mu$ is not Radon, hence excluded from $\mathcal{P}_{p}(\hat{X})$
also. (Note however that this occurrence is independent of ZFC! Indeed,
the existence of a non-Radon probability measure on a metric space
require a large cardinal axiom, as discussed above.)
\end{rem*}
\begin{defn}
\label{def:omega-couplings}Let $\mu^{\omega},\nu^{\omega}\in\mathcal{P}_{p}(X)_{\lim}^{\omega}$,
and let $\mu_{L}$ and $\nu_{L}$ be their associated Loeb measures.
Let 
\begin{multline*}
\Pi_{L}(\mu_{L},\nu_{L}):=\{\gamma_{L}\in\mathcal{P}(X_{\lim}^{\omega}\times X_{\lim}^{\omega})\mid\\
\exists\mu^{\prime\omega},\nu^{\prime\omega}\in\mathcal{P}_{p}(X)^{\omega}\exists\gamma^{\omega}\in\Pi^{\omega}(\mu^{\prime\omega},\nu^{\prime\omega})\left[\gamma_{L}=\text{Loeb}(\gamma^{\omega})\wedge\mu_{L}^{\prime}=\mu_{L}\wedge\nu_{L}^{\prime}=\nu_{L}\right]\}.
\end{multline*}
In other words, $\Pi_{L}(\mu_{L},\nu_{L})$ is the space of couplings
between $\mu_{L}$ and $\nu_{L}$ that arise as the Loeb measures
of internal couplings between internal measures $\mu^{\prime\omega}$
and $\nu^{\prime\omega}$ with the same associated Loeb measures as
$\mu^{\omega}$ and $\nu^{\omega}$. Likewise, define 
\[
W_{p,L}(\mu_{L},\nu_{L}):=\left(\inf_{\gamma_{L}\in\Pi_{L}(\mu_{L},\nu_{L})}\int_{X_{\lim}^{\omega}\times X_{\lim}^{\omega}}st\left(\left(d^{\omega}(x^{\omega},y^{\omega})\right)^{p}\right)d\gamma_{L}(x^{\omega},y^{\omega})\right)^{1/p}.
\]
\end{defn}

\begin{rem*}
Let $B^{\omega}\in\mathcal{B}^{\omega}$ be any internal Borel set
on $X^{\omega}$ contained within $X_{\lim}^{\omega}$. Then, $1_{B^{\omega}\times X^{\omega}}$
is an $S$-integrable function w.r.t. the measure $\gamma^{\omega}$.
Hence,
\[
\int1_{B^{\omega}\times X^{\omega}}d\gamma_{L}=st\left(\int1_{B^{\omega}\times X^{\omega}}d\gamma^{\omega}\right)=st\left(\mu^{\omega}(B^{\omega})\right)=\mu_{L}(B^{\omega}).
\]
It follows that the first marginal of $\gamma_{L}$ agrees with $\mu_{L}$
on all internally measurable sets, hence on all of the Loeb measurable
sets. The same reasoning applies to the second marginal also. Hence,
$\gamma_{L}$ is indeed a coupling between $\mu_{L}$ and $\nu_{L}$;
in particular, if $\gamma^{\omega}\in\Pi^{\omega}(\mu^{\omega},\nu^{\omega})$
then also $\gamma_{L}\in\Pi_{L}(\mu_{L},\nu_{L})$.
\end{rem*}

\subsubsection*{Warning. }

The space of Loeb probability measures on $X_{\lim}^{\omega}$, equipped
with $W_{p,L}$, is not obviously a well-behaved space. For one thing,
$W_{p,L}$ does not separate points: if we pick two Dirac measures
whose atoms are infinitesimally far apart, these are two different
Loeb measures but their $W_{p,L}$ distance is zero. Yet, the space
is certainly not just a ``pseudometric analogue'' of the space $(\mathcal{P}_{p}(\hat{X}),W_{p})$,
since there may be Loeb measures on $X_{\lim}^{\omega}$ that do not
push forward to $\hat{X}$. For this reason, we will use $W_{p,L}$
only as an intermediate calculating device, and avoid working with
it directly.
\begin{prop}
\label{prop:Wp Loeb contraction 1}With the same notation as the previous
definition, 
\[
W_{p,L}(\mu_{L},\nu_{L})\lessapprox W_{p}^{\omega}(\mu^{\omega},\nu^{\omega}).
\]
\end{prop}

\begin{proof}
Let $\gamma^{\omega}\in\Pi^{\omega}(\mu^{\omega},\nu^{\omega})$.
Since the function $\left(d^{\omega}(x^{\omega},y^{\omega})\right)^{p}$
is $S$-integrable on $X_{\lim}^{\omega}$, it follows that 
\[
st\left(\int_{X^{\omega}\times X^{\omega}}\left(d^{\omega}(x^{\omega},y^{\omega})\right)^{p}d\gamma^{\omega}(x^{\omega},y^{\omega})\right)=\int_{X_{\lim}^{\omega}\times X_{\lim}^{\omega}}st\left(\left(d^{\omega}(x^{\omega},y^{\omega})\right)^{p}\right)d\gamma_{L}(x^{\omega},y^{\omega}).
\]
By quantifying over all $\gamma^{\omega}\in\Pi^{\omega}(\mu^{\omega},\nu^{\omega})$,
we see that 
\begin{multline*}
\inf_{\gamma^{\omega}\in\Pi^{\omega}(\mu^{\omega},\nu^{\omega})}st\left(\int_{X^{\omega}\times X^{\omega}}\left(d^{\omega}(x^{\omega},y^{\omega})\right)^{p}d\gamma^{\omega}(x^{\omega},y^{\omega})\right)\\
=\inf_{\gamma_{L}:\gamma^{\omega}\in\Pi^{\omega}(\mu^{\omega},\nu^{\omega})}\int_{X_{\lim}^{\omega}\times X_{\lim}^{\omega}}st\left(\left(d^{\omega}(x^{\omega},y^{\omega})\right)^{p}\right)d\gamma_{L}(x^{\omega},y^{\omega}).
\end{multline*}
Since (from the preceding remark) $\{\gamma_{L}:\gamma^{\omega}\in\Pi^{\omega}(\mu^{\omega},\nu^{\omega})\}\subseteq\Pi_{L}(\mu_{L},\nu_{L})$,
it follows that 
\begin{multline*}
\inf_{\gamma^{\omega}\in\Pi^{\omega}(\mu^{\omega},\nu^{\omega})}st\left(\int_{X^{\omega}\times X^{\omega}}\left(d^{\omega}(x^{\omega},y^{\omega})\right)^{p}d\gamma^{\omega}(x^{\omega},y^{\omega})\right)\\
\geq\inf_{\gamma_{L}\in\Pi_{L}(\mu_{L},\nu_{L})}\int_{X_{\lim}^{\omega}\times X_{\lim}^{\omega}}st\left(\left(d^{\omega}(x^{\omega},y^{\omega})\right)^{p}\right)d\gamma_{L}(x^{\omega},y^{\omega}):=W_{p,L}^{p}(\mu_{L},\nu_{L}).
\end{multline*}
On the other hand, for each $\gamma^{\omega}\in\Pi^{\omega}(\mu^{\omega},\nu^{\omega})$,
it holds (by definition of $st$) that 
\[
st\left(\int_{X^{\omega}\times X^{\omega}}\left(d^{\omega}(x^{\omega},y^{\omega})\right)^{p}d\gamma^{\omega}(x^{\omega},y^{\omega})\right)\approx\int_{X^{\omega}\times X^{\omega}}\left(d^{\omega}(x^{\omega},y^{\omega})\right)^{p}d\gamma^{\omega}(x^{\omega},y^{\omega})
\]
which implies that that 
\begin{align*}
\inf_{\gamma^{\omega}\in\Pi^{\omega}(\mu^{\omega},\nu^{\omega})}st\left(\int_{X^{\omega}\times X^{\omega}}\left(d^{\omega}(x^{\omega},y^{\omega})\right)^{p}d\gamma^{\omega}(x^{\omega},y^{\omega})\right) & \approx\inf_{\gamma^{\omega}\in\Pi^{\omega}(\mu^{\omega},\nu^{\omega})}\int_{X^{\omega}\times X^{\omega}}\left(d^{\omega}(x^{\omega},y^{\omega})\right)^{p}d\gamma^{\omega}(x^{\omega},y^{\omega})\\
 & =\left(W_{p}^{\omega}(\mu^{\omega},\nu^{\omega})\right)^{p}.
\end{align*}
 Hence, $W_{p,L}(\mu_{L},\nu_{L})\lessapprox W_{p}^{\omega}(\mu^{\omega},\nu^{\omega})$.
\end{proof}
\begin{lem}
\label{lem:Loeb tightness pushdown criterion}Let $\mu_{L}$ be a
Loeb probability measure on $X_{\lim}^{\omega}$. Then the $st_{\hat{d}}$-pushforward
of $\mu_{L}$ is well-defined and belongs to $\mathcal{P}(\hat{X})$
iff for every $\varepsilon>0$, there exists a compact $K_{\varepsilon}\subseteq\hat{X}$
such that $\mu_{L}(st_{\hat{d}}^{-1}(K_{\varepsilon}))>1-\varepsilon$.
\end{lem}

\begin{rem*}
Combining this lemma with Proposition \ref{prop:compact lifting},
we see that $[(\mu_{i})]=\mu^{\omega}\in\mathcal{P}^{\omega}(X^{\omega})$
has a Loeb measure $\mu_{L}$ whose $st_{\hat{d}}$-pushforward is
well-defined and belongs to $\mathcal{P}(\hat{X})$, if and only if:
$\mu^{\omega}\in\mathcal{P}^{\omega}(X_{\text{lim}}^{\omega})$, and
$\omega$-almost surely, for every $\varepsilon>0$ the probability
measures $\mu_{i}$ are concentrated on sets $F_{\varepsilon,i}$
which are uniformly totally bounded (meaning $\mu_{i}(F_{\varepsilon,i})>1-\varepsilon$).
This provides an explicit criterion for the satisfaction of condition
(1) in Definition \ref{def:(ultralimit-of-pmm-space} above. (Compare
also \cite[Theorem 8.3]{pasqualetto2021ultralimits}.)
\end{rem*}
\begin{proof}
One direction is obvious: if there exists a Radon measure $\mu$ such
that $\mu_{L}\circ st_{\hat{d}}^{-1}=\mu$, then we just use the tightness
of $\mu$ to produce a $K_{\varepsilon}$ such that $\mu(K_{\varepsilon})>1-\varepsilon$,
hence $\mu_{L}(st_{\hat{d}}^{-1}(K_{\varepsilon}))>1-\varepsilon$
also.

On the other hand, fix an $\varepsilon>0$ and consider the $K_{\varepsilon}$
such that $\mu_{L}(st_{\hat{d}}^{-1}(K_{\varepsilon}))>1-\varepsilon$.
By Lemma \ref{lem:measurability of st}, we know that $st_{\hat{d}}$
restricted to $K_{\varepsilon}$ is measurable, because $K_{\varepsilon}$
is compact (hence separable). In other words, we can push forward
$\mu_{L}$ onto $\hat{X}$ on all but at most $\varepsilon$ of the
support of $\mu_{L}$. Selecting a sequence of $\varepsilon_{n}$'s
converging to zero, we observe that $\mu_{L}(st_{\hat{d}}^{-1}(\cup_{n}K_{\varepsilon_{n}}))=1$,
and since $\cup_{n}K_{\varepsilon_{n}}$ is separable, we have that
$st_{\hat{d}}$ is measurable on a large enough codomain to push forward
all the mass of $\mu_{L}$ (and also that the support of $\mu_{L}\circ st_{\hat{d}}^{-1}$
is separable). Hence $\mu_{L}\circ st_{\hat{d}}^{-1}$ is a (Radon)
probability measure on $\hat{X}$ as desired.
\end{proof}
\begin{prop}
\label{prop:Wp Loeb contraction 2}Let $\mu,\nu\in\mathcal{P}_{p}(\hat{X})$.
Let $\mu_{L}$ and $\nu_{L}$ be Loeb measures such that $\mu=\mu_{L}\circ st_{\hat{d}}^{-1}$
and $\nu=\nu_{L}\circ st_{\hat{d}}^{-1}$. Then, 
\[
W_{p}(\mu,\nu)\leq W_{p,L}(\mu_{L},\nu_{L}).
\]
\end{prop}

\begin{proof}
Let $\gamma_{L}\in\Pi_{L}(\mu_{L},\nu_{L})$. We claim that the $st_{\hat{d}}$-pushforward
of $\gamma_{L}$ onto $\hat{X}\times\hat{X}$ is in fact a (Radon)
coupling $\gamma\in\Pi(\mu,\nu)$. This claim is sufficient to prove
the proposition, because if so, change-of-variable w.r.t. the map
$st_{\hat{d}}$ implies
\[
\int_{\hat{X}\times\hat{X}}\hat{d}(\hat{x},\hat{y})^{p}d\gamma(\hat{x},\hat{y})=\int_{X_{\lim}^{\omega}\times X_{\lim}^{\omega}}st\left(\left(d^{\omega}(x^{\omega},y^{\omega})\right)^{p}\right)d\gamma_{L}(x^{\omega},y^{\omega})
\]
(where we have used the fact that $\hat{d}(\hat{x},\hat{y})^{p}=\left(\hat{d}(x^{\omega},y^{\omega})\right)^{p}=st\left(d^{\omega}(x^{\omega},y^{\omega})\right)^{p}$
whenever $st_{\hat{d}}(x^{\omega})=\hat{x}$ and $st_{\hat{d}}(y^{\omega})=\hat{y}$).
Then, quantifying over all $\gamma_{L}\in\Pi_{L}(\mu_{L},\nu_{L})$
shows that 
\[
\inf_{\gamma\in\Pi(\mu,\nu)}\int_{\hat{X}\times\hat{X}}\hat{d}(\hat{x},\hat{y})^{p}d\gamma(\hat{x},\hat{y})\leq\inf_{\Pi_{L}(\mu_{L},\nu_{L})}\int_{X_{\lim}^{\omega}\times X_{\lim}^{\omega}}st\left(\left(d^{\omega}(x^{\omega},y^{\omega})\right)^{p}\right)d\gamma_{L}(x^{\omega},y^{\omega})
\]
as desired.

We make use of the previous lemma. Since $\mu=\mu_{L}\circ st_{\hat{d}}^{-1}$
and $\nu=\nu_{L}\circ st_{\hat{d}}^{-1}$, we can find compact sets
$C_{\mu}$ and $C_{\nu}$ in $\hat{X}$ such that $\mu_{L}(st_{\hat{d}}^{-1}C_{\mu})>1-\varepsilon/2$
and similarly $\nu_{L}(st_{\hat{d}}^{-1}C_{\nu})>1-\varepsilon/2$.
This implies that $\gamma_{L}(st_{\hat{d}}^{-1}C_{\mu}\times st_{\hat{d}}^{-1}C_{\nu})>1-\varepsilon$.
By nearly identical reasoning to the proof of the previous lemma,
we observe that $st_{\hat{d}}\times st_{\hat{d}}$ is automatically
measurable when the range is restricted to the compact set $C_{\mu}\times C_{\nu}$.
And so, by exhausting the support of $\gamma_{L}$ using preimages
of compact sets, we see that $st_{\hat{d}}\times st_{\hat{d}}$ is
measurable on a large enough codomain to push forward $\gamma_{L}$
to a Radon probability measure on $\hat{X}\times\hat{X}$, which we
denote by $\gamma$.

Lastly, we observe that the pushforward of $\gamma_{L}$ is in fact
a coupling between $\mu$ and $\nu$; this is a standard argument
but we include it for completeness. Indeed, let $B\subseteq\hat{X}$
be any Borel set: then, $\gamma(B\times\hat{X})=\gamma_{L}(st_{\hat{d}}^{-1}(B)\times X_{\lim}^{\omega})$.
Since the first marginal of $\gamma_{L}$ is $\mu_{L}$, this implies
that $\gamma(B\times\hat{X})=\mu_{L}(st_{\hat{d}}^{-1}(B))$. Hence
the first marginal of $\gamma$ is identical to $\mu_{L}\circ st_{\hat{d}}^{-1}=\mu$.
By identical reasoning, the second marginal of $\gamma$ is $\nu$,
as desired.
\end{proof}

\begin{cor}
\label{cor:monad Loeb pushdown consistency}(1) Let $\mu^{\omega}\approx_{W_{p}^{\omega}}\nu^{\omega}$,
with associated Loeb measures $\mu_{L}$ and $\nu_{L}$. Then, if
the pushforwards $\mu_{L}\circ st_{\hat{d}}^{-1}$ and $\nu_{L}\circ st_{\hat{d}}^{-1}$
are both well-defined and belong to $\mathcal{P}_{p}(\hat{X})$, they
are equal.

(2) Let $\mu^{\omega}\in(\mathcal{P}_{p}(X)_{\lim}^{\omega},W_{p}^{\omega})$
be a lifting of $\mu\in(\mathcal{P}_{p}(\hat{X}),W_{p})$. Let $\mu_{L}$
be the Loeb measure associated to $\mu^{\omega}$. Then, if $\mu_{L}\circ st_{\hat{d}}^{-1}$
is well-defined and belongs to $(\mathcal{P}_{p}(\hat{X}),W_{p})$,
it holds that $\mu=\mu_{L}\circ st_{\hat{d}}^{-1}$. 
\end{cor}

\begin{rem*}
Compare Corollary \ref{cor:monad Loeb pushdown consistency} with
the recent result \cite[Theorem 4.3]{duanmu2018finitely} (the same
result appears, under stricter hypotheses, as \cite[Theorem 9]{duanmu2020existence}),
which occurs in a different setting than ours but nonetheless has
a similar flavor. (Duanmu et al. are able to avoid pathological phenomena
similar to our Example \ref{exa:Loeb pushdown far away} by restricting
to the case where the underlying space is bounded and $\sigma$-compact.) 
\end{rem*}
\begin{proof}
(1) Combining Propositions \ref{prop:Wp Loeb contraction 1} and \ref{prop:Wp Loeb contraction 2},
we see 
\[
W_{p}(\mu_{L}\circ st_{\hat{d}}^{-1},\nu_{L}\circ st_{\hat{d}}^{-1})\leq W_{p,L}(\mu_{L},\nu_{L})\lessapprox W_{p}^{\omega}(\mu^{\omega},\nu^{\omega}).
\]
But $W_{p}^{\omega}(\mu^{\omega},\nu^{\omega})\approx0$, and $W_{p}(\mu_{L}\circ st_{\hat{d}}^{-1},\nu_{L}\circ st_{\hat{d}}^{-1})$
is a positive real, so $W_{p}(\mu_{L}\circ st_{\hat{d}}^{-1},\nu_{L}\circ st_{\hat{d}}^{-1})=0$.

(2) Let $\tilde{\mu}^{\omega}$ be a lifting of $\mu$ such that $\tilde{\mu}_{L}\circ st_{\hat{d}}^{-1}=\mu$;
we know that such a $\tilde{\mu}^{\omega}$ exists thanks to Theorem
\ref{thm:radon representation}. Since both $\mu^{\omega}$ and $\tilde{\mu}^{\omega}$
are liftings of $\mu$, it follows that $\mu^{\omega}\approx_{W_{p}^{\omega}}\tilde{\mu}^{\omega}$.
Therefore, part (1) shows that $\mu_{L}\circ st_{\hat{d}}^{-1}=\tilde{\mu}_{L}\circ st_{\hat{d}}^{-1}=\mu$.
\end{proof}
One might ask if a certain converse to Corollary \ref{cor:monad Loeb pushdown consistency}
holds --- namely, if we have two different Loeb measures $\mu_{L}$
and $\mu_{L}^{\prime}$ that both push forward to some $\mu\in\mathcal{P}_{p}(\hat{X})$,
is it the case that the internal measures $\mu^{\omega}$ and $\mu^{\prime\omega}$
which generated $\mu_{L}$ and $\mu_{L}^{\prime}$ must be infinitesimal
$W_{p}^{\omega}$ distance apart? The following shows that this is
not the case.
\begin{example}
\label{exa:Loeb pushdown far away}Let $\mu\in(\mathcal{P}_{2}(\hat{X}),W_{2})$.
Suppose that $\mu_{L}$ is a Loeb measure on $X^{\omega}$ that pushes
forward to $\mu$, and that $\mu^{\omega}$ is an internal measure
in $(\widehat{\mathcal{P}_{2}(X)},\hat{W}_{2})$. Then, it is \emph{not
}necessarily the case that $\mu^{\omega}$ is a $\hat{W}_{2}$-lifting
of $\mu$. Indeed, take, as in Example \ref{exa:bad two point ultra-measure},
the internal measure $(1-N^{-1})\delta_{0}+N^{-1}\delta_{N^{1/2}}$
with $N\in\mathbb{N}^{\omega}\backslash\mathbb{N}$. We know that
this measure is not a $\hat{W}_{2}$-lifting of $\delta_{0}$, but
its associated Loeb measure is $\delta_{0}$ (with underlying space
$\mathbb{R}^{\omega}$), which pushes forward to $\delta_{0}$ (with
underlying space $\mathbb{R}$). 
\end{example}

The source of the issue may be roughly stated as follows: going from
internal probability measures to Loeb probability measures is compatible
with total variation, which only controls the $p$-Wasserstein distance
for probability measures which are both contained in a fixed set of
bounded diameter. Stated a bit more explicitly: 
\begin{lem}
\label{lem:TV-omega}(1) Let $\mu^{\omega},\nu^{\omega}\in\mathcal{P}(X)^{\omega}$.
Then $\mu_{L}=\nu_{L}$ iff $TV^{\omega}(\mu^{\omega},\nu^{\omega})\approx0$. 

(2) Suppose that $\mu^{\omega}$ and $\nu^{\omega}$ are both supported
within an internal subset of $X^{\omega}$ with diameter at most $D\in\mathbb{R}^{\omega}$.
Then, for all $p\in[1,\infty)$, $W_{p}^{\omega}(\mu^{\omega},\nu^{\omega})\leq D\cdot TV^{\omega}(\mu^{\omega},\nu^{\omega})$.
\end{lem}

\begin{proof}
(1) Suppose that $TV^{\omega}(\mu^{\omega},\nu^{\omega})\approx0$.
Then $\mu^{\omega}(B^{\omega})\approx\nu^{\omega}(B^{\omega})$ for
every internal Borel set $B^{\omega}$, so $st\circ\mu^{\omega}(B^{\omega})=st\circ\nu^{\omega}(B^{\omega})$
for every internal Borel set $B^{\omega}$. Since the Loeb measure
is the completion of the premeasure $st\circ\mu^{\omega}$, this shows
that $\mu_{L}=\nu_{L}$.

Conversely, suppose $TV^{\omega}(\mu^{\omega},\nu^{\omega})\not\approx0$.
Then, there exists an internal Borel set $B^{\omega}$ and a standard
$\delta>0$ such that $|\mu^{\omega}(B^{\omega})-\nu^{\omega}(B^{\omega})|>\delta$,
hence $st\circ\mu^{\omega}(B^{\omega})\neq st\circ\nu^{\omega}(B^{\omega})$,
and hence $\mu_{L}\neq\nu_{L}$.

(2) This follows immediately from \L o\'{s}'s theorem applied to the
inequality $W_{p}(\mu,\nu)\leq D\cdot TV(\mu,\nu)$, which holds on
general metric spaces. 
\end{proof}
The two parts of the preceding lemma tell us that if $\mu_{L}=\nu_{L}$
and both $\mu^{\omega}$ and $\nu^{\omega}$ have support which has
bounded diameter, then $W_{p}^{\omega}(\mu^{\omega},\nu^{\omega})\approx0$.
However this is actually \emph{not} enough to deduce that $\mu^{\omega}$
(and hence also $\nu^{\omega}$) must be a lifting of $\mu_{L}\circ st_{\hat{d}}^{-1}$,
even if this pushforward is well-defined. Nonetheless, it \emph{does}
turn out to be true that if $\mu^{\omega}$ has bounded support and
$\mu_{L}\circ st_{\hat{d}}^{-1}$ is well-defined (and thus belongs
to $\mathcal{P}_{p}(\hat{X})$), then indeed $\mu^{\omega}$ is automatically
a lifting of $\mu_{L}\circ st_{\hat{d}}^{-1}$, as the next lemma
shows; in this way, we have a partial converse to part (2) of Corollary
\ref{cor:monad Loeb pushdown consistency}.
\begin{lem}
\label{lem:bounded diameter pushfwd lifting}Let $p\in[1,\infty)$.
Let $\mu^{\omega}\in\mathcal{P}_{p}(X)_{\lim}^{\omega}$, and suppose,
in addition, that $\mu^{\omega}$ is supported within $B(e^{\omega},D)$
for some $D\in\mathbb{R}$, that is, $\mu^{\omega}(X^{\omega}\backslash B(e^{\omega},D))=0$.
Then, if the $st_{\hat{d}}$-pushforward of $\mu_{L}$ is well-defined
and belongs to $(\mathcal{P}_{p}(\hat{X}),W_{p})$, it holds that
$\mu^{\omega}$ is a $\hat{W}_{p}$-lifting of $\mu_{L}\circ st_{\hat{d}}^{-1}$.
\end{lem}

\begin{proof}
We proceed by way of a \emph{semi-discrete transport} argument. In
particular, it turns out to be valuable to reason using transport
\emph{maps} rather than merely reasoning using transport plans.

Let $\mu:=\mu_{L}\circ st_{\hat{d}}^{-1}\in\mathcal{P}_{p}(\hat{X})$.
Let $(X_{1},X_{2},\ldots)$ be a sequence of i.i.d. random variables
with distribution $\mu$, and let $\hat{\mu}_{n}=\frac{1}{n}\sum_{j=1}^{n}\delta_{x_{j}}$
be the $n$th (random) empirical measure for $\mu$ (where $x_{j}=X_{j}$
for each $j=1,\ldots,n$). Note that $W_{2}(\mu,\hat{\mu}_{n})\rightarrow0$
with probability $1$; at the same time, notice that for \emph{any}
sequence of empirical measures $(\hat{\mu}_{n})$, we have that 
\begin{align*}
W_{p}^{p}(\mu,\hat{\mu}_{n}) & =\min_{\gamma\in\Pi(\mu,\tilde{\mu}_{n})}\int_{\text{supp}(\mu)\times\text{supp}(\hat{\mu}_{n})}\hat{d}^{p}(x,y)d\gamma(x,y)\\
 & \geq\int_{\text{supp}(\mu)\times\text{supp}(\hat{\mu}_{n})}\left(\min_{\tilde{y}\in\text{supp}(\mu_{n})}\hat{d}^{p}(x,\tilde{y})\right)d\gamma(x,y)\\
 & =\int_{\text{supp}(\mu)}\left(\min_{\tilde{y}\in\text{supp}(\mu_{n})}\hat{d}^{p}(x,\tilde{y})\right)d\mu.
\end{align*}
Let $\tau_{n}$ denote the map which sends $x\in supp(\mu)$ to $\text{argmin}_{\tilde{y}\in\text{supp}(\hat{\mu}_{n})}\hat{d}^{p}(x,\tilde{y})$;
in the case where the argmin is not unique, we simply chose a minimizer
arbitrarily in such a way that $\tau_{n}$ is measurable (and this
is easy to do since $\hat{\mu}_{n}$ has finite support, so for instance
one can map $x$ lexicographically in the case of a tie, i.e. we index
the points in $\text{supp}(\hat{\mu}_{n})$ and map $x$ to the point
with the lowest index that instantiates the minimum). Note that
\[
\int_{\text{supp}(\mu)}\left(\min_{\tilde{y}\in\text{supp}(\hat{\mu}_{n})}\hat{d}^{p}(x,\tilde{y})\right)d\mu=\int_{\text{supp}(\mu)}\left(\hat{d}^{p}(x,\tau_{n}(x)\right)d\mu.
\]
At the same time, if we consider the pushforward measure $\mu_{n}:=\mu\circ\tau_{n}^{-1}$,
we see that $\mu_{n}$ has support contained within $\text{supp}(\hat{\mu}_{n})$,
and 
\[
W_{p}^{p}(\mu,\mu_{n})\leq\int_{\text{supp}(\mu)}\left(\hat{d}^{p}(x,\tau_{n}(x)\right)d\mu\leq W_{p}^{p}(\mu,\hat{\mu}_{n}).
\]
Note also that, given any point $y_{j,n}\in\text{supp}(\mu_{n})$,
we have that $\mu_{n}(y_{j,n})=\mu(\tau_{n}^{-1}(y_{j,n}))$, so we
can write $\mu_{n}$ as a sum of point masses like so: 
\[
\mu_{n}=\sum_{j=1}^{n}\mu(\tau_{n}^{-1}(y_{j,n}))\delta_{y_{j,n}}.
\]

In what follows, we condition on the event where $W_{p}(\mu,\hat{\mu}_{n})\rightarrow0$;
this implies that $W_{p}(\mu,\mu_{n})\rightarrow0$. 

Let $\varepsilon>0$. Suppose that $n$ is chosen so that 
\[
\int_{\text{supp}(\mu)}\left(\hat{d}^{p}(x,\tau_{n}(x)\right)d\mu\leq W_{p}^{p}(\mu,\hat{\mu}_{n})<\varepsilon^{p}.
\]
For each $y_{j,n}$, we select a lifting $y_{j,n}^{\omega}\in st_{\hat{d}}^{-1}(y_{j,n})$;
this gives us a $\hat{W}_{p}$-lifting of the measure $\mu_{n}$,
namely 
\[
\mu_{n}^{\omega}:=\sum_{j=1}^{n}\mu(\tau_{n}^{-1}(y_{j,n}))\delta_{y_{j,n}^{\omega}}.
\]
At the same time, let $E_{j,n}^{\omega}$ be an internal measurable
set such that
\[
E_{j,n}^{\omega}\subseteq st_{\hat{d}}^{-1}(\tau_{n}^{-1}(y_{j,n}))\text{ and }\mu_{L}(st_{\hat{d}}^{-1}(\tau_{n}^{-1}(y_{j,n}))\backslash E_{j,n}^{\omega})<\frac{\varepsilon}{n}.
\]
(Such an $E_{j,n}^{\omega}$ is guaranteed to exist by Fact \ref{fact:Loeb symmetric difference}.)
In other words, $\{E_{j,n}^{\omega}\}_{j=1}^{n}$ is a finite family
of internal measurable sets which approximates the $\sigma(\mathcal{B}^{\omega})$-measurable
partition $\{st_{\hat{d}}^{-1}(\tau_{n}^{-1}(y_{j,n}))\}_{j=1}^{n}$
(which partitions $st_{\hat{d}}^{-1}(\text{supp}(\mu))$). Likewise,
we may define the ``partial transport map'' 
\[
\tilde{\tau}_{n}^{\omega}:\bigcup_{j=1}^{n}E_{j,n}^{\omega}\rightarrow X^{\omega}
\]
\[
\tilde{\tau}_{n}^{\omega}(x^{\omega})=\begin{cases}
y_{1,n}^{\omega} & x^{\omega}\in E_{1,n}^{\omega}\\
\vdots & \vdots\\
y_{j,n}^{\omega} & x^{\omega}\in E_{j,n}^{\omega}\\
\vdots\\
y_{n,n}^{\omega} & x^{\omega}\in E_{j,n}^{\omega}
\end{cases}.
\]
Since $\tilde{\tau}_{n}^{\omega}$ has a range with finitely many
values and every $E_{j,n}^{\omega}$ is internal, it is clear that
$\tilde{\tau}_{n}^{\omega}$ is internally measurable. We then extend
$\tilde{\tau}_{n}^{\omega}$ to a transport map in a slightly arbitrary
fashion; for simplicity, we set 
\[
\tau_{n}^{\omega}(x^{\omega})=\begin{cases}
\tilde{\tau}_{n}^{\omega}(x^{\omega}) & x^{\omega}\in\bigcup_{j=1}^{n}E_{j,n}^{\omega}\\
e^{\omega} & \text{else}.
\end{cases}
\]
By the same reasoning, $\tau_{n}^{\omega}$ is an internally measurable
map. We may therefore consider the pushforward measure $\mu^{\omega}$
with respect to $\tau_{n}^{\omega}$, that is, $\mu^{\omega}\circ(\tau_{n}^{\omega})^{-1}$.

By unpacking definitions, we observe that 
\[
W_{p}^{\omega}(\mu^{\omega},\mu^{\omega}\circ(\tau_{n}^{\omega})^{-1})\leq\left(\int_{X^{\omega}}\left(d^{\omega}(x^{\omega},\tau_{n}^{\omega}(x^{\omega}))\right)^{p}d\mu^{\omega}\right)^{1/p}.
\]
Indeed, since $\tau_{n}^{\omega}$ is internal, it can be written
as $[(\tau_{n,i})]$ for maps $\tau_{n,i}:X_{i}\rightarrow X_{i}$.
At the same time, $\mu^{\omega}=[(\mu_{i})]$ for some sequence $(\mu_{i})_{i\in\mathbb{N}}$
of measures in $\mathcal{P}(X_{i})$. For each of these maps, consider
the induces transport plan $\gamma_{i}:=(Id\times\tau_{n,i})_{\#}\mu_{i}$,
which is a transport plan whose first marginal is $\mu_{i}$ and whose
second marginal is $\mu_{i}\circ\tau_{n,i}^{-1}$. Then $\gamma^{\omega}:=[(\gamma_{i})]$
is a transport plan between $[(\mu_{i})]:=\mu^{\omega}$ and $[(\mu_{i}\circ\tau_{n,i}^{-1})]:=[(\mu_{i})]\circ[(\tau_{n,i}^{-1})]:=\mu^{\omega}\circ(\tau_{n}^{\omega})^{-1}$.
It remains only to note that 
\begin{align*}
W_{p}^{\omega}(\mu^{\omega},\mu^{\omega}\circ(\tau_{n}^{\omega})^{-1}) & \leq\left(\int_{X^{\omega}\times X^{\omega}}(d^{\omega}(x^{\omega},y^{\omega}))^{p}d\gamma^{\omega}\right)^{1/p}\\
 & :=\left[\left(\left(\int_{X_{i}\times X_{i}}(d_{i}(x_{i},y_{i}))^{p}d\gamma_{i}(x_{i},y_{i})\right)^{1/p}\right)\right]\\
 & =\left[\left(\left(\int_{X_{i}}(d_{i}(x_{i},\tau_{n,i}(x_{i})))^{p}d\mu_{i}(x_{i})\right)^{1/p}\right)\right]\\
 & :=\left(\int_{X^{\omega}}(d^{\omega}(x^{\omega},\tau_{n}^{\omega}(x^{\omega})))^{p}d\mu^{\omega}\right)^{1/p}.
\end{align*}

Now, 
\begin{align*}
\int_{X^{\omega}}\left(d^{\omega}(x^{\omega},\tau_{n}^{\omega}(x^{\omega}))\right)^{p}d\mu^{\omega} & =\sum_{j=1}^{n}\int_{E_{j,n}^{\omega}}\left(d^{\omega}(x^{\omega},\tau_{n}^{\omega}(x^{\omega}))\right)^{p}d\mu^{\omega}+\int_{X^{\omega}\backslash\bigcup_{j=1}^{n}E_{j,n}^{\omega}}\left(d^{\omega}(x^{\omega},\tau_{n}^{\omega}(x^{\omega}))\right)^{p}d\mu^{\omega}\\
 & <\sum_{j=1}^{n}\int_{E_{j,n}^{\omega}}\left(d^{\omega}(x^{\omega},\tau_{n}^{\omega}(x^{\omega}))\right)^{p}d\mu^{\omega}+D^{p}\varepsilon^{p}\\
 & \approx\sum_{j=1}^{n}\int_{E_{j,n}^{\omega}}st\left(d^{\omega}(x^{\omega},\tau_{n}^{\omega}(x^{\omega}))\right)^{p}d\mu_{L}+D^{p}\varepsilon^{p}\\
(*) & \approx\sum_{j=1}^{n}\int_{E_{j,n}^{\omega}}\left(st(d^{\omega}(x^{\omega},\tau_{n}^{\omega}(x^{\omega})))\right)^{p}d\mu_{L}+D^{p}\varepsilon^{p}
\end{align*}
In particular, in line $(*)$ we have used the fact that $(\cdot)^{p}$
is continuous, and $d^{\omega}(x^{\omega},\tau_{n}^{\omega}(x^{\omega}))$
is nearstandard --- this nearstandardness, in turn, follows from
the fact that each $E_{j,n}^{\omega}$ has finite diameter. 

By definition, $st((d^{\omega}(x^{\omega},\tau_{n}^{\omega}(x^{\omega})))=\hat{d}(st_{\hat{d}}(x^{\omega}),st_{\hat{d}}(\tau_{n}^{\omega}(x^{\omega})))$;
by construction, 
\begin{align*}
st_{\hat{d}}(\tau_{n}^{\omega}(x^{\omega})) & =y_{j,n}\\
x^{\omega}\in E_{j,n}^{\omega} & =\tau_{n}(st_{\hat{d}}(x^{\omega}))
\end{align*}
since $st_{\hat{d}}(E_{j,n}^{\omega})\subseteq\tau_{n}^{-1}(y_{j,n})$.
Therefore, 
\begin{align*}
\int_{E_{j,n}^{\omega}}\left(st(d^{\omega}(x^{\omega},\tau_{n}^{\omega}(x^{\omega})))\right)^{p}d\mu_{L} & =\int_{E_{j,n}^{\omega}}\left(\hat{d}(st_{\hat{d}}(x^{\omega}),\tau_{n}(st_{\hat{d}}(x^{\omega})))\right)^{p}d\mu_{L}\\
 & \leq\int_{st_{\hat{d}}^{-1}(\tau_{n}^{-1}(y_{j,n}))}\left(\hat{d}(st_{\hat{d}}(x^{\omega}),\tau_{n}(st_{\hat{d}}(x^{\omega})))\right)^{p}d\mu_{L}\\
 & =\int_{\tau_{n}^{-1}(y_{j,n})}\left(\hat{d}(x,\tau_{n}(x))\right)^{p}d(\mu_{L}\circ st_{\hat{d}}^{-1}).
\end{align*}
Since $\mu=\mu_{L}\circ st_{\hat{d}}^{-1}$, and $\text{supp}(\mu)=\bigsqcup_{j=1}^{n}\tau_{n}^{-1}(y_{j,n})$,
we deduce that 
\begin{align*}
\int_{X^{\omega}}\left(d^{\omega}(x^{\omega},\tau_{n}^{\omega}(x^{\omega}))\right)^{p}d\mu^{\omega} & <\int_{\text{supp}(\mu)}\left(\hat{d}(x,\tau_{n}(x))\right)^{p}d\mu+D^{p}\varepsilon^{p}\\
 & \leq W_{p}^{p}(\mu,\mu_{n})+D^{p}\varepsilon^{p}\\
 & \leq(D^{p}+1)\varepsilon^{p}.
\end{align*}
Therefore $W_{p}^{\omega}(\mu^{\omega},\mu^{\omega}\circ(\tau_{n}^{\omega})^{-1})<(D^{p}+1)^{1/p}\varepsilon$.

At the same time, we claim that
\[
W_{p}^{\omega}(\mu_{n}^{\omega},\mu^{\omega}\circ(\tau_{n}^{\omega})^{-1})<D\varepsilon.
\]
Indeed, 
\[
\mu_{n}^{\omega}=\sum_{j=1}^{n}\mu(\tau_{n}^{-1}(y_{j,n}))\delta_{y_{j,n}^{\omega}};\qquad\mu^{\omega}\circ(\tau_{n}^{\omega})^{-1}=\mu^{\omega}\left(X^{\omega}\backslash\bigcup_{j=1}^{n}E_{j,n}^{\omega}\right)\delta_{e^{\omega}}+\sum_{j=1}^{n}\mu^{\omega}(E_{j,n}^{\omega})\delta_{y_{j,n}^{\omega}}.
\]
Since for all $j=1,\ldots,n$, 
\[
\mu_{L}\left(st_{\hat{d}}^{-1}(\tau_{n}^{-1}(y_{j,n}))\backslash E_{j,n}^{\omega}\right)<\frac{\varepsilon}{n}
\]
and 
\[
\mu(\tau_{n}^{-1}(y_{j,n}))=\mu_{L}(st_{\hat{d}}^{-1}(\tau_{n}^{-1}(y_{j,n})))
\]
we see that 
\[
\mu(\tau_{n}^{-1}(y_{j,n}))-\frac{\varepsilon^{p}}{n}<\mu^{\omega}(E_{j,n}^{\omega})\lessapprox\mu(\tau_{n}^{-1}(y_{j,n}))
\]
and
\[
\mu^{\omega}\left(X^{\omega}\backslash\bigcup_{j=1}^{n}E_{j,n}^{\omega}\right)<\varepsilon.
\]
In particular, this implies that 
\begin{align*}
TV^{\omega}(\mu_{n}^{\omega},\mu^{\omega}\circ(\tau_{n}^{\omega})^{-1}) & :=\frac{1}{2}\left(\sum_{j=1}^{n}|\mu^{\omega}(E_{j,n}^{\omega})-\mu(\tau_{n}^{-1}(y_{j,n}))|+\mu^{\omega}\left(X^{\omega}\backslash\bigcup_{j=1}^{n}E_{j,n}^{\omega}\right)\right)\\
 & \leq\frac{1}{2}\left(\sum_{j=1}^{n}\frac{\varepsilon}{n}+\varepsilon\right)=\varepsilon.
\end{align*}
At the same time, all the points $e^{\omega}$ and $y_{j,n}^{\omega}$
(for $j=1,\ldots,n$) reside in a domain with diameter at most $D$.
Hence, by Lemma \ref{lem:TV-omega}, we know that
\[
W_{p}^{\omega}(\mu_{n}^{\omega},\mu^{\omega}\circ(\tau_{n}^{\omega})^{-1})\leq D\cdot TV^{\omega}(\mu_{n}^{\omega},\mu^{\omega}\circ(\tau_{n}^{\omega})^{-1})<D\varepsilon.
\]
 So from the triangle inequality, we have that 
\[
W_{p}^{\omega}(\mu^{\omega},\mu_{n}^{\omega})<\left((D^{p}+1)^{1/p}+D\right)\varepsilon.
\]

Lastly, let $\tilde{\mu}^{\omega}$ be any $\hat{W}_{p}$-lifting
of $\mu$; from the fact that $W_{p}(\mu,\mu_{n})<\varepsilon$, we
know that $W_{p}^{\omega}(\tilde{\mu}^{\omega},\mu_{n}^{\omega})<\varepsilon$.
Therefore, 
\[
W_{p}^{\omega}(\mu^{\omega},\tilde{\mu}^{\omega})<\left((D^{p}+1)^{1/p}+D+1\right)\varepsilon.
\]
But $\varepsilon>0$ was arbitrary, so we conclude that $W_{p}^{\omega}(\mu^{\omega},\tilde{\mu}^{\omega})\approx0$.
Hence $\mu^{\omega}$ is also a lifting of $\mu$, as desired. 
\end{proof}

\section{Ultralimits of $CD(K,\infty)$ spaces\label{sec:Ultralimits-of-CD(K,infty)}}

The aim of this section is to prove the following result. 
\begin{thm}
\label{thm:main} Fix a non-principal ultrafilter $\omega$ on $\mathbb{N}$.
Consider a sequence $(X_{i},d_{i},e_{i},\mu_{i})$ of pointed metric
measure spaces with $\mu_{i}\in\mathcal{P}(X_{i})$, where each $(X_{i},d_{i},\mu_{i})$
satisfies the synthetic Ricci curvature bound ``$CD(K_{i},\infty)$'',
for $\omega$-almost all $i\in\mathbb{N}$. Let $(X,d,e,\mu)$ be
the metric measure ultralimit of $(X_{i},d_{i},e_{i},\mu_{i})_{i\in\mathbb{N}}$
in the sense of Definition \ref{def:(ultralimit-of-pmm-space}, so
that $\mu\in\mathcal{P}(X)$. Then $(X,d,\mu)$ satisfies the synthetic
Ricci curvature lower bound ``$CD(K,\infty)$'' with $K=st(K^{\omega})$.
\end{thm}

\begin{rem*}
We draw the reader's attention to a number of complications which
are present.%
{} We also mention why it is not possible to directly prove the stability
of the ``strong $CD(K,\infty)$'' property using a similar strategy
to ours.

First, one concern is the following: we might like to witness the
$CD(K,\infty)$ property in the limiting space $\mathcal{P}_{2}(\hat{X})$
by taking an ultraproduct of geodesics in $\mathcal{P}_{2,ac(\mu_{i})}(X_{i})$
along which the relative entropy functional is $K_{i}$-geodesically
convex, and then pushing down the resulting curve to $\mathcal{P}_{2}(\hat{X})$.
But what can conceivably happen is that the ultraproduct of the geodesics
only gives us a limiting geodesic in the bigger space $\widehat{\mathcal{P}_{2}(X)}$,
whereas the geodesics inside $\mathcal{P}_{2}(\hat{X})$ are instead
produced by ultraproducts of sequences of geodesics that do \emph{not}
witness the (weak) $K_{i}$-geodesic convexity of the relative entropy.
It turns out that a careful argument is able to sidestep this problem,
but this is one reason why our proof is not a single paragraph argument
(as with a proof of the stability of Alexandrov-type synthetic sectional
curvature bounds w.r.t. ultralimits). 

Simultaneously, it is quite possible that $\mathcal{P}_{2,ac}(\hat{X})$
contains geodesics which do \emph{not} arise as pushdowns of ultraproducts
of geodesics; rather, of curves in $\mathcal{P}_{2}(X)_{\lim}^{\omega}$
whose length is \emph{infinitesimally close} to the $W_{2}^{\omega}$-distance
between the endpoints --- so, for instance, ultraproducts of sequences
of a.c. curves of the form $g_{i}:[0,1]\rightarrow\mathcal{P}_{2,ac}(X_{i})$
where $\omega$-a.s., 
\[
|\text{length}(g_{i})-W_{2}(g_{i}(0),g_{i}(1))|\rightarrow0.
\]
Since the $CD(K,\infty)$ property has nothing to say about approximate
geodesics (rather, one would need to use something like the \emph{lax}
$CD(K,\infty)$ property from \cite{sturm2006geometry1}), it is non-obvious
how the pushdown of $[(g_{i})]$ interacts with the relative entropy
on $\mathcal{P}_{2}(\hat{X})$ even if the pushdown is well-defined
and has range in $\mathcal{P}_{2,ac(\mu)}(\hat{X})$. We note that
similar problems in the setting of Gromov-Hausdorff-type convergence
of metric measure spaces have already been discussed, for instance
at the very end of the last chapter of \cite{villani2008optimal}
which addresses the (non)-stability of the curvature-dimension condition
``strong $CD(0,N)$''. The type of example discussed therein by
Villani, of a sequence of spaces where the limiting space has vastly
more (and ill-behaved) geodesics makes it hard to suggest that, without
some other side condition (such as the $RCD(K,\infty)$ property),
it is plausible that ``ultralimit of strong $CD(K,\infty)$ spaces
is strong $CD(K,\infty)$'' fails to hold in general. 
\end{rem*}
We begin with a number of preparatory lemmas.
\begin{lem}
\label{lem:Radon-Nikodym Loeb}Let $(\mu_{i})$ and $(\nu_{i})$ be
sequences of probability measures on $(X_{i},d_{i})$, with ultraproducts
$\mu^{\omega}$ and $\nu^{\omega}$ in $(\mathcal{P}_{\lim}^{\omega}(X),W_{2}^{\omega})$. 

(1) We have $\nu_{i}\ll\mu_{i}$ $\omega$-a.s., with Radon-Nikodym
derivative $\frac{d\nu_{i}}{d\mu_{i}}:=h_{i}$, iff 
\[
\forall B\in\mathcal{B}^{\omega}\quad\nu^{\omega}(B)=\int_{B}h^{\omega}d\mu^{\omega}.
\]

(2) Furthermore, we have that $h_{i}\in L^{p}(\mu_{i})$ for some
fixed $p\in(1,\infty)$, with $\Vert h_{i}\Vert_{L^{p}(\mu_{i})}\leq C$
$\omega$-a.s. for some uniform constant $C$, iff
\[
\int_{X^{\omega}}|h^{\omega}|^{p}d\mu^{\omega}<C^{p}.
\]
Similarly, if $h_{i}\in L^{\infty}(\mu_{i})$, then $\Vert h_{i}\Vert_{L^{\infty}(\mu_{i})}\leq C$
$\omega$-a.s., iff 
\[
\mu^{\omega}\{x^{\omega}:h^{\omega}>C\}=0.
\]

(3) Suppose that 
\[
\forall B\in\mathcal{B}^{\omega}\quad\nu^{\omega}(B)=\int_{B}h^{\omega}d\mu^{\omega}.
\]
Then if $h^{\omega}$ is $S$-integrable with respect to $\mu_{L}$,
\[
\forall B\in\sigma(\mathcal{B}^{\omega})\quad\nu_{L}(B)=\int_{B}st\circ h^{\omega}d\mu_{L}.
\]
\end{lem}

\begin{proof}
 (1, 2) These are immediate consequences of \L o\'{s}'s theorem.

(3) %
Suppose that $\nu^{\omega}=\int h^{\omega}d\mu^{\omega}$ and $h^{\omega}$
is $S$-integrable. From the definition of $S$-integrable function,
we have that 
\[
\int st\circ h^{\omega}d\mu_{L}=st\left(\int h^{\omega}d\mu^{\omega}\right).
\]
Now, if $B^{\omega}\in\mathcal{B}^{\omega}$ is any internal measurable
set, and $h^{\omega}$ is any internal function, it holds that 
\[
st\left(1_{B^{\omega}}(x^{\omega})\cdot h^{\omega}(x^{\omega})\right)=\begin{cases}
st(h^{\omega}(x^{\omega})) & x^{\omega}\in B^{\omega}\\
0 & \text{else}
\end{cases}
\]
in other words, $st\left(1_{B^{\omega}}\cdot h^{\omega}\right)=1_{B^{\omega}}\cdot st(h^{\omega})$.
Therefore, 
\[
\int_{B^{\omega}}st\circ h^{\omega}d\mu_{L}=st\left(\int_{B^{\omega}}h^{\omega}d\mu^{\omega}\right).
\]
On the other hand, it holds, directly from the definition of the Loeb
measure, that for any internal measurable set $B^{\omega}\in\mathcal{B}^{\omega}$,
\[
\nu_{L}(B^{\omega})=st\left(\nu^{\omega}(B^{\omega})\right).
\]
Hence, 
\[
(\forall B^{\omega}\in\mathcal{B}^{\omega})\qquad\nu_{L}(B^{\omega})=st\left(\nu^{\omega}(B^{\omega})\right)=st\left(\int_{B^{\omega}}h^{\omega}d\mu^{\omega}\right)=\int_{B^{\omega}}st\circ h^{\omega}d\mu_{L}.
\]

This equality then extends to all $B\in\sigma(\mathcal{B}^{\omega})$
by an application of Fact \ref{fact:Loeb symmetric difference}. Indeed,
given any $B\in\sigma(\mathcal{B}^{\omega})$, take some $B^{\omega}\in\mathcal{B}^{\omega}$
with $\mu_{L}(B^{\omega}\Delta B)=0$. Then, 
\[
\int_{B^{\omega}}st\circ h^{\omega}d\mu_{L}=\int_{B}st\circ h^{\omega}d\mu_{L}.
\]
On the other hand, suppose that $\nu_{L}(B^{\omega}\Delta B)=\delta>0$.
Then (again by Fact \ref{fact:Loeb symmetric difference}), we can
find some internal $D^{\omega}$ with $D^{\omega}\subset B^{\omega}\Delta B$
and $\nu_{L}(D^{\omega})>\frac{\delta}{2}$, and hence $\nu^{\omega}(D^{\omega})>\frac{\delta}{2}$
also. Now, compute that 
\begin{align*}
\frac{\delta}{2}<\nu^{\omega}(D^{\omega}) & =\int_{D^{\omega}}h^{\omega}d\mu^{\omega}\\
 & \approx\int_{D^{\omega}}st\circ h^{\omega}d\mu_{L}\\
 & \leq\int_{B^{\omega}\Delta B}st\circ h^{\omega}d\mu_{L}\\
 & =0.
\end{align*}
Thus we have a contradiction; it follows that $\nu_{L}(B^{\omega}\Delta B)=0$.

Consequently, 
\[
\nu_{L}(B)=\nu_{L}(B^{\omega})=\int_{B^{\omega}}st\circ h^{\omega}d\mu_{L}=\int_{B}st\circ h^{\omega}d\mu_{L}
\]
as desired.
\end{proof}
\begin{defn}
We write $\nu^{\omega}\ll^{\omega}\mu^{\omega}$ to denote the fact
that $\nu_{i}\ll\mu_{i}$ $\omega$-almost surely. 
\end{defn}

\begin{lem}
\label{lem:pushdown of density}(Cf. \cite[Section 5]{anderson82representation})
Let $\mu^{\omega}\in\mathcal{P}_{2}(X)^{\omega}$ be an internal probability
measure whose Loeb measure $\mu_{L}$ pushes forward to a Radon probability
measure $\mu$ on $(\hat{X},d)$. Let $f^{\omega}$ be an internal
measurable function such that $\int_{X^{\omega}}f^{\omega}d\mu^{\omega}=1$,
and further assume that $f^{\omega}$ is $S$-integrable with respect
to $\mu$. Let $\nu^{\omega}:=f^{\omega}d\mu^{\omega}$, and let $\nu_{L}$
denote the Loeb measure associated to $\nu^{\omega}$.

Then, $\nu_{L}$ also pushes forward to a Radon measure $\nu$ on
$\hat{X}$, and $\nu\ll\mu$. 
\end{lem}

\begin{proof}
Let $f^{\omega}$ be an $S$-integrable function with respect to $\mu^{\omega}$,
with $\nu^{\omega}(A)=\int_{A}f^{\omega}d\mu^{\omega}$ for every
internal Borel set $A\in\mathcal{B}^{\omega}$. Then, 
\[
\nu_{L}(A)=st\left(\int_{A}f^{\omega}d\mu^{\omega}\right)=\int_{A}st\circ f^{\omega}d\mu_{L}
\]
and also, for all $B\in\sigma(\mathcal{B}^{\omega})$, 
\[
\nu_{L}(B)=\int_{B}(st\circ f^{\omega})d\mu_{L}.
\]
Let $\mu:=\mu_{L}\circ st_{\hat{d}}^{-1}$. By assumption, $\mu$
has separable support. Let $\mathbb{E}_{st_{\hat{d}}^{-1}}$ denote
the conditional expectation of $st\circ f^{\omega}$ with respect
to the sub-$\sigma$-algebra of $\sigma(\mathcal{B}^{\omega})$ formed
by taking the $st_{\hat{d}}$-preimage of the Borel $\sigma$-algebra
on $\hat{X}$ restricted to $supp(\mu)$. (That this is indeed a sub-$\sigma$-algebra
follows from Lemma \ref{lem:measurability of st}.) Note that for
all Borel $\tilde{B}\subseteq supp(\mu)$, 
\[
\nu_{L}(st_{\hat{d}}^{-1}(\tilde{B}))=\int_{st_{\hat{d}}^{-1}(\tilde{B})}(st\circ f^{\omega})d\mu_{L}=\int_{st_{\hat{d}}^{-1}(\tilde{B})}\mathbb{E}_{st_{\hat{d}}^{-1}}(st\circ f^{\omega})d\mu_{L}.
\]
Since $\mathbb{E}_{st_{\hat{d}}^{-1}}(st\circ f^{\omega})$ is constant
on fibers $st_{\hat{d}}^{-1}(x)$ for $\mu_{L}$-almost all $x$,
we see that $\mathbb{E}_{st_{\hat{d}}^{-1}}(st\circ f^{\omega})\circ st_{\hat{d}}^{-1}:supp(\mu)\rightarrow\mathbb{R}$
is a well-defined Borel measurable function, and change-of-variable
for measures shows that 
\begin{align*}
\int_{st_{\hat{d}}^{-1}(\tilde{B})}\mathbb{E}_{st_{\hat{d}}^{-1}}(st\circ f^{\omega})d\mu_{L} & =\int_{\tilde{B}}\left(\mathbb{E}_{st_{\hat{d}}^{-1}}(st\circ f^{\omega})\circ st_{\hat{d}}^{-1}\right)d(\mu_{L}\circ st_{\hat{d}}^{-1})\\
 & =\int_{\tilde{B}}\left(\mathbb{E}_{st_{\hat{d}}^{-1}}(st\circ f^{\omega})\circ st_{\hat{d}}^{-1}\right)d\mu.
\end{align*}
Hence, for all Borel $\tilde{B}\subseteq supp(\mu)$, 
\[
(\nu_{L}\circ st_{\hat{d}}^{-1})(\tilde{B})=\int_{\tilde{B}}\left(\mathbb{E}_{st_{\hat{d}}^{-1}}(st\circ f^{\omega})\circ st_{\hat{d}}^{-1}\right)d\mu.
\]
Hence, if $\nu^{\omega}\ll^{\omega}\mu^{\omega}$ with ``internal
Radon-Nikodym derivative'' $f^{\omega}$, and $f^{\omega}$ is $S$-integrable
with respect to $\mu^{\omega}$, and $\mu_{L}$ pushes forward to
a Radon measure $\mu$ on $\hat{X}$, it follows that $\nu_{L}$ also
pushes forward to a Radon measure $\nu$ on $\hat{X}$, and $\nu\ll\mu$
with $\frac{d\nu}{d\mu}=\mathbb{E}_{st_{\hat{d}}^{-1}}(st\circ f^{\omega})\circ st_{\hat{d}}^{-1}$. 
\end{proof}
The next two lemmas are not stated using sharp assumptions, but rather
using assumptions which are simply strong enough to run the argument
for Theorem \ref{thm:main}.
\begin{lem}
\label{lem:lifting with integral functional}(Stability of integral
functionals under lifting) Let $\mathcal{F}(\nu):\mathcal{P}(\hat{X})\rightarrow\mathbb{R}$
be an integral functional of the form
\[
\mathcal{F}(\nu)=\begin{cases}
\int_{\hat{X}}\varphi\left(\frac{d\nu}{d\mu}\right)d\mu & \nu\ll\mu\\
\infty & \nu\not\ll\mu
\end{cases}
\]
where $\varphi:\mathbb{R}_{+}\rightarrow\mathbb{R}$ is some continuous
function such that %
{} $\varphi(x)\leq cx^{\alpha}+d$ for constants $c,d,\alpha>0$. Let
$\varphi^{\omega}:=[(\varphi)]$ be the hyperreal extension of $\varphi$,
let $\mu^{\omega}$ be an internal probability measure whose associated
Loeb measure $\mu_{L}$ pushes forward to a Radon probability measure
$\mu$ on $\hat{X}$, and let $\nu\ll\mu$. Suppose that $\frac{d\nu}{d\mu}\in L^{q}(\mu)$
where $q>\alpha$. There exists an $S$-integrable lifting of $\frac{d\nu}{d\mu}\circ st_{\hat{d}}$,
denoted $H^{\omega}$, and for any such lifting,%
\[
\mathcal{F}(\nu)\approx\int_{X^{\omega}}\varphi^{\omega}\left(H^{\omega}\right)d\mu^{\omega}.
\]
\end{lem}

\begin{proof}
Let $h(x)=\frac{d\nu}{d\mu}(x)$. By change of variables, it follows
from the fact that $\mu_{L}\circ st_{\hat{d}}^{-1}=\mu$ that 
\[
\int_{\hat{X}}\varphi\left(h(x)\right)d\mu(x)=\int_{X^{\omega}}\varphi\left(h(st_{\hat{d}}(y))\right)d\mu_{L}(y).
\]
In particular the function $h\circ st_{\hat{d}}$ is $\mu_{L}$-measurable,
so there exists an internal function $H^{\omega}:X^{\omega}\rightarrow\mathbb{R}^{\omega}$
which is an $S$-integrable lifting of $h\circ st_{\hat{d}}$.%
{} In particular $H^{\omega}\approx h\circ st_{\hat{d}}$ $\mu_{L}$-almost
surely. Moreover, since $h\in L^{q}(\mu)$, it follows by change of
variables that $h\circ st_{\hat{d}}\in L^{q}(\mu_{L})$, and therefore
$st\left(\int|H^{\omega}|^{p}d\mu^{\omega}\right)<\infty$ by Fact
\ref{fact:S-integrable lifting}. At the same time, $\varphi:\mathbb{R}\rightarrow\mathbb{R}$
is continuous, so $\varphi^{\omega}$ has the property that $\varphi^{\omega}(x)=\varphi(x)$
for all (standard) $x\in\mathbb{R}$, and additionally $\varphi^{\omega}(x)\approx\varphi^{\omega}(y)$
whenever $x\approx y$ \cite[Theorem 7.1.1]{goldblatt2012lectures}.
Therefore, for $\mu_{L}$-almost all $y$,
\[
\varphi\left(h(st(y))\right)=st(\varphi^{\omega}(H^{\omega}(y))).
\]
It follows that
\[
\int_{X^{\omega}}\varphi\left(h(st_{\hat{d}}(y))\right)d\mu_{L}(y)=\int_{X^{\omega}}st\left(\varphi^{\omega}(H^{\omega}(y))\right)d\mu_{L}(y).
\]
Lastly, let us check that $\varphi^{\omega}\circ H^{\omega}$ is $S$-integrable.
Using our polynomial bound on $\varphi$, we deduce that 
\[
\varphi^{\omega}\circ H^{\omega}\leq c(H^{\omega})^{\alpha}+d.
\]
It follows that $st\left(\int|\varphi^{\omega}\circ H^{\omega}|^{q/\alpha}d\mu^{\omega}\right)<\infty$,
which implies (thanks to Fact \ref{fact:p>1 S-integrable}) that since
$q/\alpha>1$, that%
{} $\varphi\circ H^{\omega}$ is $S$-integrable.%

Finally, it follows from $S$-integrability of $\varphi^{\omega}(H^{\omega}(y))$,
and the fact that $\mu_{L}$ is a finite measure, that
\[
\int_{X^{\omega}}st\left(\varphi^{\omega}(H^{\omega}(y))\right)d\mu_{L}(y)\approx\int_{X^{\omega}}\varphi^{\omega}(H^{\omega}(y))d\mu^{\omega}(y).
\]
Hence $\int_{\hat{X}}\varphi\left(h(x)\right)d\mu(x)\approx\int_{X^{\omega}}\varphi^{\omega}(H^{\omega}(y))d\mu^{\omega}(y)$
as desired.
\end{proof}
We also have a sort of converse:
\begin{lem}
\label{lem:pushdown with integral functional}(Stability of internal
integral functionals under pushdown) Let $\varphi:\mathbb{R}_{+}\rightarrow\mathbb{R}_{+}$
be a continuous, \emph{convex} function, such that $\varphi(x)\leq cx^{\alpha}+d$
for constants $c,d,\alpha>0$. Let $\varphi^{\omega}:=[(\varphi)]$
denote the hyperreal extension of $\varphi$. Let $\mu^{\omega}\in\mathcal{P}(X)^{\omega}$
be an internal probability measure whose associated Loeb measure $\mu_{L}$
pushes forward to a Radon probability measure $\mu$ on $\hat{X}$.
Define
\[
\mathcal{F}^{\omega}(\nu^{\omega}):=\begin{cases}
\int\varphi^{\omega}\circ f^{\omega}d\mu^{\omega} & \nu^{\omega}\ll^{\omega}\mu^{\omega};\nu^{\omega}=f^{\omega}d\mu^{\omega}\\
\infty^{\omega} & \text{otherwise}.
\end{cases}
\]
(Here $\infty^{\omega}$ denotes a point formally added to $\mathbb{R}^{\omega}$
such that $x^{\omega}<\infty^{\omega}$ for all $x^{\omega}\in\mathbb{R}^{\omega}$.) 

Then, if $st\left(\int|f^{\omega}|^{q}d\mu^{\omega}\right)<\infty$
for some $q>\max\{1,\alpha\}$, it follows that $\nu_{L}:=\text{Loeb}(\nu^{\omega})$
pushes forward to some Radon probability measure $\nu$ on $\hat{X}$,
and $\nu\ll\mu$; and %
\[
\mathcal{F}(\nu):=\int_{\hat{X}}\varphi\left(\frac{d\nu}{d\mu}\right)d\mu\lessapprox\mathcal{F}^{\omega}(\nu^{\omega}).
\]
\end{lem}

\begin{proof}
Suppose that $\nu^{\omega}=f^{\omega}d\mu^{\omega}$, and $st\left(\int|f^{\omega}|^{q}d\mu^{\omega}\right)<\infty$
for some $q>1$ (and therefore $f^{\omega}$ is $S$-integrable w.r.t.
$\mu^{\omega}$, by Fact \ref{fact:p>1 S-integrable}), and that $\mu_{L}\circ st_{\hat{d}}^{-1}$
is well-defined and a Radon probability measure on $\hat{X}$. By
Lemma \ref{lem:pushdown of density}, it follows that $\nu:=\nu_{L}\circ st_{\hat{d}}^{-1}$
is also well-defined and a Radon probability measure on $\hat{X}$,
and $\nu\ll\mu$. 

Since $\varphi^{\omega}\circ f^{\omega}\leq c(f^{\omega})^{\alpha}+d$,
it follows, as in the proof of Lemma \ref{lem:lifting with integral functional}
that $\varphi^{\omega}\circ f^{\omega}\in SL^{q/\alpha}(\mu^{\omega})$
with $q/\alpha>1$, and in particular $\varphi^{\omega}\circ f^{\omega}$
is $S$-integrable. 

Therefore, 
\begin{align*}
st\left(\int\varphi^{\omega}\circ f^{\omega}d\mu^{\omega}\right) & =\int st(\varphi^{\omega}\circ f^{\omega})d\mu_{L}\\
 & =\int\varphi\circ st(f^{\omega})d\mu_{L}.
\end{align*}
Now, $\varphi\circ st(f^{\omega})$ is $\sigma(\mathcal{B}^{\omega})$-measurable.
As before, in Lemma \ref{lem:pushdown of density}, we take the conditional
expectation $\mathbb{E}_{st_{\hat{d}}^{-1}}\varphi\circ st(f^{\omega})$;
since $\varphi$ is convex, Jensen's inequality%
{} for conditional expectations \cite[Proposition 10.1.9]{bogachev2007measure2}
shows that $\varphi\circ\mathbb{E}_{st_{\hat{d}}^{-1}}st(f^{\omega})\leq\mathbb{E}_{st_{\hat{d}}^{-1}}\varphi\circ st(f^{\omega})$
except on a set of $\mu_{L}$-measure zero. Thus, 
\begin{align*}
\int\varphi\circ\mathbb{E}_{st_{\hat{d}}^{-1}}st(f^{\omega})d\mu_{L} & \leq\int\mathbb{E}_{st_{\hat{d}}^{-1}}\varphi\circ st(f^{\omega})d\mu_{L}\\
 & =\int\varphi\circ st(f^{\omega})d\mu_{L}\\
 & =\int st(\varphi^{\omega}\circ f^{\omega})d\mu_{L}\\
 & \approx\mathcal{F}^{\omega}(\nu^{\omega}).
\end{align*}
Performing a change of variables, as in Lemma \ref{lem:pushdown of density},
we see 
\[
\int_{X^{\omega}}\varphi\circ\mathbb{E}_{st_{\hat{d}}^{-1}}st(f^{\omega})d\mu_{L}=\int_{\hat{X}}\varphi\circ\left(\mathbb{E}_{st_{\hat{d}}^{-1}}st(f^{\omega})\circ st_{\hat{d}}^{-1}\right)d\mu_{L}\circ st_{\hat{d}}^{-1}.
\]
Since $\mu=\mu_{L}\circ st_{\hat{d}}^{-1}$ and $\frac{d\nu}{d\mu}=\mathbb{E}_{st_{\hat{d}}^{-1}}st(f^{\omega})\circ st_{\hat{d}}^{-1}$,
this proves the lemma.
\end{proof}
\begin{rem*}
In particular, we will apply Lemmas \ref{lem:lifting with integral functional}
and \ref{lem:pushdown with integral functional} in the case where
$\varphi(x)=x\log x-x+1$. This is a convex function from $\mathbb{R}_{+}$
to $\mathbb{R}_{+}$ which is polynomially bounded, and with this
choice of $\varphi$, we have (assuming that $\nu$ and $\mu$ are
probability measures)
\begin{align*}
\mathcal{F}(\nu) & =\int\left(\frac{d\nu}{d\mu}\log\frac{d\nu}{d\mu}-\frac{d\nu}{d\mu}+1\right)d\mu\\
 & =\int\frac{d\nu}{d\mu}\log\frac{d\nu}{d\mu}d\mu=H(\nu\mid\mu).
\end{align*}
\end{rem*}

\begin{lem}
\label{lem:special liftings of densities}(special liftings of bounded
densities) Let $\mu\in\mathcal{P}_{2}(\hat{X})$, and let $\nu\in\mathcal{P}_{2,ac(\mu)}(\hat{X})$
have density $f=\frac{d\nu}{d\mu}$. Suppose that $f$ is bounded,
and $E=supp(f)\ensuremath{\subseteq}supp(\mu)$ has bounded diameter.
Then, given any $\mu^{\omega}\in\mathcal{P}_{2}(X)^{\omega}$ %
{} such that $\text{Loeb}(\mu^{\omega})$ pushes forward to $\mu$,
there exists an $S$-integrable lifting $\breve{f}^{\omega}$ of $f\circ st_{\hat{d}}:st_{\hat{d}}^{-1}(\text{supp}(\nu))\rightarrow\mathbb{R}$
such that the internal measure 
\[
\nu^{\omega}(A):=\int_{A}\breve{f}^{\omega}d\mu^{\omega}\quad\forall A\in\mathcal{B}^{\omega}
\]
is an internal probability measure, belonging to $\mathcal{P}_{2}(X)_{\lim}^{\omega}$,
and such that the $\hat{W}_{2}$-pushdown of $\nu^{\omega}$ is $\nu$. 
\end{lem}

\begin{proof}
\textbf{}This can be deduced from Lemma \ref{lem:bounded diameter pushfwd lifting}
together with Lemma \ref{lem:Radon-Nikodym Loeb}.

Let $\mu_{L}$ denote $\text{Loeb}(\mu^{\omega})$ and suppose $\mu=\mu_{L}\circ st_{\hat{d}}^{-1}$.
By change of variables, we know that for all Borel $B\subset supp(\mu)\subseteq\hat{X}$,
\[
\nu(B)=\int_{B}fd(\mu_{L}\circ st_{\hat{d}}^{-1})=\int_{st_{\hat{d}}^{-1}(B)}(f\circ st_{\hat{d}})d\mu_{L}.
\]
From the fact that $f\circ st_{\hat{d}}$ is $\mu_{L}$-integrable
and bounded, we know that there exists a lifting $f^{\omega}$ of
$f\circ st_{\hat{d}}$ such that $f^{\omega}$ is bounded except on
a set of $\mu_{L}$-measure zero. This is not quite enough for us,
so we tweak $f^{\omega}$ in several ways. 

First, we restrict $f^{\omega}$ to be strictly positive only on a
set of finite diameter. (This is not guaranteed by the fact that $f\circ st_{\hat{d}}$
is strictly positive only on $st_{\hat{d}}^{-1}(E)$! Away from this
set, it may be the case that $f^{\omega}$ takes nonzero infinitesimal
values.) So let $E^{\omega}$ be an internal set of bounded diameter
containing $st_{\hat{d}}^{-1}(E)$ (such a set exists since $E$,
and therefore $st_{\hat{d}}^{-1}(E)$, has bounded diameter); define
\[
f^{\omega}\upharpoonright_{E^{\omega}}(x^{\omega})=\begin{cases}
f^{\omega}(x^{\omega}) & x^{\omega}\in E^{\omega}\\
0 & \text{else}.
\end{cases}
\]
We observe that $f^{\omega}\upharpoonright_{E^{\omega}}$ is itself
an internal function, since in particular it has the representation
$f^{\omega}\upharpoonright_{E^{\omega}}=\left[\left(f_{i}\upharpoonright_{E_{i}}\right)\right]$,
where $f^{\omega}=[(f_{i})]$ and $E^{\omega}=[(E_{i})]$. Moreover,
since $f^{\omega}\approx f^{\omega}\upharpoonright_{E^{\omega}}$
pointwise, we know that $f^{\omega}\upharpoonright_{E^{\omega}}$
is also an $S$-integrable lifting of $f\circ st_{\hat{d}}$.

 Next, let $C=\Vert f\Vert_{L^{\infty}(\mu)}$. Define
\[
\tilde{f}^{\omega}(x)=\begin{cases}
f^{\omega}\upharpoonright_{E^{\omega}}(x) & f^{\omega}\upharpoonright_{E^{\omega}}(x)\leq2C\\
2C & \text{otherwise}.
\end{cases}
\]
Note that $\tilde{f}^{\omega}$ is internally defined, by the same
reasoning as for $f^{\omega}\upharpoonright_{E^{\omega}}$. It is
also bounded (not even just bounded $\mu^{\omega}$-a.s.); and obviously
differs with $f^{\omega}$ only on a $\mu_{L}$-null set. It is therefore
also a lifting of $f\circ st_{\hat{d}}$. And since $0\leq\tilde{f}^{\omega}\leq f^{\omega}$
pointwise, and $f^{\omega}$ is $S$-integrable, we know that $\tilde{f}^{\omega}$
is also $S$-integrable. 

Now, %
it follows, since $\tilde{f}^{\omega}$ is a lifting of $f\circ st_{\hat{d}}$,
that 
\[
\int\tilde{f}^{\omega}d\mu^{\omega}\approx\int f\circ st_{\hat{d}}d\mu_{L}=1.
\]
Therefore, we modify $\tilde{f}^{\omega}$ once more, by setting 
\[
\breve{f}^{\omega}:=\frac{1}{\int\tilde{f}^{\omega}d\mu^{\omega}}\tilde{f}^{\omega}.
\]
With this normalization it is clear that $\breve{f}^{\omega}d\mu^{\omega}$
is an internal probability measure, and since $\breve{f}^{\omega}\approx\tilde{f}^{\omega}$
pointwise it is still the case that $\breve{f}^{\omega}$ is an $S$-integrable
lifting of $f\circ st_{\hat{d}}$. 

Finally, let $\nu^{\omega}:=\breve{f}^{\omega}d\mu^{\omega}$. Since
$\nu^{\omega}$ is an internal probability measure with bounded support,
it automatically belongs to $W_{p}(X)_{\lim}^{\omega}$ for every
$p\in[1,\infty)$, in particular for $p=2$. It follows from Lemma
\ref{lem:Radon-Nikodym Loeb} (3) that 
\[
\nu_{L}=\int st\circ\breve{f}^{\omega}d\mu_{L}=\int f\circ st_{\hat{d}}d\mu_{L}.
\]
Likewise, 
\[
\nu_{L}\circ st_{\hat{d}}^{-1}=\int fd(\mu_{L}\circ st_{\hat{d}}^{-1})=\int fd\mu=\nu
\]
so since $\nu^{\omega}$ has bounded support, we know from Lemma \ref{lem:bounded diameter pushfwd lifting}
that $\nu^{\omega}$ is automatically a $\hat{W}_{2}$-lifting of
$\nu$. Hence, the proof is complete.

\end{proof}
The preceding lemma gets us most of the way to deducing the main theorem
of this section: %

\begin{proof}[Proof of Theorem \ref{thm:main}]
 To reduce notational burden, we suppress reference to the pointed
metric measure isomorphism referred to in Definition \ref{def:(ultralimit-of-pmm-space},
and work only with the ultralimit metric space $(\hat{X},\hat{d})$
equipped with a reference measure $\mu$. No loss of generality occurs
since by \cite[Proposition 4.12]{sturm2006geometry1}, the $CD(K,\infty)$
property is invariant with respect to pmm isomorphism.

Suppose that $\omega$-almost all the $(X_{i},d_{i},\mu_{i})$'s are
$CD(K_{i},\infty)$, and suppose that the Loeb measure $\mu_{L}$
associated to $\mu^{\omega}=[(\mu_{i})]$ pushes forward to a Radon
measure on $\hat{X}$, that is, $\mu:=\mu_{L}\circ st_{\hat{d}}^{-1}\in\mathcal{P}(\hat{X})$.
Assume that $st(K^{\omega})\neq-\infty$ otherwise there is nothing
to prove. Let $v_{0},v_{1}\in\mathcal{P}_{2,ac(\mu)}(\hat{X})$, with
$\mu$-densities $f_{0}$ and $f_{1}$.%
{} By \cite[Remark 4.6]{sturm2006geometry1}, it suffices to show that
\[
H(\nu_{\frac{1}{2}}\mid\mu)\leq\frac{1}{2}H(\nu_{0}\mid\mu)+\frac{1}{2}H(\nu_{1}\mid\mu)-\frac{1}{8}st(K^{\omega})W_{2}^{2}(\nu_{0},\nu_{1}).
\]

\textbf{Step 1}. First, let us proceed under the assumption that both
$f_{0}$ and $f_{1}$ are uniformly bounded, and also have compact
support. Let $f_{0}^{\omega}$ and $f_{1}^{\omega}$ be liftings of
$f_{0}\circ st_{\hat{d}}$ and $f_{1}\circ st_{\hat{d}}$ respectively,
chosen according to Lemma \ref{lem:special liftings of densities}.

Now, consider sequences $(f_{0,i})_{i\in\mathbb{N}}$ and $(f_{1,i})_{i\in\mathbb{N}}$
of functions from $X_{i}$ to $\mathbb{R}_{+}$, such that $[(f_{0,i})]=f_{0}^{\omega}$
and $[(f_{1,i})]=f_{1}^{\omega}$. Note that $f_{0,i}d\mu_{i}$ and
$f_{1,i}d\mu_{i}$ belong to $\mathcal{P}_{2,ac(\mu_{i})}(X_{i})$
and have bounded densities for $\omega$-almost all $i\in\mathbb{N}$.
From the fact that $X_{i}$ is $CD(K_{i},\infty)$, there exists a
$W_{2}$-geodesic connecting $f_{0,i}d\mu_{i}$ and $f_{1,i}d\mu_{i}$
along which the relative entropy $H(\cdot\mid\mu_{i})$ is $K_{i}$-convex;
note that automatically, this geodesic must have density with respect
to $\mu_{i}$ for all $t\in[0,1]$ (otherwise $H(\cdot\mid\mu_{i})$
blows up to $\infty$ at some intermediate time, contradicting $K_{i}$-convexity).
Let $(f_{t,i}d\mu_{i})_{t\in[0,1]}$ denote this distinguished constant
speed geodesic; observe that %
{} the estimate of Rajala from \cite[Theorem 1.3]{rajala2012interpolated}
shows that for $\omega$-almost all $i\in\mathbb{N}$,
\[
\Vert f_{t,i}\Vert_{L^{\infty}(\mu_{i})}\leq e^{K_{i}^{-}D_{i}^{2}/12}(\Vert f_{0,i}\Vert_{L^{\infty}(\mu_{i})}+\Vert f_{1,i}\Vert_{L^{\infty}(\mu_{i})})
\]
where $D_{i}=diam(supp(f_{0,i})\cup supp(f_{1,i}))$. 

By \L o\'{s}'s theorem (in particular Lemma \ref{lem:Radon-Nikodym Loeb}
(2)), this implies %
{} that $f_{\frac{1}{2}}^{\omega}:=[(f_{\frac{1}{2},i})]$ is bounded
by $e^{K^{-^{\omega}}(D^{\omega})^{2}/12}(\max f_{0}^{\omega}+\max f_{1}^{\omega})$,
except on a set of $\mu^{\omega}$-measure zero (unless we are in
the degenerate case where $st(K^{\omega})=-\infty$). Note that 
\begin{align*}
D^{\omega} & =diam(supp(f_{0}^{\omega})\cup supp(f_{1}^{\omega}))<\infty
\end{align*}
since the liftings $f_{0}^{\omega}$ and $f_{1}^{\omega}$ have been
selected in accordance with Lemma \ref{lem:special liftings of densities}.
It follows that 
\[
st\left(\Vert f_{\frac{1}{2}}^{\omega}\Vert_{L^{\infty}(\mu^{\omega})}\right)\leq st\left(e^{-K^{\omega^{-}}D^{\omega^{2}}/12}\left(\Vert f_{0}^{\omega}\Vert_{L^{\infty}(\mu^{\omega})}+\Vert f_{1}^{\omega}\Vert_{L^{\infty}(\mu^{\omega})}\right)\right)<\infty.
\]
This implies that $f_{\frac{1}{2}}^{\omega}$ is $S$-integrable,
so that $st\circ f_{\frac{1}{2}}^{\omega}$ is $\mu_{L}$-integrable.
In particular, Lemma \ref{lem:Radon-Nikodym Loeb} (3) allows us to
deduce that 
\[
\forall B\in\sigma(\mathcal{B}^{\omega})\quad\text{Loeb}(f_{\frac{1}{2}}^{\omega}d\mu^{\omega})(B)=\int_{B}st\circ f_{\frac{1}{2}}^{\omega}d\mu_{L}.
\]
And, from Lemma \ref{lem:pushdown of density}, we know that $\text{Loeb}(f_{\frac{1}{2}}^{\omega}d\mu^{\omega})\circ st_{\hat{d}}^{-1}$
is well-defined (and moreover $\text{Loeb}(f_{\frac{1}{2}}^{\omega}d\mu^{\omega})\circ st_{\hat{d}}^{-1}\ll\mu_{L}\circ st_{\hat{d}}^{-1}=\mu$). 

At the same time,%
{} we know that 
\[
W_{2}(f_{0,i}d\mu_{i},f_{\frac{1}{2},i}d\mu_{i})=W_{2}(f_{\frac{1}{2},i}d\mu_{i},f_{1,i}d\mu_{i})=\frac{1}{2}W_{2}(f_{0,i}d\mu_{i},f_{1,i}d\mu_{i})\quad\omega\text{-a.s.}
\]
and hence 
\[
W_{2}^{\omega}(f_{0}^{\omega}d\mu^{\omega},f_{\frac{1}{2}}^{\omega}d\mu^{\omega})=W_{2}^{\omega}(f_{\frac{1}{2}}^{\omega}d\mu^{\omega},f_{1}^{\omega}d\mu^{\omega})=\frac{1}{2}W_{2}^{\omega}(f_{0}^{\omega}d\mu^{\omega},f_{1}^{\omega}d\mu^{\omega}).
\]
On the one hand, since $f_{0}^{\omega}d\mu^{\omega}$ and $f_{1}^{\omega}d\mu^{\omega}$
are $\hat{W}_{2}$-liftings of $f_{0}d\mu$ and $f_{1}d\mu$, we know
that 
\[
W_{2}^{\omega}(f_{0}^{\omega}d\mu^{\omega},f_{1}^{\omega}d\mu^{\omega})\approx W_{2}(f_{0}d\mu,f_{1}d\mu).
\]
On the other hand, it follows from Propositions \ref{prop:Wp Loeb contraction 1}
and \ref{prop:Wp Loeb contraction 2} that 
\[
W_{2}(f_{0}d\mu,\text{Loeb}(f_{\frac{1}{2}}^{\omega}d\mu^{\omega})\circ st_{\hat{d}}^{-1})\lessapprox W_{2}^{\omega}(f_{0}^{\omega}d\mu^{\omega},f_{\frac{1}{2}}^{\omega}d\mu^{\omega})
\]
and 
\[
W_{2}(\text{Loeb}(f_{\frac{1}{2}}^{\omega}d\mu^{\omega})\circ st_{\hat{d}}^{-1},f_{1}d\mu)\lessapprox W_{2}^{\omega}(f_{\frac{1}{2}}^{\omega}d\mu^{\omega},f_{1}^{\omega}d\mu^{\omega}).
\]
It follows that 
\[
W_{2}(f_{0}d\mu,\text{Loeb}(f_{\frac{1}{2}}^{\omega}d\mu^{\omega})\circ st_{\hat{d}}^{-1})\lessapprox\frac{1}{2}W_{2}(f_{0}d\mu,f_{1}d\mu)
\]
and 
\[
W_{2}(\text{Loeb}(f_{\frac{1}{2}}^{\omega}d\mu^{\omega})\circ st_{\hat{d}}^{-1},f_{1}d\mu)\lessapprox\frac{1}{2}W_{2}(f_{0}d\mu,f_{1}d\mu).
\]
But in these last two displayed equations, all the quantities are
real. Hence 
\[
W_{2}(f_{0}d\mu,\text{Loeb}(f_{\frac{1}{2}}^{\omega}d\mu^{\omega})\circ st_{\hat{d}}^{-1})\leq\frac{1}{2}W_{2}(f_{0}d\mu,f_{1}d\mu)
\]
and 
\[
W_{2}(\text{Loeb}(f_{\frac{1}{2}}^{\omega}d\mu^{\omega})\circ st_{\hat{d}}^{-1},f_{1}d\mu)\leq\frac{1}{2}W_{2}(f_{0}d\mu,f_{1}d\mu)
\]
which implies that $\text{Loeb}(f_{\frac{1}{2}}^{\omega}d\mu^{\omega})\circ st_{\hat{d}}^{-1}$
must be a $W_{2}$-midpoint between $f_{0}d\mu$ and $f_{1}d\mu$.
So henceforth, we denote $\nu_{\frac{1}{2}}:=\text{Loeb}(f_{\frac{1}{2}}^{\omega}d\mu^{\omega})\circ st_{\hat{d}}^{-1}$
and write $f_{\frac{1}{2}}=\frac{d\nu_{\frac{1}{2}}}{d\mu}$.

Now, returning to the spaces $(\mathcal{P}(X_{i}),W_{2})$, denote
$\nu_{t,i}:=f_{t,i}d\mu_{i}$ for $t=0,\frac{1}{2},1$. Observe that
for each $i$, the $CD(K_{i},\infty)$ property implies that 
\[
H(\nu_{\frac{1}{2},i}\mid\mu_{i})\leq\frac{1}{2}H(v_{0,i}\vert\mu_{i})+\frac{1}{2}H(v_{1,i}\vert\mu_{i})-\frac{1}{8}K_{i}W_{2,i}^{2}(\nu_{0,i},\nu_{1,i}).
\]
From this, we deduce that in the space $(\mathcal{P}(X)^{\omega},W_{2}^{\omega})$,
\[
H^{\omega}([(\nu_{\frac{1}{2},i})]\mid[(\mu_{i})])\leq\frac{1}{2}H^{\omega}([(v_{0,i})]\mid[(\mu_{i})])+\frac{1}{2}H^{\omega}([(v_{1,i})]\mid[(\mu_{i})])-\frac{1}{8}K^{\omega}(W_{2}^{\omega})^{2}([(\nu_{0,i})],[(\nu_{1,i})]).
\]
Note that 
\[
H^{\omega}([(v_{0,i})]\mid[(\mu_{i})])=\int f_{0}^{\omega}\log^{\omega}f_{0}^{\omega}d\mu^{\omega}\text{ and }H^{\omega}([(v_{1,i})]\mid[(\mu_{i})])=\int f_{1}^{\omega}\log^{\omega}f_{1}^{\omega}d\mu^{\omega}
\]
and so, by Lemma \ref{lem:lifting with integral functional}, we have
that 
\[
H^{\omega}([(v_{0,i})]\mid[(\mu_{i})])\approx H(\nu_{0}\mid\mu)\text{ and }H^{\omega}([(v_{1,i})]\mid[(\mu_{i})])\approx H(\nu_{1}\mid\mu).
\]
At the same time, $(W_{2}^{\omega})([(\nu_{0,i})],[(\nu_{1,i})])\approx W_{2}(\nu_{0},\nu_{1})$
since $[(\nu_{0,i})]$ and $[(\nu_{1,i})]$ are liftings of $\nu_{0}$
and $\nu_{1}$. Lastly, $H^{\omega}([(v_{\frac{1}{2},i})]\vert[(\mu_{i})])\gtrapprox H(\nu_{\frac{1}{2}}\mid\mu)$,
from Lemma \ref{lem:pushdown with integral functional}.

Therefore, since 
\[
H^{\omega}([(\nu_{\frac{1}{2},i})]\mid[(\mu_{i})])\leq\frac{1}{2}H^{\omega}([(v_{0,i})]\mid[(\mu_{i})])+\frac{1}{2}H^{\omega}([(v_{1,i})]\mid[(\mu_{i})])-K^{\omega}\frac{1}{8}(W_{2}^{\omega})^{2}([(\nu_{0,i})],[(\nu_{1,i})])
\]
and: 
\begin{enumerate}
\item $H^{\omega}([(v_{0,i})]\mid[(\mu_{i})])\approx H(\nu_{0}\mid\mu)$, 
\item $H^{\omega}([(v_{1,i})]\mid[(\mu_{i})])\approx H(\nu_{1}\mid\mu)$, 
\item $H^{\omega}([(v_{\frac{1}{2},i})]\mid[(\mu_{i})])\gtrapprox H(\nu_{\frac{1}{2}}\mid\mu)$, 
\item $(W_{2}^{\omega})([(\nu_{0,i})],[(\nu_{1,i})])\approx W_{2}(\nu_{0},\nu_{1})$,
and 
\item $K^{\omega}\approx st(K^{\omega})$, 
\end{enumerate}
we deduce that 
\[
H(\nu_{\frac{1}{2}}\mid\mu)\lessapprox\frac{1}{2}H(\nu_{0}\mid\mu)+\frac{1}{2}H(\nu_{1}\mid\mu)-\frac{1}{8}st(K^{\omega})W_{2}^{2}(\nu_{0},\nu_{1}).
\]
But everything in the above expression is a $\mathbb{R}$-valued function/quantity,
so in fact 
\[
H(\nu_{\frac{1}{2}}\mid\mu)\leq\frac{1}{2}H(\nu_{0}\mid\mu)+\frac{1}{2}H(\nu_{1}\mid\mu)-\frac{1}{8}st(K^{\omega})W_{2}^{2}(\nu_{0},\nu_{1})
\]
as desired.

\textbf{Step 2}. The argument we have just given works under the assumption
that $\nu_{0}$ and $\nu_{1}$ have bounded densities and compact
support. We now give an approximation argument that extends the result
to general $\nu_{0},\nu_{1}\in\mathcal{P}_{2,ac(\mu)}(\hat{X})$. 

Let $\nu_{0},\nu_{1}\in\mathcal{P}_{2,ac(\mu)}(\hat{X})$ with densities
$f_{0}$ and $f_{1}$. Let $(\kappa_{0,m})_{m\in\mathbb{N}}$ and
$(\kappa_{1,m})$ be sequences of compact sets contained in $supp(\nu_{0})$
and $supp(\nu_{1})$ respectively, such that 
\[
\bigcup_{m\in\mathbb{N}}\kappa_{0,m}=supp(\nu_{0})\text{ and }\bigcup_{m\in\mathbb{N}}\kappa_{1,m}=supp(\nu_{1}).
\]
Define 
\[
(f_{0}\upharpoonright m)(x):=\begin{cases}
\min\{m,f_{0}(x)\} & x\in\kappa_{0,m}\\
0 & x\notin\kappa_{0,m}
\end{cases}
\]
and similarly
\[
(f_{1}\upharpoonright m)(x):=\begin{cases}
\min\{m,f_{1}(x)\} & x\in\kappa_{1,m}\\
0 & x\notin\kappa_{1,m}
\end{cases}.
\]
In other words, we both cut off the density at height $m$, and restrict
to some compact set, simultaneously. We also use the notation $\mu(f_{0}\upharpoonright m):=\int_{\hat{X}}(f_{0}\upharpoonright m)d\mu$
and similarly for $f_{1}$. Without loss of generality, we assume
$\mu(f_{0}\upharpoonright1)$ and $\mu(f_{1}\upharpoonright1)$ are
both strictly positive. 

Note that 
\[
H((f_{0}\upharpoonright m)d\mu\mid\mu)=\int_{\hat{X}}(f_{0}\upharpoonright m)\log(f_{0}\upharpoonright m)d\mu
\]
still makes sense even though the density $f_{0}\upharpoonright m$
does not integrate to 1. Clearly $f_{0}\upharpoonright m$ converges
pointwise monotonically to $f_{0}$ as $m\rightarrow\infty$, and
similarly for $f_{1}$. By \cite[Lemma 5.1]{ambrosio2015riemannian},
this implies that 
\[
H((f_{0}\upharpoonright m)d\mu\mid\mu)\rightarrow H(\nu_{0}\mid\mu)\text{ and }H((f_{1}\upharpoonright m)d\mu\mid\mu)\rightarrow H(\nu_{1}\mid\mu).
\]
Therefore, 
\[
\frac{1}{\mu(f_{0}\upharpoonright m)}H((f_{0}\upharpoonright m)d\mu\mid\mu)\rightarrow H(\nu_{0}\mid\mu)\text{ and }\frac{1}{\mu(f_{1}\upharpoonright m)}H((f_{1}\upharpoonright m)d\mu\mid\mu)\rightarrow H(\nu_{1}\mid\mu)
\]
and so 
\[
H\left(\frac{1}{\mu(f_{0}\upharpoonright m)}(f_{0}\upharpoonright m)d\mu\mid\mu\right)\rightarrow H(\nu_{0}\mid\mu)\text{ and }H\left(\frac{1}{\mu(f_{1}\upharpoonright m)}(f_{1}\upharpoonright m)d\mu\mid\mu\right)\rightarrow H(\nu_{1}\mid\mu).
\]

At the same time, the monotone convergence theorem implies that 
\[
\int_{\hat{X}}f_{0}\upharpoonright md\mu\rightarrow\int f_{0}d\mu\text{ and }\text{\ensuremath{\int_{\hat{X}}f_{1}\upharpoonright md\mu\rightarrow\int f_{1}d\mu}}
\]
and hence 
\[
\frac{1}{\mu(f_{0}\upharpoonright m)}\int_{\hat{X}}f_{0}\upharpoonright md\mu\rightarrow f_{0}d\mu\text{ and }\frac{1}{\mu(f_{1}\upharpoonright m)}\text{\ensuremath{\int_{\hat{X}}f_{1}\upharpoonright md\mu\rightarrow\int f_{1}d\mu.}}
\]
All the terms in the sequences 
\[
\left(\frac{1}{\mu(f_{0}\upharpoonright m)}f_{0}\upharpoonright m\right)_{m\in\mathbb{N}}\text{ and }\left(\frac{1}{\mu(f_{1}\upharpoonright m)}f_{1}\upharpoonright m\right)_{m\in\mathbb{N}}
\]
are dominated by $\frac{1}{\mu(f_{0}\upharpoonright1)}f_{0}$ and
$\frac{1}{\mu(f_{1}\upharpoonright1)}f_{1}$ respectively. So by the
dominated convergence theorem, 
\[
\int_{\hat{X}}\left|f_{0}-\frac{1}{\mu(f_{0}\upharpoonright m)}f_{0}\upharpoonright m\right|d\mu\rightarrow0\text{ and }\int_{\hat{X}}\left|f_{1}-\frac{1}{\mu(f_{1}\upharpoonright m)}f_{1}\upharpoonright m\right|d\mu\rightarrow0.
\]
This implies that 
\[
\frac{1}{\mu(f_{0}\upharpoonright m)}f_{0}\upharpoonright md\mu\rightharpoonup f_{0}d\mu\text{ and }\frac{1}{\mu(f_{1}\upharpoonright m)}f_{1}\upharpoonright md\mu\rightharpoonup f_{1}d\mu.
\]
Therefore, to show that $W_{2}(f_{0},\frac{1}{\mu(f_{0}\upharpoonright m)}f_{0}\upharpoonright m)\rightarrow0$
and similarly for $f_{1}$, it suffices to check that we also have
convergence of 2nd moments. But this is clear, since the monotone
convergence theorem implies that 
\[
\int_{\hat{X}}d(x_{0},x)^{2}f_{0}\upharpoonright md\mu\rightarrow\int d(x_{0},x)^{2}f_{0}d\mu\text{ and }\text{\ensuremath{\int_{\hat{X}}d(x_{0},x)^{2}f_{1}\upharpoonright md\mu\rightarrow\int d(x_{0},x)^{2}f_{1}d\mu}}
\]
and hence 
\[
\frac{1}{\mu(f_{0}\upharpoonright m)}\int_{\hat{X}}d(x_{0},x)^{2}f_{0}\upharpoonright md\mu\rightarrow\int d(x_{0},x)^{2}f_{0}d\mu
\]
and
\[
\text{\ensuremath{\frac{1}{\mu(f_{1}\upharpoonright m)}}\ensuremath{\ensuremath{\int_{\hat{X}}}d(\ensuremath{x_{0}},x\ensuremath{)^{2}f_{1}\upharpoonright}md\ensuremath{\mu\rightarrow\int}d(\ensuremath{x_{0}},x\ensuremath{)^{2}f_{1}}d\ensuremath{\mu}}}.
\]

Let $\nu_{0,m}:=\frac{1}{\mu(f_{0}\upharpoonright m)}f_{0}\upharpoonright md\mu$
and $\nu_{1,m}:=\frac{1}{\mu(f_{1}\upharpoonright m)}f_{1}\upharpoonright md\mu$.
By Step 1, we know that 
\[
H(\nu_{\frac{1}{2},m}\mid\mu)\leq\frac{1}{2}H(\nu_{0,m}\mid\mu)+\frac{1}{2}H(\nu_{1,m}\mid\mu)-\frac{1}{8}st(K^{\omega})W_{2}^{2}(\nu_{0,m},\nu_{1,m})
\]
where $\nu_{\frac{1}{2},m}$ is the midpoint between $\nu_{0,m}$
and $\nu_{1,m}$ constructed in Step 1. At the same time, since $H(\nu_{0,m}\mid\mu)$,
$H(\nu_{1,m}\mid\mu)$, $W_{2}^{2}(\nu_{0,m},\nu_{1,m})$ are all
convergent as $m\rightarrow\infty$, it follows, in particular, that
\begin{align*}
\sup_{m}H(\nu_{\frac{1}{2},m}\mid\mu) & \leq\sup_{m}\left[\frac{1}{2}H(\nu_{0,m}\mid\mu)+\frac{1}{2}H(\nu_{1,m}\mid\mu)-\frac{1}{8}st(K^{\omega})W_{2}^{2}(\nu_{0,m},\nu_{1,m})\right]\\
 & <\infty.
\end{align*}
Therefore, it follows, from from the tightness of sublevel sets of
the relative entropy (see for instance \cite[Proposition 4.1]{gigli2015convergence})
together with Prokhorov's theorem, that we can extract a convergent
subsequence of $(\nu_{\frac{1}{2},m})_{m\in\mathbb{N}}$ (which we
do not relabel), whose limit, which we call $\nu_{\frac{1}{2}}$,
is a midpoint of $\nu_{0}$ and $\nu_{1}$. 

Finally, by the lower semicontinuity of $H(\cdot\mid\mu)$, 
\begin{align*}
H(\nu_{\frac{1}{2}}\mid\mu) & \leq\liminf_{m\rightarrow\infty}H(\mu_{\frac{1}{2},m}\mid\mu)\\
 & \leq\liminf_{m\rightarrow\infty}\left[\frac{1}{2}H(\nu_{0,m}\mid\mu)+\frac{1}{2}H(\nu_{1,m}\mid\mu)-\frac{1}{8}st(K^{\omega})W_{2}^{2}(\nu_{0,m},\nu_{1,m})\right]\\
 & =\frac{1}{2}H(\nu_{0}\mid\mu)+\frac{1}{2}H(\nu_{1}\mid\mu)-\frac{1}{8}st(K^{\omega})W_{2}^{2}(\nu_{0},\nu_{1}).
\end{align*}
\end{proof}

\section*{Acknowledgements}

Part of this work was conducted with the support of NSF DMS grant
1814991; additionally, for the part of the work completed while the
author was at IHÉS, the author acknowledges the support of Labex CARMIN.
The author thanks Robert Anderson, David Ross, Dejan Slep\v{c}ev,
and Henry Towsner for helpful discussions. Additionally, the author
thanks Timo Schultz for pointing out an error in a previous version
of this manuscript.

\bibliographystyle{plain}
\bibliography{0_home_andrew_Documents_otrefs}

\end{document}